\documentclass[11pt,a4paper,twoside]{amsart}
\usepackage[top=1.5in, bottom=1.25in, marginparwidth=1in, inner=1.25in, outer=1.25in]{geometry}

\usepackage{lmodern,microtype,array,thm-restate}
\usepackage[T2A,T1]{fontenc} %% T2A for Cyrillic letters
\usepackage[utf8]{inputenc}
\usepackage[UKenglish]{babel}
\makeatletter
\newcommand*{\math@version@bold}{bold}
\DeclareMathOperator\DD{% or \DeclareMathOperator*\DD
	\textrm{%
		\usefont{T2A}{cmr}{\ifx\math@version\math@version@bold bx\else m\fi}{n}%
		\CYRD
	}%
} 
\makeatother
\DeclareMathOperator\sgn{sgn}
\usepackage{xcolor}
\usepackage{caption}
\usepackage[labelfont=up,margin=5pt]{subcaption}
\usepackage{wrapfig}
\usepackage{tikz}
\usepackage{tikz-cd}
\usepackage{pinlabel}

\usepackage{setspace} 
\usepackage[bookmarks=true,pdfencoding=auto,hidelinks,pagebackref]{hyperref}
\usepackage[nameinlink]{cleveref}
\let\cref\Cref
\Crefname{equation}{}{}
\Crefname{subsection}{Subsection}{Subsection}

\usepackage{amsmath}
\usepackage{amsfonts}
\usepackage{amssymb}
\usepackage{amsthm}
\usepackage{nicefrac}
\usepackage{mathtools}
\usepackage{enumerate}
\mathtoolsset{mathic=true}
\usepackage{bm} %% for bold math

\newtheorem{theorem}{Theorem}[section]

\newtheorem{lemma}[theorem]{Lemma}
\newtheorem{conjecture}[theorem]{Conjecture}
\newtheorem{question}[theorem]{Question}

\newtheorem{corollary}[theorem]{Corollary}
\newtheorem{proposition}[theorem]{Proposition}
\theoremstyle{definition}
\newtheorem{defprop}[theorem]{Definition/Proposition}
\newtheorem{definition}[theorem]{Definition}
\newtheorem{example}[theorem]{Example}
\newtheorem{remark}[theorem]{Remark}

\newcommand{\qua}{\hskip 0.4em \ignorespaces}
\def\arxiv#1{\relax\ifhmode\unskip\qua\fi
\href{http://arxiv.org/abs/#1}%
{\tt arXiv:\penalty -100\unskip#1}}

\def\MR#1{\relax\ifhmode\unskip\qua\fi
\href{http://www.ams.org/mathscinet-getitem?mr=#1}{\tt MR#1}}
\def\xox#1{\csname xx#1\endcsname}

\renewcommand*{\backrefalt}[4]{%
\ifcase #1 %
No citations.%
\or
Cited on page~#2.%
\else
Cited on pages~#2.%
\fi
}

\newcommand{\myqed}{\pushQED{\qed}\qedhere}

\title{Rasmussen invariants of Whitehead doubles and other satellites}

\newcommand{\myemail}[1]{\href{mailto:#1}{#1}}
\author{Lukas Lewark}
\address{ETH Z\"urich, R\"amistrasse 101, 8092 Z\"urich, Switzerland}
\email{\myemail{llewark@math.ethz.ch}}
\urladdr{\url{https://people.math.ethz.ch/~llewark/}}

\author{Claudius Zibrowius}
\address{Fakultät für Mathematik, Ruhr-Universität Bochum, Universitätsstraße 150, 44780 Bochum, Germany}
\email{\myemail{claudius.zibrowius@rub.de}}
\urladdr{\url{https://cbz20.raspberryip.com/}}

\hypersetup{pdfauthor={\authors},pdftitle={\shorttitle}}

\newcommand{\vc}[1]{\vcenter{\hbox{#1}}}%
\newcommand{\mypic}[2]{%
  \newcommand{#2}{%
    \vc{%
      \includegraphics[page=#1]%
      {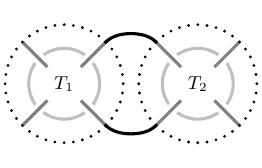}%
    }%
  }%
}%

\mypic{1}{\tanglesum}
\mypic{2}{\union}
\mypic{3}{\rationalclosure}
\mypic{4}{\Di}
\mypic{5}{\Do}
\mypic{6}{\Li}
\mypic{7}{\Lo}
\mypic{8}{\Ni}
\mypic{9}{\No}
\mypic{10}{\ConnectivityX}
\mypic{11}{\CrossingL}
\mypic{12}{\CrossingR}
\mypic{13}{\CrossingY}
\mypic{14}{\CrossingX}
\mypic{15}{\Crossing}
\mypic{16}{\ratixo}
\mypic{17}{\ratoxi}
\mypic{18}{\ratixi}
\mypic{19}{\ratiixi}
\mypic{20}{\ratmiiixi}
\mypic{21}{\muty}
\mypic{22}{\QuiverOnSurfaceDualWithQuiver}
\newcommand{\BNr}{\widetilde{\operatorname{BN}}}% reduced Bar-Natan homology
\newcommand{\CBNr}{\widetilde{\operatorname{CBN}}}% reduced Bar-Natan chain complex
\newcommand{\Khr}{\widetilde{\operatorname{Kh}}}% reduced Khovanov homology
\newcommand{\QPI}{\operatorname{\mathbb{Q}P^1}}% rational projective line =  space of slopes
\newcommand{\id}{\operatorname{id}}% identity map
\newcommand{\Cob}{\operatorname{Cob}}% identity map
\newcommand{\Mor}{\operatorname{Mor}}% morphism space
\newcommand{\Mod}{\operatorname{Mod}}% mapping class group
\newcommand{\homology}{\operatorname{H_\ast}}% homology
\newcommand{\HF}{\operatorname{HF}}% Lagrangian Floer homology
\newcommand{\HFhat}{\widehat{\operatorname{HF}}}% Lagrangian Floer homology
\newcommand{\CF}{\operatorname{CF}}% Lagrangian Floer chain complex
\newcommand{\CFtimes}{\operatorname{CF}^\times}% Lagrangian Floer chain complex
\newcommand{\CFplus}{\operatorname{CF}^+}% Lagrangian Floer chain complex
\newcommand{\BNAlgH}{\mathcal{B}}% Bar-Natan algebra B
\newcommand{\KhTl}[1]{[\![ #1 ]\!]} % [[T]]_{/l}
\newcommand{\Diag}{\mathcal{D}} % tangle diagram
\newcommand{\FourPuncturedSphereKh}{S^2_{4,\ast}}
\newcommand{\PuncturedPlane}{\mathbb{R}^2\smallsetminus\mathbb{Z}^2}
\newcommand{\field}{\mathbf{k}}
\newcommand{\F}{\mathbb{F}}
\newcommand{\Z}{\mathbb{Z}}
\newcommand{\Q}{\mathbb{Q}}

\newcommand{\QGrad}[1]{{\textcolor{violet}{#1}}}

\newcommand{\HomGrad}[1]{{\textcolor{black}{#1}}}

\newcommand{\GGzqh}[4]{\prescript{\QGrad{#3}}{}{#1}_\HomGrad{#4}}

\newcommand{\arc}[1]{\textbf{a}_{#1}}
\newcommand{\arcb}[1]{\textcolor{blue}{\textbf{a}_{#1}}}
\newcommand{\arcr}[1]{\textcolor{red}{\textbf{a}_{#1}}}
\newcommand{\aR}{\textcolor{red}{a}}
\newcommand{\bB}{\textcolor{blue}{b}}
\newcommand{\g}{\textcolor{red}{\gamma}}
\newcommand{\gb}{\textcolor{blue}{\gamma}}
\newcommand{\gi}{\textcolor{red}{\gamma_1}}
\newcommand{\gii}{\textcolor{red}{\gamma_2}}
\newcommand{\gp}{\textcolor{blue}{\gamma'}}
\newcommand{\gpi}{\textcolor{blue}{\gamma'_1}}
\newcommand{\gpii}{\textcolor{blue}{\gamma'_2}}

\newcommand{\lk}{\operatorname{lk}} %linking number
\newcommand{\conn}[1]{\operatorname{x}(#1)}

\renewcommand{\theta}{\ensuremath{\vartheta}}
\newcommand{\Wh}{W^+}
\newcommand{\Whn}{W^-}
\newcommand{\Whmp}{W^\mp}
\newcommand{\Whpm}{W^\pm}

\def\co{\colon\thinspace\relax}% colon spaces

\newcommand{\thetap}{\theta'}
\begin{document}

\begin{abstract}
We prove formulae for the $\mathbb{F}_2$-Rasmussen invariant of satellite knots of patterns with wrapping number 2,
using the multicurve technology for Khovanov and Bar-Natan homology developed by Kotelskiy, Watson, and the second author.
A new concordance homomorphism, which is independent of the Rasmussen invariant, plays a central role in these formulae.
We also explore whether similar formulae hold for the Ozsváth-Szabó invariant $\tau$.

\end{abstract}

\maketitle

% conventions:
% "Wrapping number" instead of "geometric winding number"
% "Winding number" instead of "algebraic winding number"

\section{Main results}\label{sec:introintro}

In this article, we study the values that Rasmussen invariants take on
satellite knots with wrapping number two, such as Whitehead doubles and 2-cables.
We begin by stating our main theorem for winding number zero satellites.
\begin{restatable}{theorem}{maintm} \label{thm:main}
	For every knot \(K\subset S^3\), there exists a unique integer \(\theta_2(K)\) such that
	\[
	s_2(P(K))
	=
	s_2(P_{-\theta_2(K)}(U))
	\]
	for all patterns \(P\) with wrapping number two and winding number zero.
\end{restatable}
\begin{figure}[b]
	\vspace*{-4pt}
	\centering
	\begin{subfigure}{0.45\textwidth}
		\centering
		\includegraphics[scale=.9]{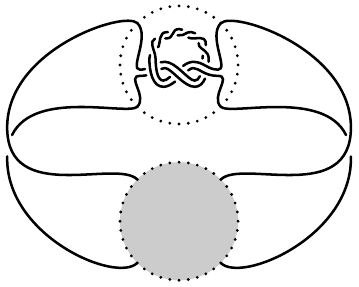}
		\caption{}\label{fig:frontispiece:1}
	\end{subfigure}
	\begin{subfigure}{0.45\textwidth}
		\centering
		\includegraphics[scale=.9]{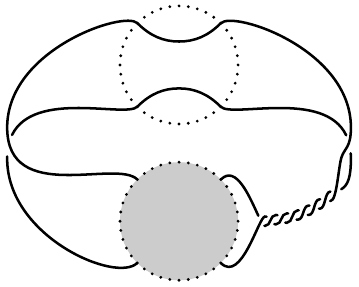}
		\caption{}\label{fig:frontispiece:2}
	\end{subfigure}
	\\
	\begin{subfigure}{0.17\textwidth}
		\centering
		\includegraphics[scale=.9]{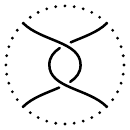}
		\caption{\(Q_{\nicefrac{-1}{2}}\)}\label{fig:Q_m1_2}
	\end{subfigure}
	\begin{subfigure}{0.17\textwidth}
		\centering
		\includegraphics[scale=.9]{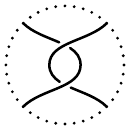}
		\caption{\(Q_{\nicefrac{1}{2}}\)}\label{fig:Q_1_2}
	\end{subfigure}
	\begin{subfigure}{0.17\textwidth}
		\centering
		\includegraphics[scale=.9]{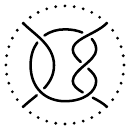}
		\caption{\(T_a\)}\label{fig:T_a}
	\end{subfigure}
	\begin{subfigure}{0.17\textwidth}
		\centering
		\includegraphics[scale=.9]{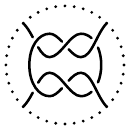}
		\caption{\(T_b\)}\label{fig:T_b}
	\end{subfigure}
	\begin{subfigure}{0.17\textwidth}
		\centering
		\includegraphics[scale=.9]{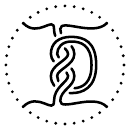}
		\caption{\(T_c\)}\label{fig:T_c}
	\end{subfigure}
	\caption{An application of \cref{thm:main}:
		Plugging a suitable pattern tangle \(T_P\) into the grey discs in (a) and (b)
		results in the two knots \(P(T_{2,3})\) and \(P_{-4}(U)\), respectively, where \(T_{2,3}\) denotes the right-handed trefoil.
		Some possibilities for \(T_P\) are shown in (c)--(g). 
		For a fixed \(P\),
		those two knots have the same Rasmussen invariant~\(s_2\) according to \cref{thm:main} since \(\theta_2(T_{2,3})=4\).}
	\label{fig:mainthm}
\end{figure}
\begin{theorem}\label{thm:main2}
	The knot invariant $\theta_2$ induces a homomorphism from the smooth knot concordance group to the integers, which is linearly independent of the one given by $s_2$.
\end{theorem}
Let us explain our notation.
For \(p\) a prime, \(s_p\) denotes the Rasmussen invariant over the field~\(\mathbb{F}_p\) \cite{mvt,LipshitzSarkarRefinementS}, while \(s_0\) denotes the original Rasmussen invariant over \(\Q\) \cite{RasmussenSlice}.
By a \emph{pattern} $P$, we mean a knot in the standard solid torus.
We write \(P(K)\) for the satellite knot with pattern \(P\) and companion knot~\(K\).
For \(t\in\Z\), we denote by \(P_t\) the pattern obtained from~\(P\) by applying \(t\) right-handed full twists to the standard torus.
We denote the unknot by \(U\).

For many classical knot invariants, such as the signature or the Alexander polynomial, one can express the invariant of a satellite as a function of the pattern and the invariant of the companion.
%For the \(\tau\) invariant, \(\tau(P(K))\) is not determined by \((P, \tau(K))\) for general patterns \(P\), but it is for $P = D^{\pm}_t$, at least.
The Rasmussen invariant \(s_2\), on the other hand, displays a more complicated behaviour: % even for twisted Whiteheads doubles:
%\begin{proposition}
%        For any two given integers \(x\in 2\Z\) and \(y\in 4\Z\),
%        there exists a knot \(K\) such that \(s_2(K) = x\) and \(\theta_2(K) = y\).
%\end{proposition}
%In particular,
Indeed, the independence of \(s_2\) and \(\theta_2\) claimed in \cref{thm:main2} implies that the pair \((P,s_2(K))\) does not determine \(s_2(P(K))\).
%A simple explicit example for that, which can be computed from the table of values given at the end of \cref{subsec:multicurves},
%is given by contrasting \(s_2(\Wh(U)) = s_2(U) = 0\) with
%\begin{equation}
%\label{eq:nondetex} s_2(\Wh(K)) = 2, \qquad s_2(K) = 0 \qquad \text{for } K = T(-3,4) \# T(2,3)^{\# 3}.
%\end{equation}
%\noindent
However, \cref{thm:main} tells us that for all winding number zero, wrapping number two patterns~\(P\),
the value \(s_2(P(K))\) is instead determined by the pair \((P, \theta_2(K))\).
This is the raison d'être of the knot invariant $\theta_2$. %Let us now examine its properties.

\subsection*{The proof of the theorems in a nutshell}\label{subsec:multicurves}
Observe that one may cut the satellite knot \(P(K)\) along a sphere into two tangles, as can be seen in \cref{fig:mainthm}(a): one \emph{pattern tangle}~\(T_P\) associated with~\(P\) and one \emph{companion tangle}~\(T_K\) that is determined by~\(K\).
Now, we use Bar-Natan's extension of Khovanov homology to tangles \cite{BarNatanKhT}.
Bar-Natan associates with a tangle \(T\) a chain complex~\(\KhTl{T}\)
over a certain additive category \(\Cob_{/l}\), which depends on the number of endpoints of \(T\).
The complexes \(\KhTl{T}\) are well-defined up to homotopy equivalence and well-behaved under gluing.
Thus from \(\KhTl{T_P}\) and \(\KhTl{T_K}\), one may calculate the Khovanov homology of \(P(K)\),
and from that homology (precisely: the reduced Khovanov homology over \(\mathbb{F}_2[X]/(X^2 + X)\)) one may read off \(s_2(P(K))\).

Since \(P\) has wrapping number two, \(T_P\) and \(T_K\) are Conway tangles, i.e.~tangles with four endpoints.
Chain complexes over \(\Cob_{/l}\) for four endpoints are well-understood
by the work of Kotelskiy, Watson, and the second author \cite[Theorems~1.1 and~1.5]{KWZ}.
Indeed, such chain complexes can be encoded geometrically
as certain sets of immersed curves (so-called \emph{multicurves}) on the four-punctured sphere.
The multicurve of \(\KhTl{T}\) consists
of a finite number of immersed circles and exactly one immersed interval, which is called the \emph{non-compact component} and is denoted by $\BNr_a(T)$.
It turns out that the Rasmussen invariant \(s_2\) of \(P(K) = T_P \cup T_K\) is fully determined
by \(\BNr_a(T_P)\) and \(\BNr_a(T_K)\).
The crucial step, which relies on the multicurve technology, is to show that \(\BNr_a(T_K)\)
is, in a suitable sense, approximated by \(\BNr_a(Q_{2\theta_2(K)})\) for some \(\theta_2(K)\in\Z\),
where \(Q_{2\theta_2(K)}\) is the (rational) Conway tangle consisting of \(2\theta_2(K)\) left-handed half-twists. 
It follows that
\[
s_2(P(K)) = s_2(T_P \cup T_K) = s_2(T_P \cup Q_{2\theta_2(K)}) = s_2(P_{-\theta_2(K)}(U)),
\]
as claimed in \cref{thm:main}. A detailed proof will be given in \cref{sec:RasmussenCurves}.

\begin{wraptable}[5]{r}{0.2\textwidth}
	\centering
	\begin{tabular}{c|cc}
		% no arguments for column headers
		% first column header 'knot'
		knot        & \(s_2\) & \(\theta_2\) \\\hline
		\(T_{2,3}\)  & 2        & 4\rule{0pt}{2.3ex}\\
		\(T_{3,4}\)  & 6        & 8
	\end{tabular}
	\medskip
	\caption{}\label{table:t23t34}
\end{wraptable}

The proof of \cref{thm:main2} consists of two parts.
Firstly, by applying \cref{thm:main} with varying patterns, one can deduce that \(\theta_2\) is a concordance homomorphism, as we shall see %in the proof of \cref{prop:1implies2} 
in \cref{prop:1implies2}.
% This proof will not require any properties of $s_2$ beyond the fact that it satisfies \cref{thm:main}.
Secondly, the linear independence of $s_2$ and $\theta_2$ can be seen from their values
on the torus knots $T_{2,3}$ and $T_{3,4}$, see \cref{table:t23t34}.
This and further independence results will be proven in \cref{thm:indep} in \cref{sec:differentfields}.

\subsection*{Acknowledgements}
The second author would like to thank Matt Hedden for pointing him to \cite{Hedden,HeddenOrding}.
The authors are supported by the Emmy Noether Programme of the DFG, Project number 412851057.
The second author was also partially supported by the SFB 1085 Higher Invariants in Regensburg.
The authors thank the University of Regensburg for access to the compute cluster Athene to run their calculations.
\section{Complementary results and speculations}\label{sec:introext}

In this extended introduction, we put \cref{thm:main,thm:main2} into the context of other slice-torus invariants. 
We also explore formulae for Rasmussen invariants of satellite knots with winding number \(\pm2\) patterns and discuss some applications of our results.
The reader eager to peruse the proofs of \cref{thm:main,thm:main2} may proceed directly to \cref{subsec:structure_of_paper}, where a summary of the paper's structure will help to decide where to continue.
% ChatGPT schlägt vor: "The reader in a hurry can jump straight to Section B for a quick rundown of the paper's layout and some nifty hints on where to jump next."

\subsection{Pattern-twisting behaviour of slice-torus invariants}\label{subsec:generalize1}
It is remarkable that while the methods that go into proving \cref{thm:main} are very particular to Khovanov homology, 
the proof that $\theta_2$ is a concordance homomorphism, as claimed in \cref{thm:main2}, does not require any properties of $s_2$ beyond the fact that it satisfies \cref{thm:main};
any such invariant gives rise to a potentially new concordance homomorphism. 
Thus, one is led to wonder:
\begin{center}\itshape
	Do variations of \cref{thm:main} hold for other knot invariants?
\end{center}

We will discuss this question later for various different knot invariants in \cref{subsec:geom:other}. 
For now, we restrict to \emph{slice-torus invariants} (\cref{def:slice-torus-invariant}). 
This family of knot invariants was first studied by Livingston \cite{LivingstonComputations} and contains
%Principal examples of slice-torus invariants are 
the Rasmussen invariants, normalised as \(\nicefrac{s_c}{2}\) for any characteristic \(c\), as well as the Ozsváth-Szabó invariant \(\tau\) coming from Heegaard Floer theory.  
%The first family of knot invariants that comes to mind are \emph{slice-torus invariants}, which were first studied by Livingston and Naik \cite{LivingstonComputations}.
%These are knot invariants~\(\nu\) that induce homomorphisms from the smooth knot concordance group \(\mathcal{C}\) to \(\Z\),
%such that \(\nu(K)\) is less than or equal to the smooth four-genus \(g_4(K)\),
%and this bound is sharp for positive torus knots,
%i.e.\ \(\nu(T_{p,q}) = g_4(T_{p,q}) = (p-1)(q-1)/2\) for positive coprime integers \(p, q\) \cite{LivingstonComputations}; see also \cite{lew2}.
%Principal examples of slice-torus invariants are the Rasmussen invariants normalised as \(\nicefrac{s_c}{2}\) for any coefficients \(c\) as well as the Ozsváth-Szabó invariant \(\tau\) coming from Heegaard Floer theory.  
In the following, let \(\nu\) denote a slice-torus invariant, \(K\) a knot, and \(P\) a pattern of wrapping number two and winding number zero. 
The following result describes how \(\nu(P(K))\) is affected by twisting $P$.

\begin{proposition}\label{prop:theta:winding0}
	For any fixed triple \((\nu,P,K)\), the function \(\Z\rightarrow\Z\) given by \(t\mapsto \nu(P_t(K))\) is either constant or there exists a unique jump point \(\theta\in\Z\) such that for all \(t\in\Z\)
	\begin{equation}\label{eq:theta:winding0}
	\nu(P_{t}(K))
	=
	\nu(P_{t-1}(K))
	-
	\begin{cases*}
	1 
	&
	if \(t=\theta\)
	\\
	0
	&
	otherwise.
	\end{cases*}
	\end{equation}
\end{proposition}

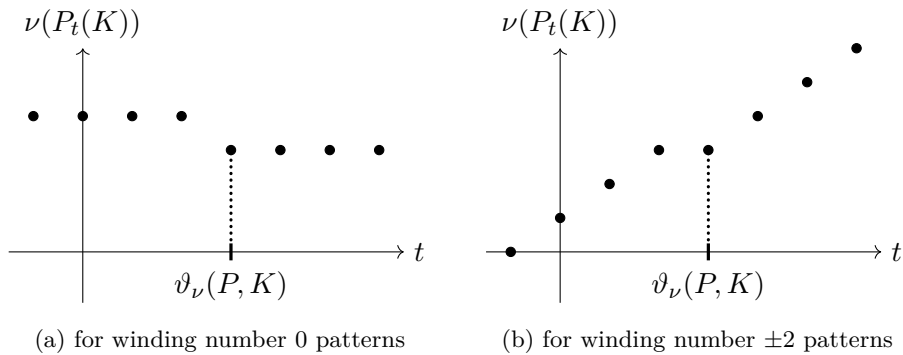
\begin{figure}[b]
	\begin{subfigure}{0.42\textwidth}\centering
		\(
		\begin{tikzpicture}[x=0.65cm, y=0.45cm]
		\draw[->] (-1.5,0) -- (6.5,0) node[right]{\(t\)};
		\draw[->] (0,-1.5) -- (0,6) node[above]{\(\nu(P_t(K))\)};
		%		\foreach \i in {-1,...,6}{\draw[gray] (\i,-0.25) -- (\i,0.25);}
		%		\foreach \i in {-1,...,5}{\draw[gray] (-0.25,\i) -- (0.25,\i);}
		\draw [fill] (-1,4) circle (1.75pt);
		\draw [fill] (0,4)  circle (1.75pt);
		\draw [fill] (1,4)  circle (1.75pt);
		\draw [fill] (2,4)  circle (1.75pt);
		\draw [fill] (3,3)  circle (1.75pt);
		\draw [fill] (4,3)  circle (1.75pt);
		\draw [fill] (5,3)  circle (1.75pt);
		\draw [fill] (6,3)  circle (1.75pt);
		\draw [very thick,line cap=round,dash pattern=on 0pt off 2\pgflinewidth] (3,2.9) -- (3,0.25);
		\draw [very thick] (3,0.25) -- (3,-0.25) node [below] {$\theta_\nu(P,K)$};
	\end{tikzpicture}
	\)
	\caption{for winding number 0 patterns}
	\label{fig:def:theta:winding0}
	\end{subfigure}
	\begin{subfigure}{0.42\textwidth}\centering
		\(
		\begin{tikzpicture}[x=0.65cm, y=0.45cm]
		\draw[->] (-1.5,0) -- (6.5,0) node[right]{\(t\)};
		\draw[->] (0,-1.5) -- (0,6) node[above]{\(\nu(P_t(K))\)};
		%		\foreach \i in {-1,...,6}{\draw[gray] (\i,-0.25) -- (\i,0.25);}
		%		\foreach \i in {-1,...,5}{\draw[gray] (-0.25,\i) -- (0.25,\i);}
		\draw [fill] (-1,0) circle (1.75pt);
		\draw [fill] (0,1)  circle (1.75pt);
		\draw [fill] (1,2)  circle (1.75pt);
		\draw [fill] (2,3)  circle (1.75pt);
		\draw [fill] (3,3)  circle (1.75pt);
		\draw [fill] (4,4)  circle (1.75pt);
		\draw [fill] (5,5)  circle (1.75pt);
		\draw [fill] (6,6)  circle (1.75pt);
		\draw [very thick,line cap=round,dash pattern=on 0pt off 2\pgflinewidth] (3,2.9) -- (3,0.25);
		\draw [very thick] (3,0.25) -- (3,-0.25) node [below] {$\theta_\nu(P,K)$};
	\end{tikzpicture}
	\)
	\caption{for winding number \(\pm2\) patterns}
	\label{fig:def:theta:winding2}
	\end{subfigure}
\caption{The behaviour of the functions \(t\mapsto \nu(P_t(K))\) for wrapping number 2 patterns in the case \(\theta_\nu(P,K)\neq\infty\)}
\label{fig:def:theta}
\end{figure}

This explains the role played in \cref{thm:main} by the value \(\theta_2(K)\), which we can now generalise as follows:

\begin{definition}\label{def:theta}
	For any fixed triple \((\nu,P,K)\), we define \(\theta_{\nu}(P,K)\in\Z\cup\{\infty\}\) 
	as the unique integer \(\theta\) satisfying \eqref{eq:theta:winding0} if such an integer exists and \(\infty\) otherwise; see \cref{fig:def:theta:winding0} for an illustration. 
	Moreover, we set 
	\[
	\theta_{\nu}(P)
	\coloneqq
	\theta_{\nu}(P,U)
	\quad
	\text{and}
	\quad
	\theta_{\nu}(K)
	\coloneqq
	\theta_{\nu}(\Wh,K),
	\]
	where \(\Wh\) is the positive Whitehead pattern, i.e.\ the pattern corresponding to the tangle in \cref{fig:Q_m1_2}.
	For \(c=0\) or a prime, we write
	\(
	\theta_c(K)
	\coloneqq 
	\theta_{\nicefrac{s_c}{2}}(K)
	\)
	and 
	\(
	\theta_c(P)
	\coloneqq
	\theta_{\nicefrac{s_c}{2}}(P)
	\). 
\end{definition}

%\begin{observation}
%	For any \(t\in\Z\), \(\theta_\nu(P_t,K)=\theta_\nu(P,K)+t\).
%\end{observation}

Livingston and Naik \cite{doubled} proved \cref{prop:theta:winding0} for $P = W^+$
and defined an invariant $t_{\nu}(K)$, which is related to \(\theta_{\nu}(K)\) by a mere translation:  $t_{\nu}(K) = \theta_{\nu}(K) - 1$.
%Our reason for the shift by $1$ is given by \cref{thm:main2}.
Our reason for introducing this shift is given by \cref{thm:main2}: 
For \(\nu=\nicefrac{s_2}{2}\), 
\(\theta_{\nu}(K)=\theta_2(K)\) induces a concordance homomorphism, whereas $t_{\nu}(K)$ does not. 

The invariant \(\theta_{\nicefrac{s_2}{2}}(K)\) does indeed agree with the integer \(\theta_2(K)\) from \cref{thm:main}. 
Before justifying this statement, we first make an observation and a computation.

\begin{proposition}\label{prop:theta_nu}
	For any slice-torus invariant \(\nu\) and knot \(K\), 
	\(
	\theta_{\nu}(K)
	\)
	is finite.
\end{proposition}

\begin{proof}%[Proof of \cref{defprop:theta_nu}]
	This is a reformulation of the first part of \cite[Theorem~2]{doubled}, where our invariant \(\theta_\nu(K) = \theta_\nu(\Wh,K)\) corresponds to \(t_{\nu}(K) + 1\) in their notation.
\end{proof}

\begin{example}\label{ex:nu_of_twist_knots}
	Observe that \(\Wh_0(U)=U\) and \(\Wh_{-1}(U)=T_{2,3}\). 
	So for all slice-torus invariants~$\nu$, \(\nu(\Wh_0(U))=0\) and \(\nu(\Wh_{-1}(U))=1\). 
	By \cref{prop:theta:winding0}, this implies
	\[
	\nu(\Wh_{t}(U))
	=
	\begin{cases*}
	1
	&
	if \(t<0\)
	\\
	0
	&
	otherwise
	\end{cases*}
	\]
	and hence \(\theta_{\nu}(U)=\theta_{\nu}(\Wh,U)=\theta_{\nu}(\Wh)=0\). 
	Similarly, one can determine \(\theta_{\nu}(\Whn)=1\), where \(\Whn\) is the negative Whitehead pattern (corresponding to the tangle in \cref{fig:Q_1_2}).
\end{example}

Now, suppose \(\nu\) is a slice-torus invariant and \(K\) a knot such that
\[
\nu(\Wh_t(K))
=
\nu(\Wh_{t-\theta(K)}(U))
\]
for some integer \(\theta(K)\) independent of \(t\).  
Then it follows that \(\theta(K)=\theta_\nu(K)\).
In particular, the integer \(\theta_2(K)\) from \cref{thm:main} agrees with \(\theta_{\nicefrac{s_2}{2}}(K)\). 
The preceding discussion also allows us to make the question from the beginning of this subsection more precise.

\begin{question}\label{que:generalize}
	Let $\nu$ be a slice-torus invariant.
	Does $\nu(P(K)) = \nu(P_{-\theta_\nu(K)}(U))$ hold
	for all knots $K$ and all patterns $P$ with wrapping number two and winding number zero?
\end{question}

The following is a variation of \cref{thm:main}, which allows us to compute \(s_2(P_t(K))\) for any \(t\in\Z\) from the three invariants \(s_2(P(U))\), \(\theta_2(K)\), and \(\theta_2(P)\). 

\begin{proposition}\label{thm:twisting}
	Suppose \(\nu\) is a slice-torus invariant for which the answer to \cref{que:generalize} is `yes'. 
	Then for all knots \(K\), integers \(t\), and patterns \(P\) of wrapping number~2 and winding number~0,
	the following holds:
	\[
	\nu(P_t(K)) = \nu(P(U)) +
	%\left\{
	%\begin{array}{ll}
	%-2 & \text{if } \theta_c(P) > 0 \\
	%0  & \text{else}
	%\end{array}
	%\right\}
	%+
	%\left\{
	%\begin{array}{ll}
	%2 & \text{if } \theta_c(P) + \theta_c(K) > t \\
	%0  & \text{else}
	%\end{array}
	%\right\}.
	\left\{
	\begin{array}{ll}
	1  & \text{if } \theta_\nu(P) \neq\infty \text{ and } t - \theta_\nu(K) < \theta_\nu(P) \leq 0, \\
	-1 & \text{if } \theta_\nu(P) \neq\infty \text{ and } t - \theta_\nu(K) \geq \theta_\nu(P) > 0, \\
	0  & \text{otherwise.}
	\end{array}
	\right.
	\]
	%where \(\sgn_+(n)\) is defined as \(1\) for \(n>0\) and \(0\) for \(n\leq 0\).
\end{proposition}
\begin{proof}
	By hypothesis, 
	\(\nu(P_t(K)) = \nu(P_{t-\theta_\nu(K)}(U))\). 
	The identity now follows from \cref{prop:theta:winding0}, observing that \(\nu(P_{t-\theta_\nu(K)}(U))\) is equal to \(\nu(P(U))\) unless the jump point 
	\(\theta_\nu(P)\) is finite and lies between \(t - \theta_\nu(K)\) and 0.
\end{proof}

In \cref{sec:RasmussenCurves,sec:differentfields}, we will see two different methods of computing the invariants \(\theta_c(K)\) and \(\theta_c(P)\). 
Firstly, they can be computed directly from definition, using the computability of the Rasmussen invariants; see \cref{prop:thetaalgo}. 
Secondly, they can also be easily read off the multicurves associated with pattern and companion tangles.
By applying \cref{prop:eta-slopes} to the curves in \cref{fig:lifts}, we compute the following values for the patterns \(P^T\) associated with the tangles \(T\) from \cref{fig:mainthm}:
\[
\theta_2(\Wh)=\theta_2(P^{T_a})=\theta_2(P^{T_c})=0,
\quad
\theta_2(\Whn)=1,
\quad
\text{and}
\quad
\theta_2(P^{T_b})=\infty.
\]
Similarly, using \cref{cor:theta-from-slopes}, we can easily determine the following values:
\[
\theta_2(3_1)=4,
\quad
\theta_2(4_1)=0,
\quad
\text{and}
\quad
\theta_2(8_{19})=8.
\]
Here \(3_{1} = T_{2,3}\) and \(8_{19} = T_{3,4}\) denote the positive torus knots. 
We have computed the invariants \(\theta_c(K)\) for a range of small knots \(K\). 
The results are available online \cite{thetatable}.
For more details about computations, see \cref{sec:RasmussenCurves,sec:differentfields}. 

Thanks to \cref{thm:main}, once we know \(\theta_2(K)\) and \(\theta_2(P)\), we also know \(\theta_2(P,K)\):

\begin{proposition}\label{thm:winding0:thetas}
	Suppose \(\nu\) is a slice-torus invariant for which the answer to \cref{que:generalize} is `yes'. 
	Then for all knots \(K\) and patterns \(P\) of wrapping number two and winding number zero,
	\[
	\theta_\nu(P,K)=\theta_\nu(P)+\theta_\nu(K).
	\]
\end{proposition}

\begin{proof}%[Proof of \cref{thm:winding0:thetas}]
	By hypothesis, the functions 
	\(t\mapsto \nu(P(K))\) 
	and 
	\(t\mapsto \nu(P(U))\) 
	are either both constant (in which the asserted identity becomes \(\infty=\infty+\text{integer}\)) or the difference \(\theta_\nu(P,U)-\theta_\nu(P,K)\) between their jump points is indeed precisely \(-\theta_\nu(K)\). 
\end{proof}

Livingston-Naik showed that for all knots \(K\) and \(t\in \Z\), the values \(\nu(\Wh_t(K))\) and \(\nu(\Whn_t(K))\) cannot both be non-zero \cite[Theorem~1]{doubled}.  
We can push this result a little further:

\begin{corollary}
	Suppose \(\nu\) is a slice-torus invariant for which the answer to \cref{que:generalize} is `yes'. 
	Then for each knot \(K\), there is a unique integer \(t\) with 
	$${\nu(\Wh_t(K)) = \nu(\Whn_t(K)) = 0}.$$
	Indeed, we have
	\[
	(\nu(\Wh_t(K)), \nu(\Whn_t(K))) = 
	\begin{cases}
	(0, -1) & t > \theta_\nu(K), \\
	(0, 0)  & t = \theta_\nu(K), \\
	(1, 0)  & t < \theta_\nu(K).
	\end{cases}
	\]
\end{corollary}
\begin{proof}
	By \cref{ex:nu_of_twist_knots}, \(\theta_\nu(\Wh)=0\) and \(\theta_\nu(\Whn)=1\) for all slice-torus invariants \(\nu\). 
	Also note that \(\Whpm(U)\) is the unknot, so \(\nu(\Whpm(U))=0\). 
	Now apply \cref{thm:twisting} to \(P=\Whpm\).
\end{proof}

%A priori, the dependence of \(\theta_\nu(P,K)\) on the companion \(K\) and the pattern \(P\) can be complicated. 
%However, if as for \(\nu=\nicefrac{s_2}{2}\), the slice-torus invariant \(\nu\) satisfies 
%\begin{equation}\label{eq:nutheta}
%\nu(P(K)) = \nu(P_{-\theta_\nu(K)}(U))
%\end{equation}
%for all knots \(K\) and all patterns \(P\) of wrapping number two and winding number zero, then
%\[
%\theta_\nu(P,K)=\theta_\nu(P)+\theta_\nu(K),
%\]
\subsection{\texorpdfstring{Comparison with the behaviour of \(\bm{\tau}\)}{Comparison with the behaviour of τ}}\label{sec:intro:tau:winding0}

This paper was inspired by results of Hedden, who computed the Ozsváth-Szabó invariant \(\tau\) for twisted Whitehead doubles~\cite{Hedden}. 
From the perspective of the previous subsection, Hedden's results say that for any knot~\(K\), 
\begin{equation}\label{eq:hedden}
\theta_{\tau}(K) = 2\tau(K).
\end{equation}
Interestingly, twisted Whitehead doubles were the first known examples of knots \(K\) with \(\tau(K) \neq \tfrac{s_0}2(K)\). 
In fact, it was Hedden and Ording \cite{HeddenOrding} who found the first such example by showing that
\begin{equation}\label{eq:heddenording}
\tau(\Wh_2(T_{2,3})) = 0\ \neq\ 1= \tfrac{s_0}2(\Wh_2(T_{2,3})).
\end{equation}
This inequality can now be seen as a consequence of the fact that \(\theta_\tau(T_{2,3}) = 2\tau(T_{2,3}) = 2\), whereas \(\theta_0(T_{2,3}) = 3\); see \eqref{eq:theta_for_trefoil} in \cref{sec:differentfields}. 
In view of identity~\eqref{eq:hedden}, even if \cref{que:generalize} has a positive answer for \(\nu=\tau\), the concordance homomorphism $\vartheta_\tau$ thus obtained is not new. 
However, identity~\eqref{eq:hedden} allows us to reformulate \cref{que:generalize} as follows:

\begin{question}\label{q:tau}
Does the Ozsváth-Szabó invariant \(\tau\) satisfy
\[
\tau(P(K)) = \tau(P_{-2\tau(K)}(U))
\]
for all knots \(K\subset S^3\) and patterns \(P\) with wrapping number 2 and winding number 0?
\end{question}

In \cref{subsec:geom:tau}, we discuss a result of Levine~\cite{MR2971610} which gives a positive answer to \cref{q:tau} for certain patterns $P$.

A common way to show that two slice-torus invariants \(\nu\), \(\nu'\) are not equal is to evaluate them
on a well-chosen twisted Whitehead double \(W\): see e.g.~\cref{eq:heddenording} above for \(s_0 \neq \tau\)
or \cite{postcard} for \(s_0 \neq s_3\). However, in this way one can only ever find knots \(W\) with
$|\nu(W) - \nu'(W)| = 1$; see \cref{ex:nu_of_twist_knots}. 
Does there exist a knot $K$ with three-genus 1 such that $|\nu(K) - \nu'(K)| = 2$?
We prove the following in \cref{subsec:geom:tau}.

\begin{restatable}{proposition}{genusone} \label{prop:genus-1-knots}
	There exists a knot \(K\) with three-genus \(1\), Alexander polynomial~\(1\), \(\tau(K) = 1\) and \(s_2(K) = -2\).
	More generally, for any prescribed triple \((g, a, b)\) of integers with ${|a| \leq g}$, $|b| \leq g$,
	there exists a knot \(L\) with three-genus $g$, Alexander polynomial~\(1\), \(\tau(L) = a\) and \(s_2(L) = 2b\).
\end{restatable}

\subsection{Geometric applications}
Let us showcase two corollaries to \cref{thm:main}.
The first answers an open question of Hedden and Pinzón-Caicedo,
to ``find any knot for which the Whitehead double of both \(K\) and \(-K\) are non-zero in concordance''~\cite{hp}.

\begin{corollary}\label{thm:hpc}
There exists a knot \(K\) such that
\(\Wh(K)\) and \(\Wh(-K)\) generate a summand of the smooth concordance group \(\mathcal{C}\) that is isomorphic to \(\Z^2\).
\end{corollary}
%\begin{proof}

%\vspace{-2ex}
\begin{wraptable}[6]{r}{0.25\textwidth}
	\vspace{1ex}
	\centering
	\begin{tabular}{r|cc}
		knot       & \(\tau\) & \(s_2\) \\\hline
		\(\Wh(K)\)  & 1        & 0 \rule{0pt}{2.3ex}\\
		\(\Wh(-K)\) & 0        & 2
	\end{tabular}
	\medskip
	\caption{}\label{table:wkminusk}
\end{wraptable}
\noindent\emph{Proof.}
Let \(K = T_{3,4}^{\# 2}\# T_{-2,3}^{\#5}\).
Using the additivity of $\tau$ and $\theta_2$ with respect to the connected sum operation~$\#$
and \cref{table:t23t34},
one computes \(\tau(K) = 1 \) and \(\theta_2(K) = -4\), and thus finds the values shown
in \cref{table:wkminusk}.
So the split epimorphism \((\tau,\nicefrac{s_2}2)\colon \mathcal{C}\to\Z^2\) admits the right-sided inverse
\((\Wh(K), \Wh(-K))\), proving the existence of the desired summand.
\qed\bigskip
%\end{proof}

The second corollary concerns a conjecture of Hedden \cite{hp}.
Let \(P\) be a pattern with wrapping number two and winding number zero.
%\comment{Claudius: Shouldn't we also assume that the wrapping number is 2? We apply \cref{prop:theta:winding0}.}
The conjecture then says that if \(K\mapsto P(K)\) induces a homomorphism \(\mathcal{C}\to\mathcal{C}\),
then it is the zero homomorphism, i.e.\ $P(K)$ is slice for all $K$.
Some evidence for this conjecture is provided by the fact that 
\(\nu(P(K)) = 0\) holds for all knots \(K\) and slice-torus invariants \(\nu\).
Indeed, this follows because by assumption on \(P\),
\(\nu(P(K^{\# n})) = \nu(P(K)^{\#n})\), which in turn is equal to \(n\cdot \nu(P(K))\) for all \(n\in\Z\); but the set \(\{\nu(P(K))\mid K\text{ a knot}\}\) contains at most two elements by \cref{prop:theta:winding0}.
Using \cref{thm:main}, we find another restriction for such \(P\).
%
%In the case that \(\nu = s_2\) and \(P\) has wrapping number two, we get the following additional evidence for the conjecture.
%\comment{Claudius: Why is this further evidence? \(P(K)\) being slice does not imply \(P_t(K)\) is slice for all \(t\). But the result gives a further restriction on \(P\) for wrapping number 0: \(\theta_2(P)=\infty\).}

\begin{corollary}
        Let \(P\) be a pattern with wrapping number two and winding number zero, such that
        the function \(K \mapsto P(K)\) induces an endomorphism on the smooth concordance group.
        Then \(s_2(P_t(K)) = 0\) holds for all knots \(K\) and all integers \(t\).
        In particular, \(\theta_2(P)=\infty\).
\end{corollary}
\begin{proof}
Let \(P\) be a pattern as in the hypothesis, and assume \(s_2(P_t(K)) > 0\) for some knot \(K\) and some \(t\in\Z\).
Let \(u\) be a positive integer such that \(-4u \leq t-\theta_2(K)\). Then
\[
0 < s_2(P_t(K)) = s_2(P_{t-\theta_2(K)}(U)) \leq s_2(P_{-4u}(U)) =
s_2(P(T_{2,3}^{\# u})) = u \cdot s_2(P(T_{2,3})).
\]
Here, the first two equalities follow from \cref{thm:main} and the last equality from the assumption on~\(P\). 
Consequently,  \(s_2(P(T_{2,3})) \neq 0\). 
Therefore, the value \(s_2(P(T_{2,3}^{\# u}))\) is unbounded as \(u\to\infty\). 
This contradicts the fact that the knots \(P(T_{2,3}^{\# u})\) and \(P(T_{2,3}^{\# u'})\)
are related by a smooth cobordism of genus one for any \(u,u'\in\Z\).
Hence we have led the assumption \(s_2(P_t(K)) > 0\) ad absurdum. One proceeds similarly for \(s_2(P_t(K)) < 0\).
So it follows that \(s_2(P_t(K)) = 0\) for all knots \(K\) and integers \(t\).
%\comment{Claudius: I changed a number of signs, such that we can assume \(u\) is positive.  Please double-check. }
\end{proof}

%\begin{theorem}\label{thm:geom1}
%        For all knots \(K\) and integers \(n\) with \(|n|\) less than or equal to the 3-genus \(g(K)\) of \(K\),
%        there exists a knot \(J\) satisfying the following conditions:
%        \begin{enumerate}[(i)]
%        \item \(J\) and \(K\) are topologically concordant,
%        \item \(J\) and \(K\) have S-equivalent Seifert forms,
%        \item \(J\) and \(K\) have the same 3-genus,
%        \item \(s_2(J) = 2n\).
%        \end{enumerate}
%\end{theorem}
%\todo{\cref{thm:geom1} can be proved by infecting with the Alex-1 knot \(P(-3,5,7)\) with \(\theta_2 = 4\).
%However, it does not really need \cref{thm:main}. Can we prove the strengthened version with (v) \(s_0(J) = 2m\) for given \(m\) with \(|m| \leq g(K)\), under the assumption that \cref{thm:main} holds for \(c = 0\)? That would be cool.}

\subsection{\texorpdfstring{Does \cref{thm:main} hold over arbitrary coefficients?}{Does Theorem~\ref{thm:main} hold over arbitrary coefficients?}}
%Let us start with Rasmussen invariants $s_c$ over fields of characteristic $c\neq 2$.
As discussed in \cref{subsec:multicurves}, our proof of \cref{thm:main} heavily relies on the immersed multicurves technology.
At the moment, this technology is the furthest developed for ground fields of characteristic~2.
However, multicurves exist for fields of all characteristics, and we expect our proof of \cref{thm:main} to extend without major changes to all ground fields, thus giving a positive answer to \cref{que:generalize} for $\nu = s_c$. 

\begin{conjecture}\label{conj:not2}
	For all knots \(K\subset S^3\), all \(c\) equal to 0 or a prime, and all patterns \(P\) with wrapping number two and winding number zero, we have
	\[
	s_c(P(K))
	=
	s_c(P_{-\theta_c(K)}(U)),
	\]
	where \(\theta_c(K)\) is the integer from \cref{def:theta}. 
\end{conjecture}
%If this conjecture is true, it follows that \(\theta_c\) is a concordance homomorphism.
%which will actually depend on \(c\).
%We will come back to the invariants $s_c$ and $\theta_c$ in \cref{sec:differentfields}.\comment{Claudius: actually much earlier, namely for winding number 2 patterns. I would delete this sentence.}

%\subsection{Pattern-twisting behaviour of slice-torus invariants for patterns with wrapping number two and winding number two}
\subsection{\texorpdfstring{Winding number ±2 patterns and \({\bm\theta_\nu}\)-rationality}{Winding number ±2 patterns and ϑ\_{}ν-rationality}}
\label{subsec:slice-torus2}
The techniques for the proof of \cref{thm:main} can be adapted to patterns with wrapping number 2 and winding number \(\pm2\). 
As before, this allows us to compute the Rasmussen invariant of any satellite knot (with such a pattern) as a function of invariants extracted independently from the pattern and the companion. 
Interestingly, the technical difficulty that tied us to working over characteristic \(c=2\) in \cref{thm:main} disappears: We can now work over arbitrary characteristics \(c\).

However, compared to the setting considered previously, there is one fundamental difference, which is not particular to the Rasmussen invariants, but which appears in the behaviour the Ozsváth-Szabó invariant \(\tau\), too. 
In fact, this structural difference can be understood for all slice-torus invariants.
So we defer the statement of our winding number ±2 pattern results for the Rasmussen invariants to \cref{sec:intro:winding2} and first consider the following general setting: 
As in \cref{subsec:generalize1}, let \(\nu\) denote a slice-torus invariant and \(K\) a knot, but now \(P\) denotes a pattern of wrapping number two and winding number~\(\pm2\). 
In analogy with \cref{prop:theta:winding0}, we have the following result, which is illustrated in \cref{fig:def:theta:winding2}. 

\begin{proposition}\label{prop:theta:winding2}
	For any fixed triple \((\nu,P,K)\), the function \(\Z\rightarrow\Z\) given by \(t\mapsto \nu(P_t(K))\) is either affine of slope \(1\) or it has a single ``stationary point'', in the sense that there exists a unique integer \(\theta\in\Z\) such that for all \(t\in\Z\)
	\begin{equation}\label{eq:theta:winding2}
	\nu(P_{t}(K))
	=
	\nu(P_{t-1}(K))
	+
	\begin{cases*}
	0
	&
	if \(t=\theta\)
	\\
	1
	&
	otherwise.
	\end{cases*}
	\end{equation}
\end{proposition}
This result was shown for cables by van Cott \cite[Theorem~2]{zbMATH05705465}, and then generalised by Roberts.
It is a special case of \cite[Theorem~2]{zbMATH05964425}, but for the sake
of completeness, we will provide a quick proof in \cref{subsec:geom:slicetorus}.
In analogy to \cref{def:theta}, we have:

\begin{definition}\label{def:theta:winding2}
	For any fixed triple \((\nu,P,K)\), we define \(\theta_{\nu}(P,K)\in\Z\cup\{\infty\}\) 
	as the unique integer \(\theta\) satisfying \eqref{eq:theta:winding2} if such an integer exists and \(\infty\) otherwise; see \cref{fig:def:theta:winding2} for an illustration. 
	Moreover, we set 
	\[
	\theta_{\nu}(P)
	\coloneqq
	\theta_{\nu}(P,U)
	\quad
	\text{and}
	\quad
	\thetap_\nu(K)
	\coloneqq
	\theta_{\nu}(C_{2,1},K),
	\]
	where \(C_{2,1}\) is the $(2,1)$-cable pattern, i.e.\ the pattern corresponding to the rational tangle \(Q_{-1}\) in \cref{fig:Qm1}.
	For \(c=0\) or prime, we use the subscript \(c\) in place of the subscript \(\nu=\nicefrac{s_c}{2}\).
%	We write
%	\(
%	\thetap_c(K)
%	\coloneqq 
%	\thetap_{\nicefrac{s_c}{2}}(K)
%	\)
%	and 
%	\(
%	\theta_c(P)
%	\coloneqq
%	\theta_{\nicefrac{s_c}{2}}(P)
%	\). 
\end{definition}

The key difference between winding number 0 and winding number \(\pm2\) patterns is the following:
Unlike $\theta_\nu(K)$, the invariant $\thetap_\nu(K)$ need not be finite for all knots \(K\); in other words, there is no analogue of \cref{prop:theta_nu}. 
This motivates the following terminology.

\begin{definition}\label{def:theta-rational}
	If \(\thetap_\nu(K)\) is finite, we call \(K\) \emph{\(\theta_\nu\)-rational}, or \emph{\(\theta_c\)-rational} if $\nu = \nicefrac{s_c}{2}$ for \(c=0\) or prime. 
\end{definition}

%The choice of terminology will become clear in \cref{def:theta-rational}, where we will characterize \(\theta_c\)-rational knots \(K\) in terms of their associated tangle invariants \(\BNr_a(T_K;\F_c)\).

\begin{remark}\label{rmk:slice}
Since for all patterns $P$, $P(K)$ and $P(J)$ are concordant if $K$ and $J$ are,
it follows that $\theta_{\nu}(P, K)$ is a concordance invariant, for winding number $0$ and $\pm 2$.
So, $\theta'_{\nu}(K)$ and \(\theta_\nu\)-rationality are also concordance invariants.
In particular, slice knots are \(\theta_\nu\)-rational with $\theta'_\nu = \theta_\nu = 0$, since
$\nu(C_{2,\pm 1}(U))=0$ implies that $\theta'_\nu(U) = 0$, and 
$\theta_\nu(U) = 0$ was computed in \cref{ex:nu_of_twist_knots}.
\end{remark}

\begin{example}\label{ex:rationality}
	For \(K=T_{2,3}\) the right-handed trefoil knot, one may compute that  $s_2(C_{2,2t+1}(K)) = 4 + 2t$ holds for all integers $t$.
  So the trefoil knot is not \(\theta_2\)-rational.
  On the other hand, the figure eight knot is $\theta_{\nu}$-rational with $\theta'_\nu = 0$ for all $\nu$, as a consequence of the following proposition.
\end{example}

% \begin{observation}\label{obs:theta0:a}
% 	Any slice knot \(K\) is \(\theta_\nu\)-rational.  In fact, \(\thetap_\nu(K)=0=\theta_\nu(K)\). 
% \end{observation}
%\begin{proof}
%	The cable knot \(C_{2,\pm1}(K)\) is concordant to \(C_{2,\pm1}(U)=U\) and hence \(s_c(C_{2,1}(K))=s_c(C_{2,-1}(K))\). 
%	This shows \(\thetap_\nu(K)=0\). 
%	The second equality follows similarly.
%\end{proof}
\begin{proposition}\label{obs:theta0:b}
	Any knot \(K\) with order two in the smooth concordance group, such as an amphicheiral knot, is \(\theta_\nu\)-rational and satisfies \(\thetap_\nu(K)=0\).
\end{proposition}
\begin{proof}
	The knot $-C_{2,1}(K)$ is isotopic to $C_{2,-1}(-K)$, which is concordant to $C_{2,-1}(K)$.
	So we have $-\nu(C_{2,1}(K)) = \nu(C_{2,-1}(K))$. But since $\nu(C_{2,1}(K)) - \nu(C_{2,-1}(K)) \in \{1,0\}$
	by \cref{prop:theta:winding2},
	it follows that $\nu(C_{2,1}(K)) = \nu(C_{2,-1}(K)) = 0$, and thus \(\thetap_\nu(K)=0\).
\end{proof}

\begin{question}\label{que:is-thetap-nu-redundant}
        Does \(\thetap_{\nu}(K)=\theta_{\nu}(K)=0\) hold for all slice-torus invariants \(\nu\) and {\(\theta_\nu\)\nobreakdash-rational} knots \(K\)? 
\end{question}

%In \cref{thm:theta-rationals}, we shall see that \(\thetap_{c}(K)=\theta_{c}(K)\) for any \(\theta_c\)-rational knot \(K\), answering half of \cref{que:is-thetap-nu-redundant} for \(\nu=\nicefrac{s_c}{2}\). The other half is still open in this case. 
For \(\nu=\tau\), \cref{que:is-thetap-nu-redundant} is already known to admit a positive answer. 
The answer, which relates \(\theta_\nu\)-rationality to Hom's concordance invariant \(\varepsilon(K)\in\{0,\pm1\}\), is an immediate consequence of \cite[Theorem~1]{HomCables}:

\begin{proposition}
	For any knot \(K\), the following conditions are equivalent: \(\thetap_{\tau}(K)=0\); \(K\) is \(\theta_\tau\)-rational; \(\varepsilon(K)=0\). 
	Moreover, each of them implies \(\theta_{\tau}(K)=2\tau(K)=0\).\qed
\end{proposition}

%The invariant \(\varepsilon(K)\) distinguishes between two cases in her formula for computing the Ozsváth-Szabó \(\tau\) invariant of \((p,q)\)-cables \(C_{p,q}(K)\) of knots \(K\) for any coprime integers \(p,q\in\Z\) \cite{HomCables}. 
%Setting \(p=2\), the formula says that if \(\varepsilon(K)\neq0\) then \(t\mapsto \tau(C_{2,2t+1}(K))\) is an affine function of slope 1. 
%If on the other hand \(\varepsilon(K)=0\), then there exists a unique integer \(t\)  such that \(\tau(C_{2,2t+1}(K))=\tau(C_{2,2t-1}(K))\).
%In fact, this integer \(t\) is always 0. 
%Since \(\varepsilon(K)=0\) also implies that \(\tau(K)=0\), this suggests that this integer \(t=0\) should perhaps be interpreted as \(2\tau(K)\), which earlier in \cref{sec:intro:tau:winding0}, we identified as the Heegaard Floer analogue of \(\theta_c(K)\).
%This matches our observation that \(\theta_c(K)=0\) for all knots \(K\) that we know are \(\theta_c\)-rational. 
%However, since many of those are slice, we do not know if this is just a coincidence.
%We expect it is not.

%\begin{question}
%     Is there any \(c\) and \(\theta_c\)-rational knot \(K\) with \(\theta_c(K)\neq 0\)? 
%\end{question}

Computational evidence so far seems to suggest that \(\theta_c\)-rationality is independent of the characteristic $c$ and in fact agrees with \(\theta_\tau\)-rationality \cite{thetatable}. 
%For knots \(K\) up to 8 crossings and in fact all knots we have checked,  \(\varepsilon(K)=0\) iff \(K\) is \(\theta_c\)-rational.
% evidence so far: knots that have infinite order in the concordance group and that are  \(\theta_c\)-rational for c=2,3,5 and \(\theta_\tau\)-rational
%7_7
%8_1
%8_{13}
%9_{14}
%10_{130}
%10_{135}
%10_{136}
%10_{141}
% there are plenty of examples with ε≠0 and that are not \(\theta_c\)-rational for any c=2,3,5. 
%\todo{Only knot missing is \(8_{18}\).}
%\todo{Maybe comment on whether $\theta_c$-rationality depends on $c$} % Claudius: this is implicit in the previous question.

\begin{question}
	Is there a \(\theta_c\)-rational knot \(K\) with \(\varepsilon(K)\neq0\)? 
	Or conversely, a non-\(\theta_c\)-rational knot \(K\) with \(\varepsilon(K)=0\)? 
\end{question}

The condition \(\varepsilon(K)=0\) admits a geometric description in terms of Hanselman, Rasmussen, and Watson's multicurve invariant \(\HFhat(S^3\smallsetminus\mathring{\nu}(K))\) of the knot exterior \(S^3\smallsetminus\mathring{\nu}(K)\) \cite{HRW}. 
It is equivalent to the condition that a particular component of the multicurve invariant \(\HFhat(S^3\smallsetminus\mathring{\nu}(K))\) agrees with  \(\HFhat(S^3\smallsetminus\mathring{\nu}(U))\) \cite[discussion after Example~48]{HRWcompanion}, which is a huge restriction. 
Therefore, the case \(\varepsilon(K)\neq0\) should be understood as the generic case. 

For the same reason, ``most'' knots should not be \(\theta_c\)-rational: 
In \cref{def:theta-rational} and \cref{prop:theta-rational}, we will characterise \(\theta_c\)-rational knots \(K\) in terms of their associated tangle invariants \(\BNr_a(T_K;\F_c)\).
Specifically, a knot is \(\theta_c\)-rational if and only if the component \(\BNr_a(T_K;\F_c)\) of the multicurve invariant \(\BNr(T_K;\F_c)\) is equal to \(\BNr_a(Q_n;\F_c)\) for some rational tangle \(Q_n\), \(n\in2\Z\). %
So there might actually be a deeper reason for the similarity between \(\theta_c\)- and \(\theta_\tau\)-rationality.
\subsection{\texorpdfstring{Winding number ±2 patterns, Rasmussen invariants, and \(\bm{\tau}\)}{Winding number ±2 patterns, Rasmussen invariants, and τ}}
%\subsection{Rasmussen invariants and patterns with wrapping number two and winding number two}
\label{sec:intro:winding2}

As discussed in the previous subsection, the behaviour of Rasmussen invariants of cables with winding number \(\pm2\) exhibits a phenomenon that is absent from the setting of patterns with winding number 0: 
We need to distinguish between \(\theta_c\)-rational and non-\(\theta_c\)-rational knots. 
This dichotomy is also reflected in the behaviour of the Rasmussen invariants of satellites with general patterns of winding number \(\pm2\). 

First, we consider \(\theta_c\)-rational companions. For these, the formula looks almost the same as for winding number 0 patterns: All the information about the companion is captured by the integer \(\thetap_c(K)\). And, in fact, \(\theta_c(K) = \thetap_c(K)\).

\begin{restatable}{theorem}{thetarationals} \label{thm:theta-rationals}
	Let \(K\) be a \(\theta_c\)-rational knot. Then \(\theta_c(K) = \thetap_c(K)\), and for any pattern \(P\) with wrapping number \(2\) and winding number $\pm 2$,
	\[
	s_c(P(K))
	=
	s_c(P_{-\theta_c(K)}(U))+6\theta_c(K).
	\]
\end{restatable}

%\begin{corollary}\label{cor:thetap-is-concordance-homomorphism}
%	\(\thetap_c\) induces a concordance homomorphism when restricted to the subgroup of the smooth concordance group generated by \(\theta_c\)-rational knot.
%\end{corollary}
%\comment{Claudius: Do we know anything about this subgroup? What about $\tau$?}
%\begin{proof}
%	Same arguments as the proof of \cref{thm:main2}.
%	\comment{Claudius: Where is this proof?}
%\end{proof}

This gives a positive answer to half of \cref{que:is-thetap-nu-redundant} for \(\nu=\nicefrac{s_c}{2}\).
We conjecture that the other half holds true as well:

\begin{conjecture}\label{conj:theta-rational-implies-theta=0}
	For any \(\theta_c\)-rational knot \(K\), \(\theta_c(K)= 0\). 
\end{conjecture}

Note that if this conjecture holds, then the formula in \cref{thm:theta-rationals} becomes identical to the formula in \cref{thm:main}.

\begin{remark}
	As for \(\tau\) and \(\varepsilon\), the converse of \cref{conj:theta-rational-implies-theta=0} is false: 
	There are many knots \(K\) with \(\theta_c(K)=0\) that are not \(\theta_c\)-rational, e.g.\ \(K=T_{2,3}\#T_{2,3}\#-\!T_{3,4}\) for $c = 2$.
\end{remark}

\cref{obs:theta0:b} and \cref{thm:theta-rationals} imply a special case of
\cref{conj:theta-rational-implies-theta=0}:
%for amphicheiral knots:
\begin{corollary}
Let $K$ be a knot with order two in the concordance group, e.g.~an amphicheiral knot.
Then $K$ is $\theta_c$-rational and $\theta_c(K) = \thetap_c(K) = 0$.\qed
\end{corollary}

\begin{figure}[t]
	\begin{subfigure}{0.16\textwidth}
		\centering
		\includegraphics{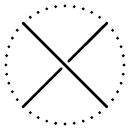}
		\caption{\(Q_{-1}\)}\label{fig:Qm1}
	\end{subfigure}
	\begin{subfigure}{0.16\textwidth}
		\centering
		\includegraphics{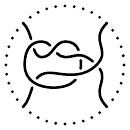}
		\caption{\(T'_a\)}\label{fig:Tp_a}
	\end{subfigure}
	\begin{subfigure}{0.16\textwidth}
		\centering
		\includegraphics{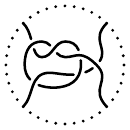}
		\caption{\(T'_b\)}\label{fig:Tp_b}
	\end{subfigure}
	\begin{subfigure}{0.16\textwidth}
		\centering
		\includegraphics{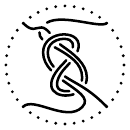}
		\caption{\(T'_c\)}\label{fig:Tp_c}
	\end{subfigure}
	\begin{subfigure}{0.16\textwidth}
	\centering
	\includegraphics{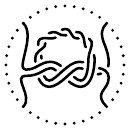}
	\caption{\(T'_d\)}\label{fig:Tp_d}
\end{subfigure}
	\begin{subfigure}{0.16\textwidth}
	\centering
	\includegraphics{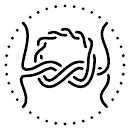}
	\caption{\(T'_e\)}\label{fig:Tp_e}
\end{subfigure}
	\caption{Some tangles for patterns with winding number \(\pm2\)}\label{fig:pattern_tangles:winding2}
\end{figure}

%The second difference to winding number 0 patterns is that for companion knots that are not \(\theta_c\)-rational, the information contained in the companion is not captured by a single integer like~\(\theta_c(K)\). 
%Instead, the answer depends on an invariant \(\theta_c(P)\in\Z\cup\{\infty\}\) that we associate with a pattern \(P\) by considering the behaviour of the function \(t\mapsto s_c(P_t(U))\). 
%In \cref{sec:geom}, we show the following:
%This function is either affine, in which case its slope is 2 and we set \(\theta_c(P) = \infty\), or there exists a unique integer \(\theta_c(P)\) such that
%\[
%s_c(P_{\theta_c(P)}(U)) = s_c(P_{\theta_c(P)-1}(U)).
%\]
%The similarity between the definitions of \(\theta_c(P)\) and the above characterisation of \(\theta_c\) is not a coincidence; see \cref{prop:eta-is-theta-if-theta-rational}.

%\begin{restatable}{theorem}{thetanonrationals}\label{thm:theta-non-rationals:non-infty}
%	Let \(K\) be a knot that is not \(\theta_c\)-rational. Then for any pattern \(P\) with wrapping number 2 and winding number $\pm 2$ with \(\theta_c(P)\neq\infty\),
%	\[
%	s_c(P(K))
%	=
%	s_c(P_{\theta_c(P)}(U))+s_c(C_{2,1-2\theta_c(P)}(K)).
%	\]
%\end{restatable}

Next, let us turn to non-\(\theta_c\)-rational knots. 
Again, the information about the companion knot is contained in a single integer. 
However, in this case, this integer is simply the Rasmussen invariant \(s_c(C_{2,1}(K))\) of the \((2,1)\)-cable of the companion. 

\begin{restatable}{theorem}{thetanonrationals}\label{thm:theta-non-rationals:non-infty}
	Let \(K\) be a knot that is not \(\theta_c\)-rational. Then for any pattern \(P\) with wrapping number \(2\) and winding number \(\pm 2\) with \(\theta_c(P)\neq\infty\),
	\[
	s_c(P(K))
	=
	s_c(P(U))+s_c(C_{2,1}(K))
	-
	\begin{cases*}
	2
	&
	if \(\theta_c(P)>0\)
	\\
	0
	&
	if \(\theta_c(P)\leq 0\)
	\end{cases*}
	\]
\end{restatable}

There is a version of \cref{thm:theta-non-rationals:non-infty} (namely \cref{thm:theta-non-rationals:infty}) for patterns \(P\) with \(\theta_c(P)=\infty\). 
In this case, the formula for \(s_c(P(K))\) looks identical to the one above, except that the condition for whether the last summand is equal to 0 or $-2$ is more subtle, taking into account also a certain order on the multicurve invariants associated with the pattern \(P\) and the companion \(K\).
This subtlety is illustrated in \cref{ex:theta-non-rationals:infty:after_thm}. 

In analogy to \cref{thm:winding0:thetas} we have the following results:

\begin{proposition}\label{prop:winding2:theta-splits}
	Let \(K\) be a knot and \(P\) a pattern with wrapping number 2 and winding number \(\pm2\).
	If \(K\) is \(\theta_c\)-rational then
	\[
	\theta_c(P,K)
	=
	\theta_c(P)
	+
	\theta_c(K).
	\]
	In particular, for such knots \(K\), \(\theta_c(P,K)=\infty\) if and only if \(\theta_c(P)=\infty\).
	If \(K\) is not \(\theta_c\)-rational and  \(\theta_c(P)\neq\infty\), then \(\theta_c(P,K)=\infty\).
\end{proposition}

The invariant \(\theta_c(P)\) is just as computable for winding number ±2 patterns as for winding number 0 patterns. 
Again, there are two methods: either directly from definition or indirectly from the multicurve invariants. 
For example, using \cref{prop:eta-slopes}, we can read off the following values for the patterns \(P^T\) associated with the tangles \(T\) in \cref{fig:pattern_tangles:winding2} from the multicurves in \cref{fig:lifts}:
\[
\theta_2(P^{T'_a})=\theta_2(P^{T'_d})=\theta_2(P^{T'_e})=\infty,
\quad 
\theta_2(P^{T'_b})=1,
\quad
\text{and}
\quad 
\theta_2(P^{T'_c})=0.
\]
Also, \(\theta_c\)-rationality can be determined either directly using \cref{cor:theta-rational-test:practical} or indirectly using the multicurves, see \cref{def:theta-rational}. 
For instance, the trefoil knot it not \(\theta_2\)-rational, but the figure-eight knot is,
as we have already seen in \cref{ex:rationality}. Check our online table \cite{thetatable} for further examples.

\begin{proof}[Proof of \cref{prop:winding2:theta-splits}]
	Applying \cref{thm:theta-rationals} to the pattern \(P_t\) for \(t\in\Z\), we see that \(s_c(P_t(K))\) and \(s_c(P_{t-\theta_c(K)}(U))\) have the same stationary points if they exist. 
	This implies the first identity. 
	For the second part, we apply \cref{thm:theta-non-rationals:non-infty} to the pattern \(P_t\) for \(t\in\Z\) and observe that 
	\[
	t\mapsto 
	s_c(P_t(U))
	-
	\begin{cases*}
	2
	&
	if \(\theta_c(P_t)>0\)
	\\
	0
	&
	if \(\theta_c(P_t)\leq 0\)
	\end{cases*}
	\]
	is an affine function, since \(\theta_c(P_t)=\theta_c(P)-t\). 
\end{proof}

\begin{remark}
	If \(\theta_c(P)=\infty=\theta_c(K)\), \(\theta_c(P,K)\) may be finite or infinite. 
	Examples can be easily constructed using the aforementioned \cref{thm:theta-non-rationals:infty}. 
\end{remark}

Based on the similarity between \(\tau\) and the Rasmussen invariants, it is natural to compare \cref{thm:theta-rationals,thm:theta-non-rationals:non-infty} to the behaviour of \(\tau\). 
The following is a special case of a result of Hom \cite[Theorem~5]{HomHFKandConcordance}.

\begin{theorem}\label{thm:tau-satellites-epsilon=0}
For any  knot \(K\) with \(\varepsilon(K)=0\) and any pattern \(P\) (of arbitrary winding or wrapping number),
\[
\tau(P(K))
=
\tau(P(U)).
\]
\end{theorem}

\begin{proof}[Sketch proof]
	Since Hom's theorem is more general, we describe the basic idea of the proof of \cref{thm:tau-satellites-epsilon=0}:
	For any knot \(K\), \(\varepsilon(K)=0\) is equivalent to the condition that the type~D structure \(\operatorname{\widehat{CFD}}(S^3\smallsetminus\mathring{\nu}(K))\) in bordered Heegaard Floer homology contains a direct summand identical to \(\operatorname{\widehat{CFD}}(S^3\smallsetminus\mathring{\nu}(U))\). 
	Since the box-tensor product commutes with the direct sum, \cite[Theorem~11.19]{LOT} implies that the homology group \(HFK^-(P(K))\) contains a direct summand isomorphic to \(HFK^-(P(U))\), and this summand contains the invariant~\(\tau\) \cite[equation~11.16]{LOT}. 
\end{proof}

This result is much stronger than \cref{thm:theta-rationals}. 
As we have just seen, the proof relies on bordered Heegaard Floer homology due to Lipshitz, Ozsváth, and Thurston \cite{LOT}. 
The multicurve theory from \cite{KWZ} is just a first step towards better understanding similar cut-and-paste technology in Khovanov homology, namely the cobordism complexes due to Bar-Natan \cite{BarNatanKhT}. 
Unfortunately, his theory (like Khovanov homology) is less geometric and therefore harder to apply to this type of problem. 
Still, by studying Bar-Natan's invariants of companion tangles on more strands, one might hope to show analogous behaviour for Rasmussen invariants:

\begin{conjecture}
	For any \(\theta_c\)-rational knot \(K\) and any pattern \(P\) (of arbitrary winding or wrapping number),
	\[
	s_c(P(K))
	=
	s_c(P(U)).
	\]
\end{conjecture}

Proving a result for \(\tau\) that is analogous to \cref{thm:theta-non-rationals:non-infty} seems less straightforward than the proof of \cref{thm:tau-satellites-epsilon=0}, but might well be within the limits of current technology. 
Combining \cref{thm:theta-non-rationals:non-infty} with Hom's cabling formula for \(\tau\) \cite{HomCables}, we make the following conjecture:

\begin{conjecture}
	Let \(K\) be a knot with \(\varepsilon(K)=\pm1\).	
	Then for any pattern \(P\) with wrapping number 2 and winding number $\pm 2$ with \(\theta_\tau(P)\neq\infty\),
	\[
	\tau(P(K))
	=
	\tau(P(U))
	%+\tau(C_{2,1}(K))
	+2\tau(K)
%	+
%	\begin{cases*}
%	1
%	&
%	if \(\varepsilon(K)=-1\)
%	\\
%	0
%	&
%	if  \(\varepsilon(K)=1\)
%	\end{cases*}
%	-
%	\begin{cases*}
%	1
%	&
%	if \(\theta_\tau(P)>0\)
%	\\
%	0
%	&
%	if \(\theta_\tau(P)\leq 0\)
%	\end{cases*}
	+
	\begin{cases*}
	-1
	&
	if \(\varepsilon(K)=1\) and \(\theta_\tau(P)>0\)
	\\
	1
	&
	if \(\varepsilon(K)=-1\) and \(\theta_\tau(P)\leq 0\)
	\\
	0
	&
	otherwise.
	\end{cases*}
	\]
\end{conjecture}

\subsection{\texorpdfstring{Properties of the knot invariants \(\bm{\theta_{\nu}}\)}{Properties of the knot invariants ϑ\_{}ν}}\label{subsec:tnu}\label{subsec:propthetacnu}
One may see easily that \(\theta_\nu\) induces a knot concordance invariant (see \cref{rmk:slice}).
However, it is not known whether \(\theta_\nu\) induces a knot concordance homomorphism $\mathcal{C}\to\Z$ in general,
i.e.\ whether %\(\theta_\nu(-K) = -\theta_\nu(K)\) and 
\(\theta_\nu(K \# J) = \theta_\nu(K) + \theta_\nu(J)\) holds for all knots $K$ and $J$.
We do have the following inequalities due to Park.%
\begin{theorem}[{\cite[Theorem~1.2]{MR3577888}}]
For all slice-torus invariants \(\nu\) and knots \(K\) and \(J\),
\[
\theta_\nu(K) + \theta_\nu(J) - 1 \leq
\theta_\nu(K \# J) \leq 
\theta_\nu(K) - \theta_\nu(-J) + 1.
\]
%\comment{Claudius: I added +1 on the RHS. Please double-check.}
\end{theorem}

We now discuss the effect of crossing changes on $\theta_\nu$.
Suppose $K_-$ is a knot and $\lambda$ an unknot in the complement of $K_-$ such that $\lk(K_-, \lambda)=0$.
Performing a $(+1)$-framed Dehn surgery on $\lambda$ transforms $K_-$ into another knot $K_+$ in~$S^3$.
We say that $K_+$ (or~$K_-$) arises from $K_-$ (or $K_+$) by a \emph{negative (or positive) generalised crossing change}
\cite{zbMATH01863417}.
If $\lambda$ bounds a disc $D$ that intersects $K_-$ transversely in $2n$ points, then $K_+$ arises from $K_-$ by tying a left-handed full-twist into the $2n$ strands of $K_-$ close to $D$.
We say that the generalised crossing change is \emph{$2n$-stranded}.
Note that a 2-stranded generalised crossing change is simply a crossing change in the usual sense.

Let us say that a knot invariant $y$ satisfies the \emph{($\mathit{2n}$-stranded) generalised/classical crossing change inequality}
if $y(K_-) \leq y(K_+)$ holds whenever the knot $K_+$ arises from $K_-$ by a negative generalised/classical crossing change (on $2n$ strands).
For example, Levine-Tristram signatures satisfy the generalised crossing change inequality;
slice-torus invariants satisfy the classical crossing change inequality;
$\tau$ (as a consequence of Ozsváth-Szabó \cite[Theorem 1.1]{osz10}), and $s_c$ with $c\neq 2$ (by Manolescu-Marengon-Sarkar-Willis \cite[Theorem~1.11]{MMSW}) satisfy the generalised crossing change inequality,
but it is unknown whether all slice-torus invariants do.

\begin{proposition}\label{prop:crossingchange}
Let $n \geq 1$.
If $\nu$ is a slice-torus invariant satisfying the $4n$-stranded generalised crossing change inequality, then
\[
\theta_{\nu}(K_-) \leq \theta_{\nu}(K_+) \leq \theta_{\nu}(K_-) + 4n^2
\]
holds if $K_+$ arises from $K_-$ by a $2n$-stranded negative generalised crossing change.
In particular, if $\nu$ satisfies the generalised crossing change inequality, then so does $\theta_{\nu}$.
\end{proposition}
The case $n = 1$ of \cref{prop:crossingchange} is implicit in \cite{MR3910442}.
We will prove \cref{prop:crossingchange} and variations of it in \cref{subsec:fulltwists}.
As a consequence of \cref{prop:crossingchange}, if $\nu$ is a slice-torus invariant satisfying the generalised crossing change inequality,
then $\theta_\nu(K) \geq 0$ holds for all knots that can be transformed into a slice knot by positive generalised crossing changes,
and $\theta_\nu(K) \leq 0$ holds for all knots that can be transformed into a slice knot by negative generalised crossing changes.
Combined, this already implies $\theta_\nu(K) = 0$ for many knots (cf.~also \cite{zbMATH06696649}), e.g.~for the $(2,1)$-cable of the figure-eight knot.

We suspect that the results of \cite{MMSW} extend to characteristic 2, i.e.~that $s_2$ satisfies the generalised crossing change inequality as well. 
This would in particular imply the following.%
\begin{conjecture}
	If $K_+$ arises from $K_-$ by a negative generalised crossing change on $2n$ strands, then we have
	$\theta_2(K_-) \leq \theta_2(K_+) \leq \theta_2(K_-) + 4n^2$.
\end{conjecture}

Finally, let us mention an inequality proven by Livingston-Naik.
\begin{theorem}[{\cite[Theorem~2]{doubled}}]\label{thm:tbln}
For all slice-torus invariants $\nu$ and all knots $K$,
\[
TB(K) + 1 \leq \theta_\nu(K) \leq -TB(-K),
\]
where $TB(K)$ denotes the maximal Thurston-Bennequin number of $K$.
\end{theorem}

If $\theta_\nu(-K) = -\theta_\nu(K)$ holds, as is the case for $\nu = s_2$, then one may clearly improve
the upper bound of \cref{thm:tbln} by $1$:
\begin{proposition}\label{prop:TB}
	For all knots \(K\), we have
	$TB(K) + 1 \leq \theta_2(K) \leq -TB(-K) - 1$.\myqed
\end{proposition}

\subsection{Further open questions}\label{sec:intro:further_questions}

Aside from being a concordance homomorphism, what other properties does \(\theta_2\) have in common with \(s_2\)?
For example, is \(\theta_2\) also an even integer, like \(s_2\) is?
First off, note that the divisibility of \(s_2\) by 2 merely comes from a somewhat arbitrary choice of normalisation, and one could equally well work with \(\nicefrac{s_2}2\). For \(\theta_2\), on the other hand, in light of \cref{thm:main}, our chosen normalisation appears to be the natural one (up to sign). From our discussion so far, there is no apparent reason that \(\{\theta_2(K) \mid K\text{ a knot}\} \subset n\Z\) should hold for any \(n\geq 2\). Yet, we make the following conjecture.
\begin{conjecture}\label{conj:divby4}
For all knots \(K\), \(\theta_2(K)\) is divisible by \(4\).
\end{conjecture}
Divisibility of \(\theta_2(K)\) by 2 would follow from a more general conjecture about multicurve geography over~$\mathbb{F}_2$,
see~\cite[Conjecture~3.10 and the preceding discussion]{KWZ_strong_inversions}.
In contrast, \(\theta_c(K)\) need not be divisible by 4 (or any integer $n\geq 2$) for $c\neq 2$, see \cref{sec:differentfields}.

A lower bound for the smooth slice-genus \(g_4(K)\) is provided by $|s_2(K)|/2$.
This inspires the following question.
\begin{question}\label{q:boundg4}
Is \(|\theta_2(K)|/4\) a lower bound for the smooth slice-genus \(g_4(K)\)?
\end{question}
On the one hand, we have not found a proof for this bound, and we have no deeper reason to assume that \cref{q:boundg4} should hold, except that $s_2$ and $\theta_2$ might be assumed to behave similarly. On the other hand, we have not found a counterexample.

It would be of interest to calculate the values of $s_2$ on iterated satellites, as it is possible for $\tau$.
This brings us to the following question.
\begin{question}
Is \(\theta_2(\Wh(K))\) determined by \(\theta_2(K)\) and \(s_2(K)\)?
\end{question}

\subsection{Structure of the paper}\label{subsec:structure_of_paper}
In \cref{sec:Conventions}, we fix our conventions and notation regarding tangles, satellite knots, and Rasmussen invariants.
\cref{sec:geom} contains those geometric proofs that do not rely on the multicurve technologies,
namely the proof that $\theta_2$ is a concordance homomorphism as claimed in \cref{thm:main2},
as well as the proofs of
\cref{prop:theta:winding0,prop:genus-1-knots,prop:theta:winding2,prop:crossingchange}.
In \cref{sec:review}, we review the construction and some properties of the multicurve invariants for Bar-Natan homology.
We use those in \cref{sec:RasmussenCurves} to prove \cref{thm:main} (the winding number 0 case)
and \cref{thm:theta-rationals,thm:theta-non-rationals:non-infty,thm:theta-non-rationals:infty} (the winding number $\pm 2$ case).
Finally, \cref{sec:differentfields} discusses the results of our computer calculations and the conjectures they invite. 

\section{Basic facts about tangles, satellites and Rasmussen invariants}\label{sec:Conventions}

\subsection{Conway tangles and rational tangles}
\begin{definition}
	An \textit{oriented Conway tangle} is an embedding of two intervals 
%	and a finite (possibly empty) collection of circles 
	into a 3-dimensional ball \(B^3\) 
	such that the preimage of \(\partial B^3\) consists of precisely the four endpoints of the intervals. 
	We consider such tangles up to isotopy fixing the boundary pointwise. 
	If we also identify tangles that are obtained by precomposition with orientation reversing diffeomorphisms of the intervals, 
	%or circles
	we obtain the notion of \textit{unoriented} Conway tangles. 
%	In the following, the term Conway tangle will refer to the unoriented notion. 
	
	We will usually think of Conway tangles in terms of their diagrams: 
	An \textit{oriented diagram of a Conway tangle} is an immersion of two intervals 
	% and a finite (possibly empty) collection of circles 
	into a closed 2-dimensional disc $D^2$ 
	such that the preimage of \(D^2\) consists of precisely the four endpoints of the intervals and such that all self-intersections are transverse double points, together with over-under information \(\Crossing\) at each such double point. 
	We consider such diagrams up to %the usual Reidemeister moves and
        isotopies fixing the boundary pointwise. 
	An \textit{unoriented} version of this notion is obtained in the same way as above. 
	
	Identifying \(D^2\) with the unit disc in \(\mathbb{C}\), we call a tangle diagram \textit{standard} if the four tangle ends lie at the four equidistant points \(\pm e^{\pm i\pi/4}\). Identifying this unit disc with \(D^2\times\{0\}\subset\mathbb{C}\times\mathbb{R}=\mathbb{R}^3\), we map these four points into the unit ball \(B^3\) and call a Conway tangle \textit{standard} if the four tangle ends are equal to those four points. 
	Standard Conway tangles and diagrams up to the usual Reidemeister moves are in one-to-one correspondence, via suitable perturbations of the standard projection \(\mathbb{R}^3\rightarrow\mathbb{R}^2\times \{0\}\) onto the first two factors. 
	
	In the following, we will use the notions of Conway tangles and their diagrams interchangeably and assume that they are standard.
	By default, all tangles and their diagrams are unoriented. 
\end{definition}

Since we only study Rasmussen invariants of knots (as opposed to links), our definition of Conway tangles is more restrictive than the usual one, which also allows closed tangle components.

\begin{figure}[b]
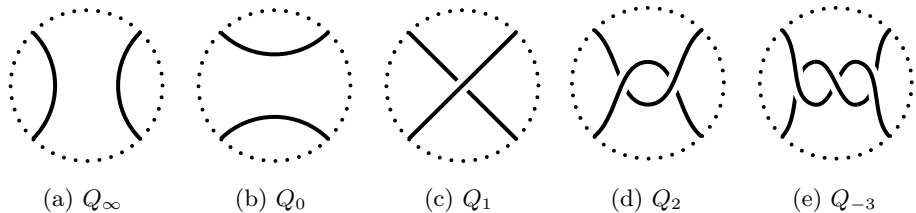

	\begin{subfigure}{0.16\textwidth}
		\centering
		\(\ratixo\)
		\caption{\(Q_{\infty}\)}\label{fig:rational_tangle:i:o}
	\end{subfigure}
	\begin{subfigure}{0.16\textwidth}
		\centering
		\(\ratoxi\)
		\caption{\(Q_{0}\)}\label{fig:rational_tangle:o:i}
	\end{subfigure}
	\begin{subfigure}{0.16\textwidth}
		\centering
		\(\ratixi\)
		\caption{\(Q_{1}\)}\label{fig:rational_tangle:oi:i}
	\end{subfigure}
	\begin{subfigure}{0.16\textwidth}
		\centering
		\(\ratiixi\)
		\caption{\(Q_{2}\)}\label{fig:rational_tangle:ii:i}
	\end{subfigure}
	\begin{subfigure}{0.16\textwidth}
		\centering
		\(\ratmiiixi\)
		\caption{\(Q_{-3}\)}\label{fig:rational_tangle:miii:i}
	\end{subfigure}
	\caption{Diagrams of some rational tangles \(Q_{\nicefrac{p}{q}}\)}\label{fig:rational_tangles}
\end{figure}

\begin{definition}\label{def:lk}
	The \textit{linking number} \(\lk(T)\) of an oriented Conway tangle \(T\) is defined as half the signed count of positive (\(\CrossingR\)) and negative crossings (\(\CrossingL\)) between the two components of the tangle \(T\) \cite[Definition~4.7]{KWZ}. 
	The linking number of an unoriented Conway tangle is well-defined up to sign. 
	In particular, the property \(\lk(T)=0\) is independent of orientations. 
\end{definition}

\begin{definition}
	The \textit{connectivity} \(\conn{T}\) of a Conway tangle \(T\) is an element of \(\{\Ni,\ConnectivityX,\No\}\) and defined as follows:
	\[
	\conn{T}=
	\begin{cases*}
	\Ni & if \(T\) connects the tangle ends as in the tangle \(\Ni\) \\
	\ConnectivityX & if \(T\) connects the tangle ends as in the tangle \(\CrossingX\)\\
	\No & if \(T\) connects the tangle ends as in the tangle \(\No\)
	\end{cases*}
	\]
\end{definition}

\begin{definition}
	A (standard) Conway tangle \(T\) is \textit{rational} if there exists an isotopy, not necessarily fixing the boundary, to the trivial tangle \(Q_\infty\) from \cref{fig:rational_tangle:i:o}.  
	A theorem of Conway \cite{Conway} puts such tangles in one-to-one correspondence with elements of \(\QPI\).  
	For each slope \(\nicefrac{p}{q}\in\QPI\), we denote the corresponding rational tangle by \(Q_{\nicefrac{p}{q}}\).
\end{definition} 

\cref{fig:rational_tangles} gives some examples of rational tangles and serves to fix our conventions.

\begin{definition}
	Given two Conway tangles \(T_1\) and \(T_2\), their \textit{union} \(T_1\cup T_2\) is the link defined by the diagram in \cref{fig:tangleunion} and their \textit{sum} \(T_1+T_2\) is the Conway tangle by the diagram in \cref{fig:tanglesum}. 
	Then, for any Conway tangle \(T\), we define the \textit{\(\nicefrac{p}{q}\)-rational filling} of \(T\) by
	\[
	T(\nicefrac{p}{q}) 
	\coloneqq
	Q_{-\nicefrac{p}{q}}\cup T.
	\]
	We call a Conway tangle \(T\) \textit{cap-trivial} if \(T(\infty)\) is the unknot. 
	%We do not want to add the additional condition that the linking number of the two components be 0. That's not the longitudinal framing for general quotients of strongly invertible knots. But it is for the double tangle T_K. 
%	\todo{ Decide if we want to also fix the standard orientation of \(T\) such that it agrees with one of the two orientations of \(T(\infty)\).}
	We write \(\muty(T)\) for the tangle obtained by rotating \(T\) by \(\pi\) around the \(y\)-axis. 
	Moreover, we write \(-T\) for the \textit{mirror} of a Conway tangle \(T\), i.e.\ the tangle obtained by switching over- and under-strands at all crossings in a diagram of \(T\).
\end{definition}

\begin{figure}
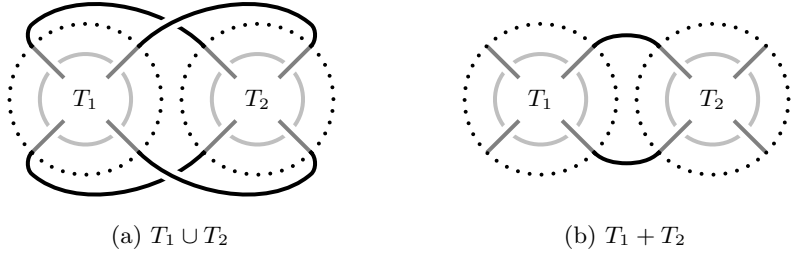

	\begin{subfigure}{0.4\textwidth}
		\centering
		\(\union\)
		\caption{\(T_1\cup T_2\)}\label{fig:tangleunion}
	\end{subfigure}
	\begin{subfigure}{0.4\textwidth}
		\centering
		\(\tanglesum\)
		\caption{\(T_1+ T_2\)}\label{fig:tanglesum}
	\end{subfigure}
	\caption{(a) The union and (b) the sum of two Conway tangles \(T_1\) and \(T_2\)}\label{fig:tangleoperations}
\end{figure}

We now collect a few identities that will be useful later. 

\begin{lemma}\label{lem:basics:tangle_identities}
	Let \(T_1\), \(T_2\), and \(T\) be Conway tangles, \(n\in\mathbb{Z}\), and \(\nicefrac{p}{q}\in\QPI\). Then:
	\begin{enumerate}[(a)]
		\item \label{lem:basics:tangle_identities:symmetry}
		\(
		T_1\cup T_2
		=
		T_2\cup T_1
		\).
		\item \label{lem:basics:tangle_identities:sum_and_union}
		\(
		(T_1+T)\cup T_2
		=
		T_1\cup (T_2+\muty(T))
		\); in particular,
		\(
		(T_1+Q_{\nicefrac{p}{q}})\cup T_2
		=
		T_1\cup (T_2+Q_{\nicefrac{p}{q}})
		\).
		\item \label{cor:twistingAndFilling}
		\(T(\nicefrac{p}{q}+n)=(T+Q_{-n})(\nicefrac{p}{q})\).
	\end{enumerate}
\end{lemma}
\begin{proof}
	The first identity follows from
	\[
	T_1\cup T_2
	=
	\big(\muty(T_1)+T_2\big)(0)
	=
	\big(\muty(T_2)+T_1\big)(0)
	=
	T_2\cup T_1.
	\]
	The second follows from 
	\begin{align*}
	(T_1+T)\cup T_2
	&=
	\big(\muty(T_1+T)+ T_2\big)(0)
	=
	\big(\muty(T)+\muty(T_1) + T_2\big)(0)
	\\
	&
	=
	\big(\muty(T_1) + T_2+\muty(T)\big)(0)
	=
	T_1\cup (T_2+\muty(T))
	\end{align*}
	and the observation that \(\muty(Q_{\nicefrac{p}{q}})=Q_{\nicefrac{p}{q}}\) for all \(\nicefrac{p}{q}\in\QPI\).
	The last identity follows from
	\begin{align*}
	T(\nicefrac{p}{q}+n)
	& 
	= 
	Q_{-\nicefrac{p}{q}-n} \cup T
	=
	(Q_{-\nicefrac{p}{q}}+Q_{-n})\cup T
	\stackrel{(b)}{=}
	Q_{-\nicefrac{p}{q}}\cup (T+Q_{-n})
	\\
	&
	=
	(T+Q_{-n})(\nicefrac{p}{q}).
	\qedhere
	\end{align*}
\end{proof}

\subsection{Satellite knots and Conway tangles}

\begin{definition}\label{def:satellite}
	A \textit{pattern} is an oriented knot in the solid torus \(S^1\times D^2\) with a fixed orientation, i.e.\ an embedding \(\iota_P\co S^1\hookrightarrow S^1\times D^2\) considered up to isotopy. 
	The \textit{wrapping number of a pattern \(P\)} is equal to the minimal number of intersection points between any representative of the equivalence class of the pattern and a meridional disc \(\{\ast\}\times D^2\), where \(\ast \in S^1\). 
	Likewise, we define the \textit{winding number of a pattern \(P\)} as the algebraic intersection number between \(P\) and such a disc. 
	
	Given an oriented knot \(K\subset S^3\), we can consider an embedding \(\iota_K\co S^1\times D^2\hookrightarrow S^3\) which identifies the solid torus with a closed tubular neighbourhood of \(K\) such that for any point \(\ast\in \partial D^2\), \(S^1\times\{\ast\}\) is mapped to a zero-framed longitude $\lambda_K$ of \(K\) and for any point \(\ast\in S^1\), \(\{\ast\}\times \partial D^2\) is mapped to a meridian \(\mu_K\) of \(K\).
	(Both of these identifications preserve orientations.)
	We then define the \textit{satellite knot \(P(K)\) of the pattern \(P\) and companion \(K\)} as the knot given by the composition \(\iota_K\circ\iota_P\). 
\end{definition}

\begin{definition}\label{def:patterntangle}
	Given an oriented Conway tangle \(T\) with connectivity \(\conn{T}\neq\No\), we can define the \textit{pattern \(P^T\) of the tangle \(T\)} as follows: 
	Consider the embedding 
	\[
	\varphi\co 
	B^3
	\lhook\joinrel\longrightarrow	
	[-2,2]\times D^2
	\longrightarrow
%	[-1,1]/\sim\times D^2 =
	S^1\times D^2
	\]
	where 
	the first map is the embedding of the unit 3-ball \(B^3\) into \([-2,2]\times D^2\) as subspaces of \(\mathbb{R}^3\) and 
	the second map is the identity on the second factor and the map $t\mapsto e^{\pi i t/2}$  on the first factor. 
	The subspace \(S^1\times\{\pm\sin(\pi/4)\}\times\{0\}\) defines a 2-component link whose restriction to \(B^3\) is the trivial tangle \(Q_0\). 
	The pattern \(P^T\) is now obtained by replacing this tangle by \(T\). The condition \(\conn{T}\neq\No\) ensures that the pattern consists of a single component.  
	
	Conversely, given a pattern \(P\) with wrapping number 2, we can find a tangle \(T_P\) with connectivity \(\conn{T_P}\neq\No\) such that \(P=P^{T_P}\). 
	We call any such tangle a \textit{pattern tangle of \(P\)}. 
	Furthermore, given an integer \(t\in\mathbb{Z}\), we define the \textit{\(t\)-twisted pattern \(P_t\) of \(P\)} as the pattern \(P^{T_P+Q_{-2t}}\), where \(T_P\) is some pattern tangle of \(P\). 
	Finally, suppose $P(K)$ is a satellite knot with companion \(K\) and pattern \(P\) with wrapping number~2 and \(t\in\mathbb{Z}\). 
	We then call \(P_t(K)\) the \textit{\(t\)-twisted satellite knot of the pattern \(P\) and companion \(K\)}.
\end{definition}

Note that pattern tangles \(T_P\) are not necessarily unique. 
For instance, the tangle $T_a$ and its Conway mutant, obtained from $T_a$ by rotation in the plane by $\pi$, both induce the same pattern although they are different tangles.
However, the twisted pattern \(P_t\) does not depend on the choice of pattern tangle, since the extra twists simply amount to a reparametrisation of the solid torus.
In particular \(P_0=P\).

\begin{definition}\label{def:double_tangle}
	Let \(K\) be an oriented knot. 
	The \textit{double tangle \(T_K\) of \(K\)} is the unoriented Conway tangle that is defined as follows: 
	Let \(L\) be the oriented link obtained as union of \(K\) with one of its zero-framed longitudes, regarded as a small push-off of \(K\). 
	Then \(L\) bounds an embedded unoriented annulus in \(S^3\), namely the image of a suitably chosen homotopy between \(K\) and the longitude.
	Now choose a small open three-ball which intersects this annulus in a trivially embedded band. 
	Then \(T_K\) is defined as the intersection of \(L\) with the complement of this three-ball. 
	We frame \(T_K\) such that \(T_K(\infty)\) is the meridional filling that produces the unknot and \(T_K(0)=L\) as oriented links if \(T_K\) is oriented such that the top and bottom left tangle ends point outwards. 
	In particular, the connectivity of \(T_K\) is \(\No\) and the linking number \(\lk(T_K)\) vanishes. 
	
	In the context of a satellite knot \(P(K)\) with wrapping number 2 pattern \(P\), we also sometimes call \(T_K\) the \textit{pattern tangle}; this terminology is justified by \cref{lem:basic:satellites}\eqref{lem:basic:satellites:tangle_decomposition}.  
	We define the \textit{pattern \(P^K\) of a knot \(K\)} as the pattern \(P^{T_K+Q_{-1}}\).
\end{definition}	

\begin{remark}
	The key property of \(T_K\) that we use in this paper is that it is cap-trivial. 
	Recall from \cite[Proposition~9]{WatsonSymmetryGroup} that up to adding twists, i.e.\ replacing \(T\) by \(T+Q_n\) for some \(n\in\Z\), cap-trivial tangles are in one-to-one correspondence with strongly invertible knots. 
	This correspondence is given by taking two-fold branched covers. 	
	In the present case, the two-fold cover of \(B^3\) branched along the tangle \(T_K\) is the exterior of \(K\# K\) for any knot \(K\).
	So the double tangle of a knot \(K\) may alternatively be defined as the quotient tangle of \(K\# K\) under the obvious strong inversion that interchanges the two summands. 
	With the conventions as in \cite{KWZ_strong_inversions}, this even recovers the same framing. 
\end{remark}

\begin{figure}
	\begin{subfigure}{0.31\textwidth}
		\centering
		\includegraphics{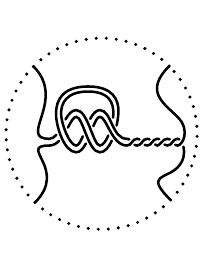}
		\caption{\(T_{3_1}\)}\label{fig:T_3_1:plain}
	\end{subfigure}
	\begin{subfigure}{0.31\textwidth}
		\centering
		\includegraphics{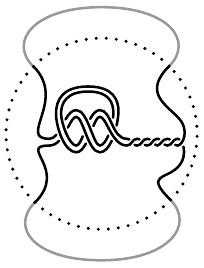}
		\caption{\(T_{3_1}(0)\)}\label{fig:T_3_1:closure_0_1}
	\end{subfigure}
	\begin{subfigure}{0.31\textwidth}
		\centering
		\includegraphics{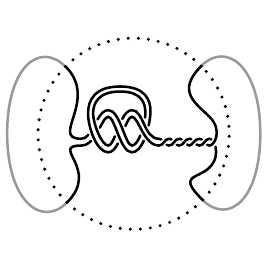}
		\caption{\(T_{3_1}(\infty)=U\)}\label{fig:T_3_1:closure_1_0}
	\end{subfigure}
	\caption{Illustration of \cref{def:double_tangle} for \(K=3_1= T_{2,3}\) the right-handed trefoil knot}\label{fig:T_3_1}
\end{figure}

For example, the double tangle \(T_U\) of the unknot \(U\) is \(Q_0=\No\). 
The double tangle \(T_{3_1}\) of the right-handed trefoil knot is shown in \cref{fig:T_3_1}, along with its two closures to the unknot and a two-component link with linking number zero.
%Note that \(T_{RHT}\) differs from the tangle in \cref{fig:T_trefoil_wo_twists} only in some twists on the right-hand side. 
The number of these twists is determined by the condition \(\lk(T_{3_1})=0\).
The orientation of \(K\) is only required in the above definition to distinguish between \(T_K\) and \(\muty(T_K)\). In fact:

\begin{proposition}\label{prop:doubletangles:orientation}
	For any oriented knot \(K\), \(T_{K^r}=\muty(T_K)\), where \(K^r\) denotes the knot \(K\) with the opposite orientation. 
	Moreover, two oriented knots \(K\) and \(J\) are isotopic if and only if the unoriented tangles $T_K$ and $T_J$ are equivalent. 
	In particular, \(K\) is a reversible knot if and only if \(T_K=\muty(T_K)\).
	Finally, mutation around the \(x\)-axis always preserves \(T_K\).
\end{proposition}

\begin{proof}
	The first statement follows from the definition of the double tangle. 
	So does the only-if-direction of the second statement. 
	For the if-direction of the second statement, orient the tangles \(T_K\) and \(T_J\) such that the top and bottom left tangle ends point outwards.
	Taking the 0-closures of the tangles and removing one component gives the knots \(K\) and \(J\), respectively. 
	So \(T_K=T_J\) as unoriented Conway tangles implies \(K=J\) as oriented knots. 
	The third statement follows from the previous two. 
	The final statement follows from the observation that twists between the two tangle strands near the tangle ends on the left can be pushed to the tangle ends on the right (this move is commonly called a flype).
\end{proof}

Again, we collect a few properties and relations for later:

\begin{lemma}\label{lem:basic:satellites}
	Let \(P\) be a pattern with wrapping number 2, \(T_P\) a pattern tangle of \(P\), \(K_1\), \(K_2\), and \(K\) oriented knots, and \(t\in\mathbb{Z}\). Then: 
	\begin{enumerate}[(a)]
		\item \label{lem:basic:satellites:tangle_decomposition}
		\(
		P(K)=T_P\cup T_K.
		\)
		More generally, 
		\(
		P_t(K)
		=
		(T_P+Q_{-2t})\cup T_K.
		\)
		\item \label{lem:basic:satellites:pattern_of_unknot}
		\(P_t(U)=T_P(2t)\). 
		\item \label{lem:basic:satellites:connected_sum}
		\(
		P_t^{K^r_1}(K_2)
		=
		T_{K_1\#K_2}(2t+1)
		=
		P_t^{K^r_2}(K_1).
		\)
	\end{enumerate}
\end{lemma}

\begin{figure}[t]
	\centering
	%	\begin{subfigure}{0.35\textwidth}
		%		\centering
		\labellist
		%		\pinlabel $T_K$ at 50 20.5
		%		\pinlabel $\rotatebox{-90}{$\rightsquigarrow$}$ at 49 50
		\pinlabel $K$ at 81 76
		\pinlabel $\textcolor{black!50!green}{\mu_K}$ at 89 36
		\pinlabel $\textcolor{black!50!green}{\lambda_K}$ at 130 48
		\endlabellist
		\includegraphics[scale=1.2]{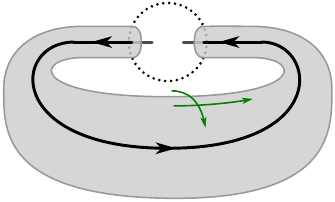}
		%		\caption{for double tangles $T_K$}\label{fig:conventions:double-tangle}
		%	\end{subfigure}
	\quad
	%	\begin{subfigure}{0.5\textwidth}
		%		\centering
		\labellist%
		\pinlabel $T_K$ at 81 76
		\pinlabel \rotatebox{-10}{$T_P$} at 51 27
		\pinlabel $\textcolor{black!50!green}{\mu_K}$ at 89 36
		\pinlabel $\textcolor{black!50!green}{\lambda_K}$ at 130 48
		\pinlabel \rotatebox{16}{$\underbrace{\hspace{2.5cm}}_{t\text{ full twists}}$} at 115 16
		\endlabellist%
		\includegraphics[scale=1.2]{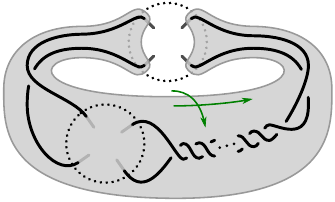}
		%		\caption{for twisted satellites $P_t(K)$ with $t>0$}\label{fig:conventions:satellite}
		%	\end{subfigure}
	\caption{Conventions for double tangles $T_K$ and twisted satellites $P_t(K)$ with $t>0$.
		On the left, a closed tubular neighbourhood of \(K\) is highlighted in grey, which, on the right, is replaced by the solid torus containing the pattern \(P\).}\label{fig:conventions}
\end{figure}
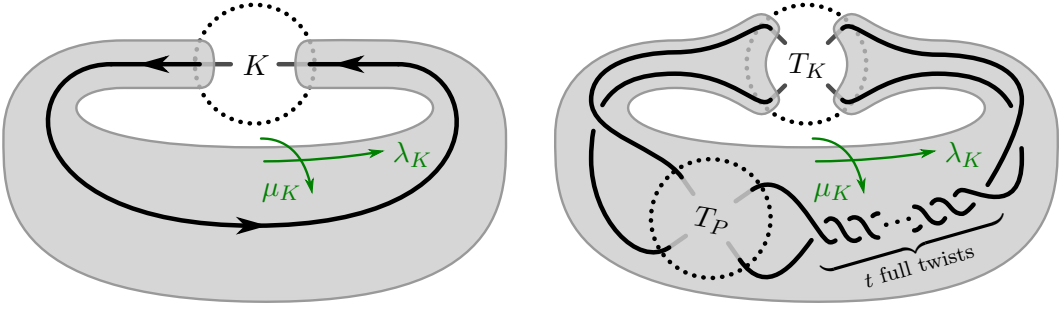

\begin{proof}
	In the construction of the satellite knot \(P(K)\) in \cref{def:satellite}, we identify a tubular neighbourhood of the companion knot \(K\) with the solid torus $S^1\times D^2$. 
	Using this identification, we can regard \(S^1\times\{\pm\sin(\pi/4)\}\times\{0\}\), the subspace of the solid torus considered in \cref{def:patterntangle}, as a link in \(S^3\). 
	By comparison with the definition of the double tangle, \cref{def:double_tangle}, we see that this link in fact coincides with \(T_K(0)=Q_0\cup T_K\).
	The pattern \(P\) is obtained from \(S^1\times\{\pm\sin(\pi/4)\}\times\{0\}\) by replacing the trivial tangle \(Q_0\) by the pattern tangle \(T_P\). 
	Hence \(P(K)\) is obtained from \(T_K(0)=Q_0\cup T_K\) by replacing \(Q_0\) by the pattern tangle \(T_P\), which proves the first identity in~\eqref{lem:basic:satellites:tangle_decomposition}.
	The second identity follows from the first, since \(T_P+Q_{-2t}\) is, by definition, a pattern tangle for \(P_t\).
	See \cref{fig:conventions} for an illustration of our conventions. 
	
	The identity in \eqref{lem:basic:satellites:pattern_of_unknot} follows from~\eqref{lem:basic:satellites:tangle_decomposition} together with the fact that \(T_U=\No\).
	The first equality in \eqref{lem:basic:satellites:connected_sum} follows from
	\begin{align*}
	P_t^{K^r_1}(K_2)
	&=
	(T_{K^r_1}+Q_{-2t-1})\cup T_{K_2}
	=
	Q_{-2t-1}\cup(T_{K_2}+T_{K_1})
	=
	Q_{-2t-1}\cup T_{K_2\#K_1}
	\\
	&=
	T_{K_1\#K_2}(2t+1),
	\end{align*}
	where the second step uses \cref{prop:doubletangles:orientation} and \cref{lem:basics:tangle_identities}\ref{lem:basics:tangle_identities:sum_and_union} and the last step uses \(K_1\# K_2=K_2\# K_1\). Using this last symmetry, the second equality now follows from the first. 
\end{proof}

\begin{definition}
	For \(t\in\mathbb{Z}\), define the pattern
	\[
	C_{2,2t+1}\coloneqq P^{Q_{-1-2t}}=P_t^{Q_{-1}}.
	\]
	If \(K\) is a knot, \(C_{2,2t+1}(K)\) is called the \textit{\((2,2t+1)\)-cable knot of \(K\)}. 
	Setting \(K=U\), 
	\(
	C_{2,2t+1}(U)
	=
	Q_0(2t+1)
%	=
%	Q_{-1}(2t)
	\)
	is the \textit{$(2,2t+1)$-torus knot}, which we also denote by \(T_{2,2t+1}\). 
\end{definition}

	Note that if \(t\geq0\), \(T_{2,2t+1}\) has the following diagram with \(2t+1\) positive crossings:
\[
T_{2,2t+1}
=
\hspace{25pt}\vc{$\underbrace{\hspace{-25pt}\includegraphics{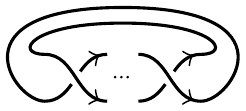}\hspace{-25pt}}_{2t+1}$}\hspace{25pt}
\]

\begin{lemma}\label{lem:cables-as-patterns}
	For any knot \(K\), \(C_{2,2t+1}(K) = T_K(2t+1) = P^{K^r}_{t}(U)\).
\end{lemma}
\begin{proof}
	This follows from \cref{lem:basic:satellites}\ref{lem:basic:satellites:connected_sum} with $K_1=U$ and $K_2=K$. 
%	\[
%	C_{2,2t+1}(K)
%	=
%	T_K(2t+1)
%	=
%	Q_{-2t-1}\cup T_K
%	=
%	(T_K+Q_{-1})\cup Q_{-2t}
%	=
%	P^K_{2t}(U).
%	\]
\end{proof}

\begin{remark}\label{def:Whitehead-Double}
	Based on our conventions, 
	the positive Whitehead double \(\Wh(K)\) of a knot \(K\) equals \(T_K(\nicefrac{1}{2})\),
        and the negative Whitehead double \(\Whn(K)\) equals \(T_K(\nicefrac{-1}{2})\). 
\end{remark}

\subsection{Rasmussen invariants}\label{sec:BNr:knots}

Rasmussen invariants are integer invariants that one can extract from Bar-Natan homology. 
The Bar-Natan homology \(\BNr(K)\) of a knot \(K\) takes the form of a bigraded \(\mathbb{Z}[H]\)-module, where \(H\) is a free variable with homological grading \(h(H)=0\) and quantum grading \(\QGrad{q}(H)=-2\). 
\(\BNr(K)\) is the homology of a chain complex \(\CBNr(K)\) of free \(\mathbb{Z}[H]\)-modules whose chain homotopy type is a knot invariant. 
The restriction \(\CBNr(K)\vert_{H=0}\) recovers Khovanov's reduced complex computing \(\Khr(K)\). 
If \(\field\) is a field, we define 
\[
\CBNr(K;\field)
\coloneqq
\CBNr(K)\otimes\field
\]
and \(\BNr(K;\field)\) as its homology.  
Usually, one takes \(\field=\F_0\coloneqq\Q\) or \(\F_c\coloneqq\Z/c\Z\), where \(c\) is a prime.
Similar to other knot homology theories, such as knot Floer homology, we have the following structure theorem:

\begin{theorem}\label{thm:CBNr:structure}
	For any knot \(K\) and any field \(\field\), \(\CBNr(K;\field)\) is bigraded chain homotopic to 
	\begin{equation}\label{eq:CBNr:decomposition}
	\left[
		\QGrad{q}^sh^0\field[H]
	\right]
	\oplus
	\bigoplus_i
	\left[
		\begin{tikzcd}
			\QGrad{q}^{a_i}h^{b_i}\field[H]
			\arrow{r}{H^{c_i}}
			&
			\QGrad{q}^{a_i+2c_i}h^{b_i+1}\field[H]
		\end{tikzcd}
	\right]
	\end{equation}
	for some uniquely determined \(s,a_i\in 2\Z, b_i, c_i\in\Z\) with $c_i > 0$. \myqed
\end{theorem}

This structure theorem has appeared in various guises in \cite{Lee,Khovanov_Frob_Ext,Turner,mvt} and 
\cite[Section~3.3.2]{KWZ}. 
As an upshot, there exists some integer \(s\in\Z\) such that 
\begin{equation}\label{eq:BNr:decomposition}
\BNr(K;\field)
\cong
\QGrad{q}^sh^0\field[H]
\oplus
(
H\text{-torsion}
).
\end{equation}

\begin{defprop}\label{defprop:s}
	Given \(c\) a prime number or 0, we define the \textit{Rasmussen invariant \(s_c(K)\) of \(K\)} as the integer \(s\) in a decomposition of \(\BNr(K;\field)\) as in \eqref{eq:BNr:decomposition}, where \(\field\) is any field of characteristic \(c\).
\end{defprop}

\begin{proof}
	Clearly, the integer \(s\) is independent of the particular decomposition \eqref{eq:BNr:decomposition}. 
	To see that the Rasmussen invariant is the same for fields of the same characteristic, observe that we may write 
	\[\CBNr(K;\field)=\CBNr(K;\F_c)\otimes\field\]
	if \(\field\) is a field of characteristic \(c\). 
	Therefore, any decomposition of the form \eqref{eq:CBNr:decomposition} for \(\F_c\) gives rise to a similar decomposition for \(\field\) with the same value of \(s\).
\end{proof}

\begin{remark}
	Originally, the Rasmussen invariants were defined in terms of spectral sequences starting at Khovanov homology.  
	The definition of the Rasmussen invariants that we have given above can be shown to be equivalent to those original definitions \cite{Khovanov_Frob_Ext}, see also \cite[Section~3.4]{KWZ}. 
	The advantage of working with Bar-Natan homology is two-fold. 
	First, \(\field[H]\)-modules are conceptually easier to work with than spectral sequences. 
	Second, and more importantly, this approach works over fields of arbitrary characteristic, because it is based on Khovanov homology over \(\Q[X]/(p(X))\) with \(p(X) = X^2 - X\). 
        It appears to us that the reasons to work with \(p(X) = X^2 - 1\) are mainly historic: 
	This is the polynomial Rasmussen \cite{RasmussenSlice} used when he constructed \(s_0\) using Lee homology \cite{Lee}. 
	This approach only generalises to fields of characteristic \(c\neq2\), but cannot be used to define \(s_2\). 
	This is because the construction of the Rasmussen invariant requires \(p(X)\) to have two distinct roots over an algebraic closure of~\(\field\).
	With the choice \(p(X) = X^2 - 1\), that condition is satisfied if and only if \(c\neq 2\).
	The approach via Bar-Natan homology gives a unified definition for all characteristics. 
\end{remark}

\begin{example}\label{ex:Rasmussen_of_two_n_torus_knots}
	For all characteristics $c$ and odd integers $n$, 	
	\[
	s_c(T_{2,n})= n+
	\begin{cases}
	-1 & n>0\\
	+1 & n<0.
	\end{cases}
	\]
\end{example}

\section{Geometric proofs of twisting behaviours}
\label{sec:geom}

\subsection{Pattern-twisting behaviour of slice-torus invariants}

%\todo{subsection not done}

\label{subsec:geom:slicetorus}

In this subsection, we continue the discussion begun in 
\cref{subsec:generalize1,{subsec:slice-torus2}} of the behaviour of $\nu(P_t(K))$
for slice-torus invariants $\nu$ and patterns $P$ with wrapping number two.
\begin{definition}[\cite{LivingstonComputations,lew2}] \label{def:slice-torus-invariant}
	A knot invariant \(\nu(K)\in\Z\) is called a \emph{slice-torus invariant} if the following three conditions are satisfied:
	\begin{enumerate}[(a)]
		\item the map from the smooth knot concordance group \(\mathcal{C}\) to \(\Z\) defined by \([K]\mapsto \nu(K)\) is a well-defined homomorphism;
		\item for all knots \(K\), \(\nu(K)\) is less than or equal to the smooth four-genus \(g_4(K)\); and
		\item this bound is sharp for positive torus knots,
		i.e.\ \(\nu(T_{p,q}) = g_4(T_{p,q}) = (p-1)(q-1)/2\)  for all positive coprime integers \(p, q\).
	\end{enumerate}
\end{definition}

We will require some well-known properties of slice-torus invariants, which we summarise in the following lemma.
\begin{lemma}\label{lemma:slicetorusprop} Let $\nu$ be a slice-torus invariant.
	\begin{enumerate}[(a)]
                \item If two knots $K$ and $J$ are related by a smooth cobordism of genus $g$, then
                      \[|\nu(K) - \nu(J)| \leq g.\]
                \item 
                      If a knot $K_+$ arises from a knot $K_-$ by a negative crossing change, then
                      \[\nu(K_-) \leq \nu(K_+) \leq \nu(K_-) + 1.\]
                \item If a knot $J$ arises from a knot $K$ by inserting $2n > 0$ positive crossings into two parallely oriented strands, then \[\nu(J) - \nu(K) \in \{n,n-1\}.\]
                \item If a knot $K$ is the plumbing of two unknotted bands with $x$ and $y$ left-handed full twists, respectively,
                      where $x,y\in\Z\setminus\{0\}$, then $\nu(K) = (\sgn x + \sgn y)/2$.                      
	\end{enumerate}
\end{lemma}
\begin{proof}[Sketch of proof]
(a) follows quickly from the definitions, since $\nu(K) - \nu(J) = \nu(K \# -J)$ and a smooth cobordism of genus $g$ between $K$ and $J$ gives rise to a smooth slice surface of genus $g$ of $K\# -J$.
For (b), see \cite[Corollary~3]{LivingstonComputations} or \cite[Corollary~4.3]{RasmussenSlice}. It is also a special case of (c).
To show (c), one uses the fact that there are smooth cobordisms of genus 1 between $J$ and $K\# T(2, 2n+1)$, and between  $J$ and $K\# T(2, 2n-1)$.
The values of $\nu$ for the knots in (d) can, for example, be computed with the sharper slice-Bennequin inequality, since the knots in question are alternating; see e.g.~\cite[Theorem~5]{lew2}, noting the different normalisation convention for slice-torus invariants.
\end{proof}

Let us now prove \cref{prop:theta:winding0} and \cref{prop:theta:winding2}.

\begin{proof}[Proof of \cref{prop:theta:winding0}]
	\label{proof:prop:theta:winding0}
	It suffices to show the following two claims:
	\begin{enumerate}[(i)]
		\item For all \(s,t\in\Z\), \(\vert\nu(P_t(K))-\nu(P_s(K))\vert\leq 1\). 
		\item For all \(t\in\Z\), \(0\leq\nu(P_t(K))-\nu(P_{t+1}(K))\leq 1\).
	\end{enumerate}
        By \cref{lemma:slicetorusprop}(a), it suffices for (i) to show that there is a genus~1 cobordism between \(P_t(K)\) and \(P_s(K)\) for any \(s,t\in\Z\). 
	This is easy to construct: 
	Let \(L\) be the two-component link obtained as the oriented resolution of \(P_t(K)\) at some crossing in the twisting region. 
	The link \(L\) may be viewed as the result of attaching a single band to \(P_t(K)\) near the resolved crossing and thus gives rise to a cobordism \(c_t\) from \(P_t(K)\) to \(L\). 
	Clearly, \(L\) does not depend on the number of twists \(t\), so we similarly get a cobordism \(c_s\) from \(P_s(K)\) to \(L\). 
	Composing \(c_t\) with the inverse of \(c_s\) gives the desired cobordism of genus~1 from \(P_t(K)\) to \(P_s(K)\).
	
	The inequalities in (ii) follow from \cref{lemma:slicetorusprop}(b) and the fact that increasing the number of full twists by~1 is equivalent to changing a positive crossing in the twisting region to a negative one. 
\end{proof}
\begin{proof}[Proof of \cref{prop:theta:winding2}]
        Let $s,t\in \Z$ with $s \leq t$. Observe that $P_t(K)$ arises from $P_s(K)$ by inserting $2(t-s)$ positive crossings        into two oriented strands with parallel orientation. So, by \cref{lemma:slicetorusprop}(c), we have
        \[
                \nu(P_t(K)) - \nu(P_s(K)) \in \{t - s, t - s - 1\}.
        \]
        From this, the proposition follows.
\end{proof}

Finally, let us prove the following statement, which implies the first half of \cref{thm:main2}.%
\begin{proposition}\label{prop:1implies2}
	Suppose \(\nu\) is a slice-torus invariant for which the answer to \cref{que:generalize} is `yes'. 
        Then $\theta_y$ induces a homomorphism from the smooth knot concordance group $\mathcal{C}$ to $\Z$.
\end{proposition}
\begin{proof}
Let knots $K$, $J$ be given. Let $P$ be the pattern $W^+ + \muty(T_K)$, where $T_K$ is the double tangle of $K$, see \cref{def:double_tangle}, and \(\muty\) is the operation that rotates tangles around their vertical axis. Then we have the following equalities for all $t\in \Z$:
\begin{align*}
\nu(W^+_{t-\theta_{\nu}(K\# J)}(U)) & = \nu(W^+_t(K \# J))
         = \nu(P_t(J)) \\
        & = \nu(P_{t-\theta_{\nu}(J)}(U))
         = \nu(W^+_{t-\theta_{\nu}(J)}(K)) \\
        & = \nu(W^+_{t-\theta_{\nu}(K)-\theta_{\nu}(J)}(U)),
\end{align*}
where the first, third, and fifth equality use the assumption that the answer to \cref{que:generalize} is `yes' for $\nu$,
and the second and fourth equality use the definition of $P$.
Setting $t$ to $\theta_{\nu}(K\# J)$ and $\theta_{\nu}(K\# J) - 1$ now respectively yields
\begin{align*}
\nu(W^+_{\theta_{\nu}(K\# J)-\theta_{\nu}(K)-\theta_{\nu}(J)}(U)) & = \nu(W^+(U)) = 0, \\
\nu(W^+_{\theta_{\nu}(K\# J)-\theta_{\nu}(K)-\theta_{\nu}(J) - 1}(U)) & = \nu(W_{-1}^+(U)) = 1,
\end{align*}
see \cref{ex:nu_of_twist_knots}.
Because of the uniqueness of the jump point (see \cref{prop:theta:winding0}),
it follows that $\theta_{\nu}(K\# J) - \theta_{\nu}(K) - \theta_{\nu}(J) = \theta_{\nu}(U) = 0$,
and thus $\theta_{\nu}$ is additive with respect to $\#$.

So, it just remains to show that $\theta_{\nu}$ is a concordance invariant.
Let $K$ and $J$ be concordant knots. It follows that for all $t\in \Z$,
the knots $W^+_t(K)$ and $W^+_t(J)$ are also concordant, and thus
$\nu(W^+_t(K)) = \nu(W^+_t(J))$. Again using the uniqueness of the jump point,
we find $\theta_{\nu}(K) = \theta_{\nu}(J)$ as desired.
This concludes the proof.
\end{proof}

\subsection{\texorpdfstring{Pattern-twisting behaviour of \(\bm{\tau}\)}{Pattern-twisting behaviour of τ}}
\label{subsec:geom:tau}

Let us now resume the investigation started in \cref{sec:intro:tau:winding0} of the behaviour of $\tau(P(K))$.

\begin{figure}
	\centering
	\labellist
        \pinlabel \rotatebox{90}{$Q_{-2u}$} at 35.5 95
        \pinlabel \rotatebox{90}{$T_J$} at 35.5 32
        \pinlabel $Q_{-2t}$ at 165 63
        \pinlabel $T_K$ at 102 63
	\endlabellist
	\includegraphics{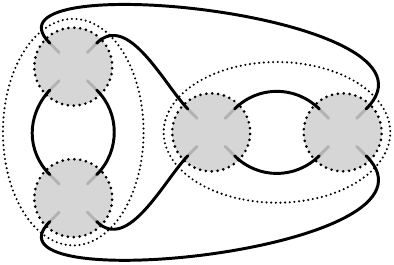}
	\caption{The knots \(D^{J,u}_t(K)\) discussed in \cref{rem:levine}.}
	\label{fig:levine}
\end{figure}
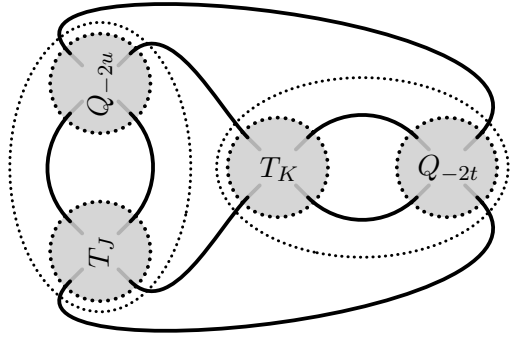
\begin{remark}\label{rem:levine}
	%Aside from some experimental calculations,
	%to be discussed in \cref{sec:examples}, that turn out to be in agreement with a positive answer to \cref{q:tau},
%	\todo{Reformulate the following paragraph about Levine's work as follows: (1) holds for $\tau$ for a larger class $\mathcal{P}$, which contains the $D^{J,u}_t$.}
	Levine~\cite{MR2971610} has established the following generalisation of Hedden's result
\eqref{eq:hedden}~$\theta_{\tau}(K) = 2\tau(K)$.
	Consider knots as shown in \cref{fig:levine}, which we denote by \(D^{J,u}_t(K)\) for knots \(K, J\) and integers $t,u$. Here, we think of \(D^{J,u}\) as a pattern, which has wrapping number 2 and winding number 0.
Indeed, in the language of \cref{def:patterntangle}, \(D^{J,u}\) is the pattern of the tangle $T$, where $T$ is $T_J + Q_{-2u}$ rotated by $\pi/2$.
% \todo{Claudius, could you check that this definition of $D^{J,u}$ makes sense and is consistent with \cref{fig:levine}?
%Claudius: Die Drehung ist nur um 90 Grad, aber die Drehrichtung ist bei nicht-reversiblen Knoten wichtig. Ich würde es über tangles definieren, dann ist es eindeutig. Beachte \(\muty(D^{J,u})=D^{J,u}\), meine neue Skizze in \cref{fig:levine} stimmt also mit dem Original überein.  }
Note that \(D^{U,\pm 1} = \Whmp\)
	and
        \begin{equation}\label{eq:levineswitch}
        D^{J,u}_t(K) = D^{K,t}_u(J)^r,
        \end{equation}
where the superscript $r$ denotes taking the reverse of a knot.
         Levine shows that
	\begin{equation}\label{eq:levine}
	\tau(D^{J,u}_t(K)) = \begin{cases}
	1  & \text{if } t < 2\tau(K) \text{ and } u < 2\tau(J), \\
	-1 & \text{if } t > 2\tau(K) \text{ and } u > 2\tau(J), \\
	0  & \text{else.}
	\end{cases}
	\end{equation}
	Assume that \(\nu\) is a slice-torus invariant for which the answer to \cref{que:generalize} is `yes', such as \(\nu = s_2\), i.e.~for all knots $K$ and patterns $P$ with wrapping number two and winding number zero, we have
        \begin{equation}\label{eq:levinereminder}
                \nu(P(K)) = \nu(P_{-\theta_{\nu}(K)}(U)).
        \end{equation}
	Then an analogous formula to \cref{eq:levine} can be deduced for \(\nu\). First note that
	\[
	\nu(D^{J,u}_t(K)) \stackrel{\eqref{eq:levinereminder}}{=} \nu(D^{J,u}_{t-\theta_\nu(K)}(U))
	\stackrel{\eqref{eq:levineswitch}}{=} \nu(D^{U,t-\theta_\nu(K)}_u(J)^r)
	\stackrel{\eqref{eq:levinereminder}}{=} \nu(D^{U,t-\theta_\nu(K)}_{u-\theta_\nu(J)}(U)^r).
	\] 
	Now, $D^{U,t-\theta_\nu(K)}_{u-\theta_\nu(J)}(U)^r$ is just a Hopf plumbing of two unknotted annuli with 
	\(t-\theta_\nu(K)\) and \(u-\theta_\nu(J)\) right-handed full twists, respectively, for which \(\nu\) can be easily determined as follows, using \cref{lemma:slicetorusprop}(d):
	\begin{equation}\label{eq:levinetheta}
	\nu(D^{J,u}_t(K)) = \begin{cases}
	1  & \text{if } t < \theta_\nu(K) \text{ and } u < \theta_\nu(J) \\
	-1 & \text{if } t > \theta_\nu(K) \text{ and } u > \theta_\nu(J) \\
	0  & \text{else.}
	\end{cases}
	\end{equation}
	In particular, a positive answer to \cref{q:tau} would recover Levine's formula \cref{eq:levine}.
\end{remark}

%\begin{proof}[Proof of \cref{prop:genus-1-knots}]
  \begin{wraptable}[9]{r}{0.26\textwidth}
    \centering
    \begin{tabular}{r|cc}
			knot        & \(\tau\) & \(s_2\) \\\hline
			\(\Wh(T_{2,3})\) & 1        & $2$ \rule{0pt}{2.3ex}\\
			\(\Wh(J)\)  & 1        & $0$ \rule{0pt}{2.3ex}\\
			\(K\)       & 1        & $-2$ \rule{0pt}{2.3ex}\\
			\(\Wh(-T_{2,3})\)& 0        & $0$ \rule{0pt}{2.3ex}\\
			\(\Wh(-J)\) & 0        & $2$ \rule{0pt}{2.3ex}
		\end{tabular}
		\medskip
    \caption{}\label{table:taus2}
	\end{wraptable}
Let us now prove \cref{prop:genus-1-knots}, which we restate for the reader's convenience.
\genusone*
\noindent\emph{Proof.}
	\label{proof:prop:genus-1-knots}
	Take again \(J = T_{3,4}^{\# 2}\# T_{-2,3}^{\#5}\), as in the proof of \cref{thm:hpc}.
	Let \(K = D^{J,0}(J)\). This knot is obtained from the Hopf plumbing \(\Sigma\) of two unknotted untwisted bands
	%which is a genus one Seifert surface \(\Sigma\) of the unknot,
	by tying the knot \(J\) into each of the two bands, with framing zero.
	So \(\Sigma'\) and \(\Sigma\) have the same genus and the same Seifert form.
	Since \(\Sigma\) is a genus one Seifert surface of the unknot, it follows that
	\(K\) has three-genus 1 and Alexander polynomial 1.
	The values of \(\tau\) and \(s_2\) for \(K\) may be computed from \cref{eq:levine} and \cref{eq:levinetheta}, respectively, using that \(\tau(J) = 1\) and \(\theta_2(J) = -4\).
	So we have constructed a knot $K$ as desired.

	The knot $L$ in the second claim may be obtained as a connected sum of $g$~summands, with each summand equal to
        one of the five knots listed in \cref{table:taus2}, or their mirror images.
	%\comment{Claudius: Don't we also need a genus 1 knot with  $\tau=0$ and $s=\pm2$?}
\qed\bigskip
%\end{proof}

\subsection{Pattern-twisting behaviour of other knot invariants}
\label{subsec:geom:other}

While we have focused on slice-torus invariants and patterns with wrapping number two so far, one could also pursue the following more general question:
For which triples $(y, \mathcal{P}, \theta)$ of a knot invariant $y$, a set of patterns $\mathcal{P}$,
and an integer-valued knot invariant $\theta$ does
\begin{equation}\label{eq:ytheta}
y(P(K)) = y(P_{-\theta(K)}(U))
\end{equation}
hold for all knots $K$ and patterns $P \in \mathcal{P}$? Let us reformulate some previously stated results in these terms:
\begin{itemize}
\item \cref{thm:main} says that \cref{eq:ytheta} holds for $y = s_2$, $\theta = \theta_2$, and
$\mathcal{P}$ the set $\mathcal{P}_1$ of all patterns with wrapping number two and winding number zero.
\item \cref{prop:theta:winding0} says that \cref{eq:ytheta} holds for $y$ equal to a slice-torus invariant $\nu$,
$\mathcal{P} = \{P_t\mid t\in\Z\} \subset \mathcal{P}_1$ for a single fixed pattern $P\in\mathcal{P}_1$,
and $\theta(K) = \theta_{\nu}(P,K)$ if $\theta_{\nu}(P,K) \neq \infty$,
and $\theta(K)$ arbitrary if $\theta_{\nu}(P,K) = \infty$.
\item Hedden's result \eqref{eq:hedden}~$\theta_{\tau}(K) = 2\tau(K)$ implies that
\cref{eq:ytheta} is satisfied for $y = \tau$, $\theta = 2\tau$ and $\mathcal{P} = \mathcal{P}_2 = \{\Wh_t \mid t\in \Z\} \cup \{\Whn_t \mid t\in \Z\} \subset \mathcal{P}_1$.
\item Levine's result discussed in \cref{rem:levine} means that 
\cref{eq:ytheta} holds for $y = \tau$, $\theta = 2\tau$ and $\mathcal{P} = \mathcal{P}_3 = \{D^{J,u}_t \mid u,t\in \Z, J\text{ any knot}\}$. Note that $\mathcal{P}_2 \subset \mathcal{P}_3 \subset \mathcal{P}_1$.
\end{itemize}\medskip

\noindent Let us also cast a quick glance at some further knot invariants:\medskip

\begin{itemize}
\item The Upsilon invariant $\Upsilon_K$, a piecewise linear function $[0, 2] \to \mathbb{R}$, is another knot invariant coming from knot Floer homology~\cite{ossz}. If the three-genus $g(K)$ of $K$ is at most $1$, then $\Upsilon_K(t) = \tau(K)\cdot(-1 + |1-t|)$ for all $t\in [0,2]$. Since $g(P(K)) \leq 1$ holds for all knots $K$ and $P\in\mathcal{P}_3$, it follows from Levine's result above that \cref{eq:ytheta} is also satisfied for $y(K) = \Upsilon_K$, $\theta = 2\tau$ and $\mathcal{P} = \mathcal{P}_3$. This was remarked first in \cite[Corollary~3.4]{MR3980298}.
\item For $n \geq 3$, the Khovanov-Rozansky $\mathfrak{sl}_n$-concordance homomorphism $\mathfrak{s}_n$ \cite{wu3,lobb1}
	is, properly normalised, a lower bound for $g_4(K)$ and satisfies $\mathfrak{s}_n(T) = g_4(T)$
	for positive torus knots $T$. But in the sense of this article, it is not a slice-torus invariant,
	since it takes values in $\tfrac{1}{n-1}\Z$ instead of $\Z$. 
	So it is conceivable that for some $n\geq 3$ and some knot $K$,
	the expression $\mathfrak{s}_n(\Wh_t(K))$ as a function of $t$ has more than one jump point.
	We do not know whether \cref{eq:ytheta} holds for $y = \mathfrak{s}_n$ and any interesting set $\mathcal{P}$ of patterns.

\item The classical Levine-Tristram signatures \(\sigma_{\omega}\), with \(\omega \in S^1\),
	satisfy \(\sigma_{\omega}(P(K)) = \sigma_{\omega}(P(U))\) for all \(P\) in the set $\mathcal{P}_4$ of patterns with winding number zero (without restrictions on the wrapping number) \cite{MR547456}.
	In other words, \eqref{eq:ytheta} holds for \(y = \sigma_{\omega}\), $\mathcal{P} = \mathcal{P}_4$, 
	and \(\theta(K) = 0\) for all knots $K$.
\end{itemize}

\subsection{Crossing changes and twists of the satellite companion}
\label{subsec:fulltwists}

In this subsection we pursue the question how a satellite $P(K)$ changes when a crossing change is applied to $K$.
This will lead us to a proof of \cref{prop:crossingchange}.
First, let us further generalise the notion of generalised crossing change from \cref{subsec:propthetacnu}.

\begin{definition}\label{def:companiontwists}
Suppose $K$ is a knot and $\lambda$ is an unknot in the complement of $K$ (the orientation of $\lambda$ is irrelevant).
Performing a $(+1)$-framed Dehn surgery on $\lambda$ transforms $K$ into another knot $J$ in~$S^3$.
Let $h = |\lk(K, \lambda)| = |\lk(J, \lambda)|$ and suppose
$\lambda$ bounds a disc $D$ that intersects~$K$ (or, equivalently,~$J$) transversely in $n$ points.
Then, we say that $J$ arises from $K$ ($K$ from $J$) by a \emph{left-handed (right-handed) $h$\nobreakdash-homologous twist on $n$ strands}.
%If $k > 0$, we say that the twist is left-handed, and right-handed otherwise.
%negative-to-positive \emph{generalised crossing change} (or \emph{null-homologous twisting}).
\end{definition}
One may achieve a left-handed twist on $n$ strands by tying a left-handed full twist into the $n$ strands of $K$ close to $D$.
A generalised negative crossing change is simply a null-homologous left-handed twist.
\begin{remark}
Note that a negative crossing change may be achieved by a null-homologous left-handed twist on two strands,
but also by a 2-homologous right-handed twist on two strands.
\cref{fig:theothertwist} contains a version of this observation for generalised crossing changes.
\end{remark}
\begin{figure}[bt]
        \vspace{3ex}
	\centering
	\labellist
	\pinlabel $K$ at 46 180
	\pinlabel \textcolor{red}{$\lambda$} at  7 265
        \pinlabel $\overbrace{\hspace{19mm}}^{n\text{ strands}}$ at 45.5 308
%        \pinlabel \textcolor{blue!50!black}{$P(\,\cdot\,)$} [r] at 42 150
        \pinlabel \textcolor{red}{\parbox{6cm}{\centering left-handed $h$-homologous\\ twist on $n$ strands}} [b] at 194 249

	\pinlabel $J$ at  363 180
        \pinlabel $-1$ at 363 242
        \pinlabel \rotatebox{90}{$\overbrace{\hspace{11mm}}^{\text{writhe }w}$} at 317 243
%        \pinlabel \textcolor{blue!50!black}{$P_{-h^2}(\,\cdot\,)$} [r] at 382 150

	\pinlabel $P(K)$  at 46 124
        \pinlabel $\underbrace{\hspace{5mm}}_{r}$ at 29.5 -6
        \pinlabel $\underbrace{\hspace{5mm}}_{r}$ at 62.5 -6
        \pinlabel $\underbrace{\hspace{19mm}}_{n\cdot r \text{ strands}}$ at 46 -20
        \pinlabel \textcolor{red}{\parbox{5cm}{\centering left-handed $(h\cdot |v|)$-homologous\\ twist on $n\cdot r$ strands}} [b] at 158 65
        \pinlabel $\rotatebox{90}{$\scriptstyle P^T$}$ at 62.5 14

	\pinlabel $P_{-h^2}(J)$ at 363.5 124
        \pinlabel $-1$ at 363 57
        \pinlabel $\scriptstyle -w-h^2$ at 346.5 14
        \pinlabel $\rotatebox{90}{$\scriptstyle P^T$}$ at 380 14

	\pinlabel $L$ at  264 124
        \pinlabel $-1$ at 264 57
        \pinlabel $\rotatebox{90}{$\scriptstyle P^T$}$ at 280.5 14
        
        \pinlabel \textcolor{blue!50!black}{\Large =} at 315 57 

        \pinlabel \scalebox{.7}{$\underbrace{}$} at 347 84
        \pinlabel \scalebox{.7}{$\underbrace{}$} at 380 84
        \pinlabel \scalebox{.7}{$\overbrace{}$}  at 347 32
        \pinlabel \scalebox{.7}{$\overbrace{}$}  at 380 32

	\endlabellist
\includegraphics{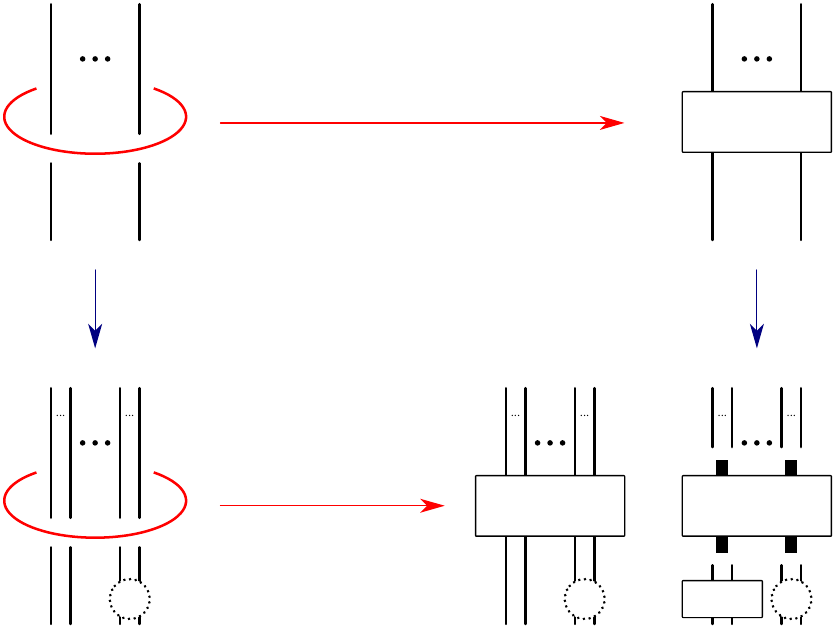}
        \vspace{5ex}
\caption{The essence of the proof of \cref{prop:satellitetwist}.
The figure shows five knots, each intersected with a neighbourhood of $\lambda$.
Away from this neighbourhood, the two knots $K$ and $J$ in the top row are identical,
and so are the three knots $P(K)$, $L$, $P_{-h^2}(J)$ in the bottom row.
Numbers in boxes indicate right-handed full twists.
In the diagram of $P_{-h^2}(J)$, bottom right, the $-1$ full twist is understood
to twist the $n$ bunches of $r$ strands each (drawn thick), but not twist each individual bunch,
cf.~\cref{fig:bunch}.
We assume that $P$ is a pattern with wrapping number $r$ and winding number $v$,
and $h = |\lk(K,\lambda)|$.}
\label{fig:twist-satellite}
\end{figure}

\begin{lemma}\label{prop:satellitetwist}
Let the knot $J$ arise from the knot $K$ by a left-handed $h$-homologous twist on $n$ strands.
Let $P$ be a pattern with winding number $v$ and wrapping number $r$.
Then $P_{-h^2}(J)$ arises from $P(K)$ by a left-handed $(h\cdot |v|)$-homologous twist on $(n\cdot r)$ strands.
\end{lemma}
\begin{proof}
By the definition of a twist, there is an unknot $\lambda$ in the complement of $K$ with $|\lk(K,\lambda)| = h$
such that $(+1)$-Dehn surgery along $\lambda$ transforms $K$ into $J$.
Denote by $L$ the knot arising from $P(K)$ by $(+1)$-Dehn surgery along $\lambda$.
Note that $L$ arises from $P(K)$ by a left-handed $(h\cdot |v|)$-homologous twist on $(n\cdot r)$ strands.
To prove the statement of the lemma, we have to show $L = P_{-h^2}(J)$.

Let us take a diagrammatic approach.
The top row of \cref{fig:twist-satellite} shows the intersection of diagrams of $K$ and $J$
with a neighbourhood of $\lambda$. 
From these diagrams, we construct diagrams of $P(K)$ and $P_{-h^2}(J)$ as follows:
First, we apply an isotopy to the pattern $P \subset S^1 \times D^2$
so that it equals the union of a $2r$-ended tangle $P^T$ sitting inside a ball with $2r$ curves that are parallel to the core curve of $S^1 \times D^2$ (as described in \cref{def:patterntangle} for $r = 2$).
We then do the following: 
\begin{itemize}
\item We replace every strand with a bunch of $r$ parallel strands;
\item we insert $P^T$ into a bunch of $r$ strands somewhere; and
\item we tie left-handed full twists into a bunch of $r$ strands somewhere, as many as the writhe of the original diagram.
\end{itemize}

From the diagram of $P(K)$, we obtain a diagram of $L$ by tying a left-handed full twist into $nr$ strands.

Let us now show that our diagrams of $L$ and $P_{-h^2}(J)$ represent the same knot.
Since these diagrams are identical outside from the neighbourhood of $\lambda$, we just have to compare them inside of that neighbourhood, as drawn in \cref{fig:twist-satellite}.

In the diagram of $L$, the full twist on $nr$ strands (consisting of $nr(nr-1)$ crossings) is isotopic to
a full twist tied into $n$ bunches of $r$ strands each (consisting of $n(n-1)r^2$ crossings),
followed by a full twist in each of the individual bunches (consisting of a total of $nr(r-1)$ crossings).
This is illustrated for $n=2=r$ in \cref{fig:bunch}.
Since $K$ is a knot, the full twists in the individual bunches may be slid around the knot so that they lie next to each other, giving a total of $n$ left-handed full twists tied into one of the bunches.
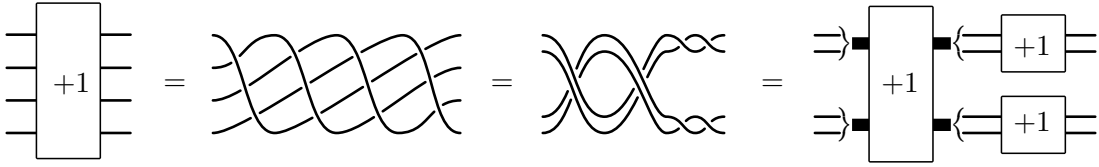
\begin{figure}[tb]
	\centering
	\labellist
	\pinlabel $+1$ at  30 36
	\pinlabel $=$  at  77 36
	\pinlabel $=$  at 226 36
	\pinlabel $=$  at 349 36
	\pinlabel $+1$ at 407 36
        \pinlabel  \rotatebox{90}{\scalebox{.7}{$\underbrace{}$}} at 381 55
        \pinlabel  \rotatebox{90}{\scalebox{.7}{$\underbrace{}$}} at 381 17.5
        \pinlabel \rotatebox{-90}{\scalebox{.7}{$\underbrace{}$}} at 434 55
        \pinlabel \rotatebox{-90}{\scalebox{.7}{$\underbrace{}$}} at 434 17.5
	\pinlabel $+1$ at 467 53
	\pinlabel $+1$ at 467 18
	\endlabellist
        \includegraphics[width=.98\textwidth]{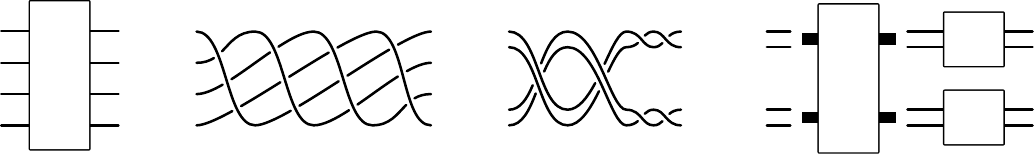}
        \caption{A full twist on four strands equals a full twist on two bunches of two strands each, followed by a full twist on each of the individual bunches.}
        \label{fig:bunch}
\end{figure}

So, the diagrams of $L$ and of $P_{-h^2}(J)$ both consist of a left-handed full twist tied into $n$ bunches of $r$ strands each, and a number of left-handed full twists tied into one of the bunches.
To show the equality of $L$ and $P_{-h^2}(J)$, it thus only remains to show that these numbers of twists agree.
For $L$, this number is $n$, as determined above.
For $P_{-h^2}(J)$, this number equals $h^2 + w$, where $w$ is the writhe of the tangle diagram $D'$ given as intersection of the knot diagram of $J$ with the neighbourhood of $\lambda$, as shown in the top right of \cref{fig:twist-satellite}.

Let us verify that $w = n - h^2$. The tangle diagram $D'$ has $n(n-1)$ crossings. Of the $n$ strands, $(n + h)/2$ are oriented upwards,
and $(n-h)/2$ are oriented downwards (or the other way around).
In a full twist, every pair of strands crosses twice, and since the full twist is left-handed, parallely oriented strands cross negatively, while oppositely oriented strands cross positively. One thus computes
\[
w(D') = -2\binom{(n + h)/2}{2}- 2\binom{(n - h)/2}{2} + 2\cdot\frac{n+h}{2}\cdot\frac{n-h}{2} = n - h^2.
\]
This concludes the proof.
\end{proof}

%We are now ready to prove the following.

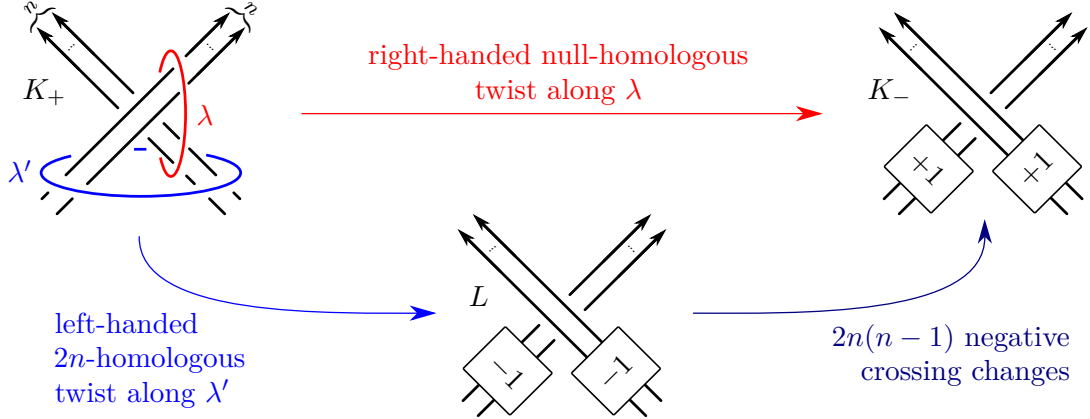
\begin{figure}[tb]
	\centering
	\labellist
	\pinlabel $K_+$ [r] at  18 145
	\pinlabel $K_-$ [r] at 388 145
	\pinlabel $L$   [r] at 203  55

	\pinlabel \textcolor{blue}{$\lambda'$} [r] at  3 110
	\pinlabel   \textcolor{red}{$\lambda$} [l] at  68 135

        \pinlabel  \rotatebox{45}{\scalebox{.7}{$\overbrace{}^{\scalebox{1.42}{$\scriptstyle \;n$}}$}} at 1 180
        \pinlabel \rotatebox{-45}{\scalebox{.7}{$\overbrace{}^{\scalebox{1.42}{$\scriptstyle \;n$}}$}} at 93 180

        \pinlabel  \rotatebox{45}{$-1$} at 255 21
        \pinlabel \rotatebox{-45}{$-1$} at 207 21

        \pinlabel  \rotatebox{45}{$+1$} at 439 111
        \pinlabel \rotatebox{-45}{$+1$} at 391 111

        \pinlabel \textcolor{red}{\parbox{6cm}{\centering right-handed null-homologous\\ twist along $\lambda$}} [b] at 230 136
        \pinlabel \textcolor{blue}{\parbox{4cm}{\raggedright left-handed $2n$-homologous\\ twist along $\lambda'$}} [tr] at 145 50
        \pinlabel \textcolor{blue!50!black}{\parbox{32mm}{\raggedleft $2n(n-1)$ negative\\ crossing changes}} [tl] at 345 45
	\endlabellist
\includegraphics[width=.95\textwidth]{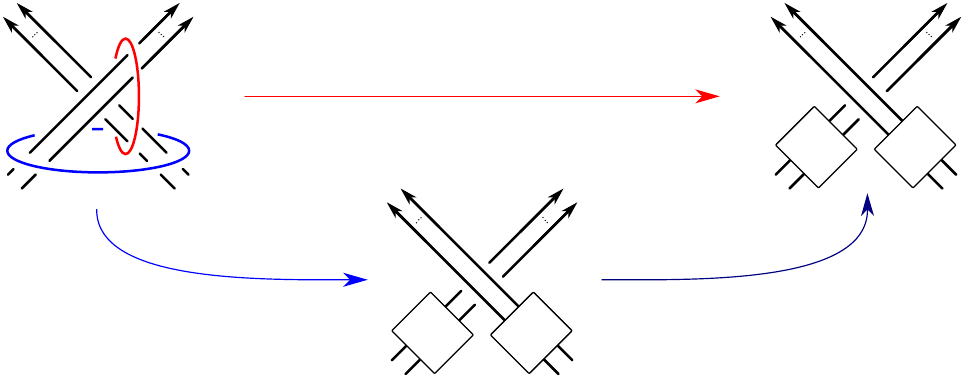}
\caption{A generalised positive crossing change on $2n$ strands may also be achieved by
a left-handed $2n$-homologous $2n$-stranded twist followed by $2n(n-1)$ negative crossing changes.}
\label{fig:theothertwist}
\end{figure}
\begin{proposition}\label{prop:generaltwists}
Let the knot $J$ arise from the knot $K$ by a left-handed $h$-homologous twist on $n$ strands.
Let $\nu$ be a slice-torus invariant satisfying the $2n$-stranded generalised crossing change inequality.
Let $P$ be a pattern with wrapping number two and winding number zero.
Assume that neither $\theta_{\nu}(J, P)$ nor $\theta_{\nu}(K, P)$ are equal to $\infty$.
Then we have
\[
\theta_{\nu}(K, P) \leq \theta_{\nu}(J, P) + h^2.
\]
In particular, if $\nu$ satisfies the generalised crossing change inequality,
then so does the knot invariant $\theta_{\nu}(\,\cdot\,, P)$
for knots $K_{\pm}$ with $\theta_{\nu}(K_{\pm}, P)\neq\infty$.
\end{proposition}
\begin{proof}
Let $t = \theta_{\nu}(K, P) - 1$.
By \cref{prop:satellitetwist}, $P_{t-h^2}(J)$ arises from $P_t(K)$ by a left-handed null-homologous twist on $2n$ strands.
Since $\nu$ satisfies the generalised crossing change inequality, it follows that $\nu(P_t(K)) \leq \nu(P_{t-h^2}(J))$.
For all $s \leq t - h^2$, we have $\nu(P_{t-h^2}(J)) \leq \nu(P_s(J))$ by \cref{prop:theta:winding0}. 
By choice of $t$, we have $\nu(P_t(K)) = \nu(P_{t+1}(K)) + 1$. All combined, we obtain
\begin{equation}\label{eq:twistproof}
\nu(P_{t+1}(K)) + 1 \leq \nu(P_s(J))
\end{equation}
for all $s \leq t - h^2$. However, since $P_{t+1}(K)$ and $P_s(J)$ are related
by smooth cobordism of genus $1$, \cref{eq:twistproof} is in fact an equality.
So the function $s\mapsto \nu(P_s(J))$ is constant for $s \leq t - h^2$.
Therefore, $\theta_{\nu}(J, P) \geq t - h^2 + 1 = \theta_{\nu}(K, P) - h^2$, as desired.
The last statement follows from setting $h=0$. 
\end{proof}
%\comment{Claudius: Wenn $\theta_{\nu}(K, P)=\infty$, dann haben wir im Allgemeinen keine Kontrolle über $\theta_{\nu}(J, P)$, richtig?  Falls die Antwort auf \cref{que:generalize} ja ist, können wir einfach $\theta_\nu(P)=\infty$ annehmen.}

\begin{proof}[Proof of \cref{prop:crossingchange}]
Applying \cref{prop:generaltwists} to the case $h = 0$, $K = K_-$, $J = K_+$ and $P = W^+$ yields
the first inequality $\theta_\nu(K_-) \leq \theta_\nu(K_+)$ to be proven.
For the second inequality, let $L$ be the knot obtained from $K_+$ by a left-handed $2n$-homologous $2n$-stranded twist as indicated in \cref{fig:theothertwist}. Applying \cref{prop:generaltwists} to $h = 2n$, $K = K_+$, $J = K_-$ and $P = W^+$, we have
$\theta_\nu(K_+) \leq \theta_\nu(L) + 4n^2$. Since $K_-$ arises from $L$ by a sequence of $2n(n-1)$ negative crossing changes, we have $\theta_\nu(L) \leq \theta_\nu(K_-)$, implying $\theta_\nu(K_+) \leq \theta_\nu(K_-) + 4n^2$ as desired.
\end{proof}

\section{Review of multicurves for Bar-Natan homology}\label{sec:review}

In~\cite{KWZ}, Artem Kotelskiy, Liam Watson, and the second author introduced the multicurve invariant \(\BNr\) for Conway tangles.
In this section, we give a brief overview of this invariant, discussing only those of its properties that we need in subsequent sections. 
More elaborate introductions to this invariant can be found in~\cite{KWZ_strong_inversions,KWZ_thin}. 

\subsection{The construction of the multicurve invariant}
%Let \(T\) be an oriented \emph{pointed} Conway tangle, that is a Conway tangle with a choice of distinguished tangle end, which we mark by~\(\ast\) and usually choose to be the one at the top left.
Given some field of coefficients \(\field\) and a (standard) oriented Conway tangle \(T\), the invariant \(\BNr(T;\field)\) is defined in two steps.
First, following a construction due to Bar-Natan, one defines a bigraded chain complex \(\KhTl{T}\) over a certain cobordism category, whose objects are crossingless tangle diagrams \cite{BarNatanKhT}. 
Any such complex can be rewritten as a chain complex \(\DD(\Diag_T;\field)\) over the following quiver algebra \cite[Theorem~1.1]{KWZ}:
% Cob and this subcategory are not equivalent, but their respective additive enlargements are!
\begin{equation}\label{eq:B_quiver}
\BNAlgH
\coloneqq
k\Big[
\begin{tikzcd}[row sep=2cm, column sep=1.5cm]
\Do
\arrow[leftarrow,in=145, out=-145,looseness=5]{rl}[description]{D_{\bullet}}
\arrow[leftarrow,bend left]{r}[description]{S_{\circ}}
&
\Di
\arrow[leftarrow,bend left]{l}[description]{S_{\bullet}}
\arrow[leftarrow,in=35, out=-35,looseness=5]{rl}[description]{D_{\circ}}
\end{tikzcd}
\Big]\Big/\Big(
\parbox[c]{90pt}{\footnotesize\centering
	$D_{\bullet} \cdot S_{\circ}=0=S_{\circ}\cdot D_{\circ}$\\
	$D_{\circ}\cdot S_{\bullet}=0=S_{\bullet}\cdot D_{\bullet}$
}\Big)
\end{equation}
We read algebra elements from right to left. 
So for instance, \(S_{\circ}=\iota_\bullet\cdot S_{\circ}\cdot\iota_\circ\), where \(\iota_\bullet\) and \(\iota_\circ\) are the two idempotent algebra elements corresponding to the constant paths at \(\Do\) and \(\Di\), respectively. 
The objects \(\Do\) and \(\Di\) correspond to the crossingless tangles \(\No\) and \(\Ni\), respectively. 
We write 
\[
D\coloneqq D_{\circ}+D_{\bullet}
\quad\text{and}\quad
S\coloneqq S_{\circ}+S_{\bullet}.
\]
The algebra \(\BNAlgH\) carries a  quantum grading \(\QGrad{q}\), which is determined by 
\[
\QGrad{q}(D_{\bullet}) = \QGrad{q}(D_{\circ}) = -2
\qquad 
\text{and}
\qquad
\QGrad{q}(S_{\bullet}) = \QGrad{q}(S_{\circ}) = -1.
\] 
The homological grading \(h\) vanishes on \(\BNAlgH\). 
Differentials of bigraded chain complexes over \(\BNAlgH\) are defined to preserve quantum grading and increase the homological grading by 1. 
More explicitly, if there is an arrow $x\xrightarrow{a} y$ in the differential, then $\QGrad{q}(y) + \QGrad{q}(a) - \QGrad{q}(x)=0,~h(y) - h(x)=1$. 
We indicate these gradings on generators \(\bullet\) by super- and subscripts like so: \(\GGzqh{\Do}{}{q}{h}\).
The bigraded chain homotopy type of \(\DD(T;\field)\) is an invariant of the tangle~\(T\) \cite[Section~4.2]{KWZ}. 

%\begin{figure}[b]
%	\centering
%	\begin{subfigure}{0.37\textwidth}
%		\centering
%		\(\pretzeltangleKh\)
%		\caption{}\label{fig:pretzeltangleKh}
%	\end{subfigure}
%	\begin{subfigure}{0.37\textwidth}
%		\centering
%		\(\pretzeltangleDownstairsBNr\)
%		\caption{}\label{fig:pretzeltangleDownstairsBNr}
%	\end{subfigure}
%	\begin{subfigure}{0.22\textwidth}
%		\centering
%		\(\pretzeltangleUpstairsBNr\)
%		\caption{}\label{fig:pretzeltangleUpstairsBNr}
%	\end{subfigure}
%	\caption{
%		\todo{Kurve für \(\Ni\) hinzufügen. Außerdem ist Figure (b) ist falsch, die Steigung muss -2 sein. }
%		The \((2,-3)\)-pretzel tangle \(P_{2,-3}\) (a), its multicurve invariant \(\textcolor{blue}{\BNr(P_{2,-3})}\) (b) and its lift to the planar cover \(\mathbb{R}^2\smallsetminus\mathbb{Z}^2\) of \(\FourPuncturedSphereKh\) (c).
%	The marked intersection point in~(b) between \(\textcolor{blue}{\BNr(P_{2,-3})}\) and \(\textcolor{red}{a_\infty}=\BNr(Q_\infty)\) is the generator of the tower \(\BNr(U)\cong\mathbb{F}_2[H]\).}\label{fig:BNrExamples}
%\end{figure}

\begin{example}\label{ex:DD}
	The chain complex corresponding to the trivial tangle \(Q_0=\No\) (with any orientation) consists of exactly one object and the differential vanishes:
	\begin{equation*}
	\DD(Q_0)=\left[\GGzqh{\Do}{}{0}{0}\right].
	\end{equation*}
	The invariant of the \((-2,3)\)-pretzel tangle \(T_a\) shown in \cref{fig:T_a} is equal to 
	% for the computation, see
	% khtpp/examples/RasmussenSOfSatellites/patterns/T_a.html
	\begin{equation*}
	\DD(T_a)
	=
	\left[
	\tikzcdset{diagrams={nodes={inner sep=1pt}}}
	\begin{tikzcd}[column sep=21pt]
	\GGzqh{\Di}{}{-2}{-2}
	\arrow{r}{S}
	&
	\GGzqh{\Do}{}{-1}{-1}
	\arrow{r}{D}
	&
	\GGzqh{\Do}{}{1}{0}
	\arrow{r}{S^2}
	&
	\GGzqh{\Do}{}{3}{1}
	\arrow{r}{D}
	&
	\GGzqh{\Do}{}{5}{2}
	\arrow{r}{S}
	&
	\GGzqh{\Di}{}{6}{3}
	&
	\GGzqh{\Di}{}{4}{2}
	\arrow[swap]{l}{D}
	&
	\GGzqh{\Do}{}{3}{1}
	\arrow[swap]{l}{S}
	&
	\GGzqh{\Do}{}{1}{0}
	\arrow[swap]{l}{D}
	\end{tikzcd}
	\right]\!.
	\end{equation*}
	(This and other computations for this paper can be found at \cite{tangle-atlas}.)
	Here, the orientation on \(T_a\) is chosen to be pointing outwards at the punctures on the top left and bottom right; altering this choice only affects the bigrading (see below). 
\end{example}

In the second step, one interprets \(\DD(T;\field)\) geometrically. 
This relies on the following classification result:
Chain homotopy classes of bigraded chain complexes over \(\BNAlgH\) are in one-to-one correspondence with certain geometric objects, called multicurves \cite[Theorem~1.5]{KWZ}. 
%(We will explain what they are in just a moment.)
Using this correspondence between homological algebra and geometry, one now defines the invariant \(\BNr(T;\field)\) as the multicurve corresponding to \(\DD(T;\field)\).

\begin{figure}[b]
	\centering
	\begin{subfigure}{0.32\textwidth}
		\centering
		\includegraphics{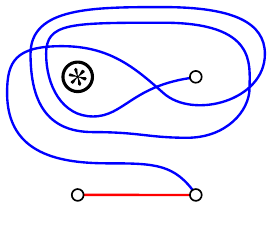}
		\caption{}\label{fig:curves_plain}
	\end{subfigure}	
	\begin{subfigure}{0.32\textwidth}
		\centering
		\includegraphics{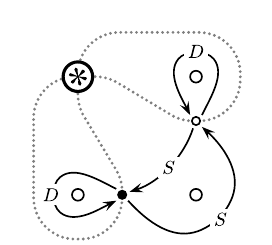}
		\caption{}\label{fig:para:abstract}
	\end{subfigure}
	\begin{subfigure}{0.32\textwidth}
		\centering
		\includegraphics{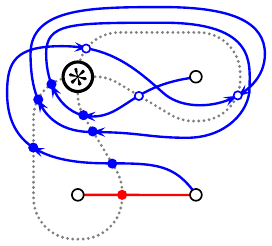}
		\caption{}\label{fig:curves_combined}
	\end{subfigure}
	\caption{%
		(a) Two non-compact curves on \(\FourPuncturedSphereKh\): The short one in red shows \(\BNr_a(\protect\No)\), the long, blue curve is \(\BNr_a(T_a)\); 
		(b) the parametrisation of \(\FourPuncturedSphereKh\) with the embedded quiver from equation~\eqref{eq:B_quiver}; and 
		(c) the curves from (a) half-way translated into chain complexes over \(\BNAlgH\). 
}
	\label{fig:BNrExamples}
\end{figure}

Multicurves are, by definition, collections of certain immersed curves on \(\FourPuncturedSphereKh\), a four-punctured sphere with one distinguished/special puncture. 
The curves also carry a bigrading, but we will ignore this for a moment. 
We distinguish between compact and non-compact curves. 
A non-compact curve in \(\FourPuncturedSphereKh\) is a non-null-homotopic immersion of an interval, with ends on the three non-special punctures of \(\FourPuncturedSphereKh\)  \cite[Definition~1.4]{KWZ}. 
We consider these curves up to homotopy relative boundary. 
For the purpose of this paper, we can completely ignore compact curves, thus avoiding the discussion of local systems. 
Since our notion of Conway tangles does not allow for closed tangle components, we know that for any Conway tangle \(T\), \(\BNr(T;\field)\) contains exactly one non-compact component \cite[Proposition~6.15]{KWZ}.
We denote this component by \(\BNr_a(T;\field)\); it is this invariant that we will focus on below. 

\begin{example}\label{ex:BNr}
	\cref{fig:curves_plain} shows two non-compact immersed curves on the four-punctured sphere \(\FourPuncturedSphereKh\), which is drawn as the plane plus a point at \(\infty\), minus the special puncture \(\ast\) and the three non-special punctures \(\circ\). 
	The short, horizontal curve (drawn in red) shows the invariant \(\BNr(Q_0;\field)=\BNr_a(Q_0;\field)\), which consists of a single component. 
	The other curve (drawn in blue) is equal to the invariant  \(\BNr(T_a;\field)=\BNr_a(T_a;\field)\). 
%	The invariant \(\BNr_a(T_{3,-3};\field)\) for the \((3,-3)\)-pretzel tangle \(T_{3,-3}\) from \cref{fig:T_b} agrees with \(\BNr(Q_\infty;\field)\); however, \(\BNr(T_{3,-3};\field)\) also contains a compact component (which is not shown).
	The invariants of both tangles are independent of the field of coefficients \(\field\); this is not true in general. 
\end{example}

At multiple points in this paper, we need to apply the correspondence between chain complexes and multicurves very explicitly, so we now briefly describe how this works. 
Consider the two dotted grey arcs in \cref{fig:para:abstract}, which start and end on the special puncture of \(\FourPuncturedSphereKh\) and separate the three non-special punctures. 
We call these arcs the \textit{parametrisation of \(\FourPuncturedSphereKh\)}. 
Furthermore, fix an embedding of the quiver \eqref{eq:B_quiver} into \(\FourPuncturedSphereKh\) as shown also in \cref{fig:para:abstract}. 
Now, given an immersed curve \(\gamma\) (say, non-compact), we homotope it such that it intersects the parametrisation of \(\FourPuncturedSphereKh\) minimally. 
Then each intersection point corresponds to a generator of the associated chain complex, namely an object \(\Di\) if the intersection point lies on the grey arc at the top and an object \(\Do\) if it lies on the arc on the left, as prescribed by the vertices of the embedded quiver. 
The paths between those intersection points that meet the parametrisation only in their two endpoints correspond to the differentials of the chain complex. 
More explicitly, suppose \(\alpha\co[a_0,a_1]\rightarrow\FourPuncturedSphereKh\) is a subpath of \(\gamma\) such that the preimage of the parametrisation is equal to \(\{a_0,a_1\}\). Then the chain complex contains an arrow 
\[
\begin{tikzcd}
\alpha(a_0)
\arrow{r}{b_{\alpha}}
&
\alpha(a_1)
\end{tikzcd}
\quad\text{or}\quad
\begin{tikzcd}
\alpha(a_1)
\arrow{r}{b_{\alpha}}
&
\alpha(a_0)
\end{tikzcd}
\]
where the direction of the arrow and \(b_{\alpha}\in\BNAlgH\) are determined by the unique path in the embedded quiver that is homotopic to \(\alpha\) (relative to the parametrisation). 

\begin{example}\label{ex:BNr2DD}
	The curve \(\BNr_a(Q_0)\) from \cref{ex:BNr} intersects the parametrisation in exactly one point \(\Do\). 
	So the corresponding complex consists of a single object \(\Di\) and the differential vanishes. 
	This is precisely the complex \(\DD(Q_0)\) from \cref{ex:DD}. 
	The curve \(\BNr_a(T_a)\) intersects the left arc of the parametrisation three times and the other arc six times. 
	So the corresponding complex contains three generators \(\Di\) and six generators \(\Do\). 
	The directions of the differentials are shown in \cref{fig:curves_combined}. 
	The corresponding complex agrees with \(\DD(T_a)\) from \cref{ex:DD}.
\end{example}

The reverse direction of the correspondence between chain complexes and multicurves, i.e.\ translating algebra into geometry, is generally much more involved and relies on the arrow pushing algorithm from \cite{HRW}. 
Notice, however, that any chain complex represented by a directed graph with only 0-, 1-, or 2-valent vertices whose arrows are labelled alternately by powers of \(S,D\in\BNAlgH\) arises directly from a multicurve in the way described above. 
So in practice, it suffices to find such a representative in the homotopy class of a given chain complex. If the number of objects is small this is usually straightforward. 

\begin{remark}\label{obs:direct_sums}
	The multicurve corresponding to a direct sum of chain complexes is equal to the disjoint union of the multicurves corresponding to the individual summands. 
\end{remark}

%One direction of the correspondence, namely going from multicurves to chain complexes is much easier than the other, so we focus on this. 

%The generators of the type~D structures carry a bigrading, that is a quantum grading \(\QGrad{q}\) and a homological grading \(h\). 

The parametrisation of \(\FourPuncturedSphereKh\) not only plays a key role for translating between chain complexes and multicurves, it is also important for the definition of the bigrading on immersed curves. 
The bigrading takes the form of a \(\Z^2\)-valued function \(\mathrm{gr}=(\QGrad{q},h)\) on the set of intersection points with the parametrisation that is compatible with how these intersection points are connected by the curve, in the sense that the bigrading on the corresponding chain complex be well-defined. 
Thus, the bigrading of any given component of a multicurve is determined by the bigrading of a single generator of this component.

In general, the bigrading on \(\DD(T)\) depends on an orientation of the tangle \(T\). 
If \(T\) is an unoriented tangle, \(\BNr(T)\) only carries \emph{relative} bigradings, i.e.\ bigradings that are well-defined up to an overall shift that is determined by the linking number of the tangle (\cref{def:double_tangle}). 

\begin{proposition}
	If \(T\) is a Conway tangle with linking number \(0\), the (absolute) bigrading on \(\BNr(T)\) is independent of the orientation of \(T\). 
\end{proposition}
\begin{proof}
	This follows directly from \cite[Proposition~4.8]{KWZ}. 
\end{proof}

This proposition applies in particular to the trivial tangle \(\No\) and, more generally, to the double tangle \(T_K\)  for any knot \(K\). 

\subsection{Properties of the multicurve invariant}

The reader may have already noticed the appearance of two conspicuously distinct four-punctured spheres in the construction of the multicurve invariants above:  
the surface \(\FourPuncturedSphereKh\), on which the curves live, and the boundary of the unit 3-ball \(B^3\) with the four marked points \(\{\pm e^{\pm i\pi/4}\}\times\{0\}\), which are the tangle ends~\(\partial T\).
The parametrisation of \(\FourPuncturedSphereKh\) used for translating between chain complexes and multicurves distinguishes one of the four punctures from the other three. 
A choice of a distinguished point \(\ast\) among the four marked points on \(\partial B^3\) is in fact also required in the definition of the homological invariant \(\DD(T;\field)\). 
(We always choose the top left tangle end \((-e^{-i\pi/4},0)\) as the point \(\ast\).)

This coincidence suggests that one should identify \(\FourPuncturedSphereKh\) with \(\partial B^3\smallsetminus \partial T\). 
We do this as follows. 
Assume that the four punctures in \cref{fig:para:abstract} are the points \(\pm e^{\pm i\pi/4}\) on the unit circle in \(\mathbb{C}\subset \mathbb{C}\cup\{\infty\}=S^2\).
We identify this circle with \(S^1\times \{0\}\subset \partial B^3\) via the identity map. 
We then extend this map to the entire spheres such that the interior of the unit disc in \(\mathbb{C}\subset S^2\) is mapped to those points on \(\partial B^3\) with positive \(z\)-coordinate (i.e.\ the front). 
%for an illustration, see \cref{fig:para:with_tangle}. 
Of course, identifying these spheres begs the question if the identification is natural with respect to the action of the mapping class group \(\Mod(\FourPuncturedSphereKh)\): 
Does adding a twist to the tangle ends correspond to twisting the curves?
At least if \(\field=\mathbb{F}_2\), the answer is yes: 
\begin{theorem}[{\cite[Theorem~1.13]{KWZ}}]\label{thm:Kh:Twisting:mod2}
	For all oriented Conway tangles \(T\) and  
	\( 
	\tau\in \Mod(\FourPuncturedSphereKh)
	\), 
	\[
	\BNr(\tau(T;\F_2))
	=
	\tau(\BNr(T;\F_2)).\myqed
	\]
\end{theorem}

We expect that \cref{thm:Kh:Twisting:mod2} generalises to arbitrary fields, but we will not prove it in this article. 
We only need to understand how the mapping class group acts on non-compact curves \(\BNr_a(T;\field)\).
\begin{proposition}\label{thm:Kh:Twisting:arcs}
		For all oriented Conway tangles \(T\), fields \(\field\), and 
	\( 
	\tau\in \Mod(\FourPuncturedSphereKh)
	\), 
	\[
	\BNr_a(\tau(T;\field))
	=
	\tau(\BNr_a(T;\field)).
	\]
\end{proposition}

\begin{proof}[Sketch proof]
	Following the same strategy of the proof of \cref{thm:Kh:Twisting:mod2}, it suffices to consider how single twists act on the multicurves. Using Bar-Natan's gluing formalism for cobordism complexes of tangles, we compute how a single twist acts on the objects and morphisms in \(\BNAlgH\). The action on objects is given by 
	\[ 
	\Di
	\mapsto 
	\left[
	\begin{tikzcd}[column sep=15pt]
	\Di
	&
	\Di
	\arrow{r}{1}
	\arrow[swap]{l}{D}
	&
	\Di
	\end{tikzcd}\right]
	\qquad
	% \]
	% and
	% \[ 
	\Do
	\mapsto 
	\left[
	\begin{tikzcd}[column sep=15pt]
	\Do
	\arrow{r}{D}
	&
	\Di
	\end{tikzcd}\right]
	\]
	with the appropriate grading shifts. The action on morphisms is the same as \cite[Figure~45a]{KWZ}, except that some of the vertical arrows carry minus signs. From this, one can compute the action on type~D structures over \(\BNAlgH\) corresponding to elementary curve segments. 
	Again, the answer differs from \cite[Figure~47]{KWZ} only in a number of minus signs.
	From this we conclude that the underlying curves of \(\BNr(T;\field)\) behave under twisting as expected. Only the local systems may change. Since non-compact components carry no interesting local systems---all minus signs can be ``pushed off the ends'', so to speak---the claim follows. 
\end{proof}

\begin{remark}
	A more careful analysis of the signs in the sketch proof above shows that the local systems of curves corresponding to reduced type~D structures whose differentials only contain linear combinations of \(D\), \(S\), and \(S^2\in\BNAlgH\) do not change under adding twists. 
	However, there do exist tangles \(T\), for which the local systems on the associated multicurves seem to change under twisting, suggesting that the definition of local systems is unnatural. 
%	For example, see the last component of the invariant over \(\F_3\) without any twist (\texttt{options -c3}) and with a single twist (\texttt{options -c3p.x1.}): 
%	\url{https://cbz20.raspberryip.com/code/khtpp/examples/ThinLinksAndConwaySpheres/T_CKMC.html}.
	As will be shown in forthcoming work, in which \cref{thm:Kh:Twisting:mod2} is generalised in full, this is because of the unfortunate choice of the basis 
	\[
	\{\iota_\bullet,\iota_\circ, H^n\cdot S_\bullet, H^n\cdot S_\circ,H^n\cdot S_\circ S_\bullet,H^n\cdot S_\bullet S_\circ,H^n\cdot D_\bullet,H^n\cdot D_\circ\mid n\geq0\}
	\] 
	of the algebra \(\BNAlgH\) used in \cite{KWZ} in the definition of local systems. 
	Replacing \(H\) by \((-H)\) in this basis makes local systems become natural under twisting.
\end{remark}

We note a few more useful facts about the invariant \(\BNr_a(T;\field)\):

\begin{theorem}\label{thm:BNr-contains-a-single-arc}
	For any oriented Conway tangle \(T\), the non-compact curve \(\BNr_a(T;\field)\) connects the two non-special tangle ends that are connected via the tangle \(T\). \myqed
\end{theorem}

\begin{theorem}\label{thm:BNra:mutation_invariance}
	For all oriented Conway tangles \(T\), \(\BNr_a(\muty(T);\field)=\BNr_a(T;\field)\).\myqed
\end{theorem}

\Cref{thm:BNr-contains-a-single-arc} is a special case of \cite[Proposition~6.15]{KWZ}. 
\Cref{thm:BNra:mutation_invariance} is the key observation used in the proof of invariance of the Rasmussen invariant under Conway mutation \cite[Theorem~9.17]{KWZ}.

\begin{corollary}\label{cor:BNra:mutation_invariance}
	For any oriented knot \(K\), \(\BNr_a(T_{K^r};\field)=\BNr_a(T_{K};\field)\). 
\end{corollary}
\begin{proof}
	This follows from \cref{prop:doubletangles:orientation} and \cref{thm:BNra:mutation_invariance}. 
\end{proof}

\begin{definition}
	Given a non-compact curve \(\gamma\), let \(-\gamma\) denote its mirror, i.e.\ the curve obtained by applying the involution that multiplies the \(z\)-coordinate by \((-1)\) and reverses all gradings. 
	Similarly, given a chain complex \(X\) over \(\BNAlgH\), let \(-X\) be the complex obtained by reversing the directions of all differentials and reversing all gradings. 
\end{definition}

The following is \cite[Propositions~4.26 and~7.1]{KWZ}. 

\begin{proposition}
	For any oriented Conway tangle \(T\), 
	\[
	\DD(-T;\field)=-\DD(T;\field)
	\quad
	\text{and}
	\quad
	\BNr(-T;\field)=-\BNr(T;\field).\myqed
	\]
\end{proposition}

\subsection{The pairing theorem for multicurves}

The following is a specialisation of \cite[Proposition~4.31 and Theorem~7.2]{KWZ} to knots.
\begin{theorem}[Pairing Theorem]\label{thm:pairing:BNr}
	Let \(K=T\cup T'\) be an oriented knot where \(T\) and \(T'\) are two oriented Conway tangles. 
	Then \(\BNr(K;\field)\) is bigraded isomorphic to both 
	\begin{equation}\label{eq:MorAndHF}
	h^0q^1\homology\left(
	\Mor\Big(\!-\DD(T;\field),\DD(T';\field)\Big)
	\right)
	\quad
	\text{and}
	\quad
	h^0q^1\HF\Big(\!-\BNr(T;\field),\BNr(T';\field)\Big),
	\end{equation}
	where \(\HF\) denotes the wrapped Lagrangian Floer homology.\myqed
\end{theorem}

We now explain how the two expressions in \eqref{eq:MorAndHF} are defined. 

The morphism space between two chain complexes \(C\) and \(C'\) over \(\BNAlgH\) is a chain complex over~\(\field\). 
Its underlying vector space is generated by homomorphisms from \(C\) to \(C'\). 
These do not necessarily preserve the bigrading nor do they necessarily commute with the differentials. 
The differential \(\partial\) on \(\Mor(C,C')\) is defined by pre- and post-composition with the differentials of \(C\) and \(C'\) in the usual way \cite[Section~2]{KWZ}.

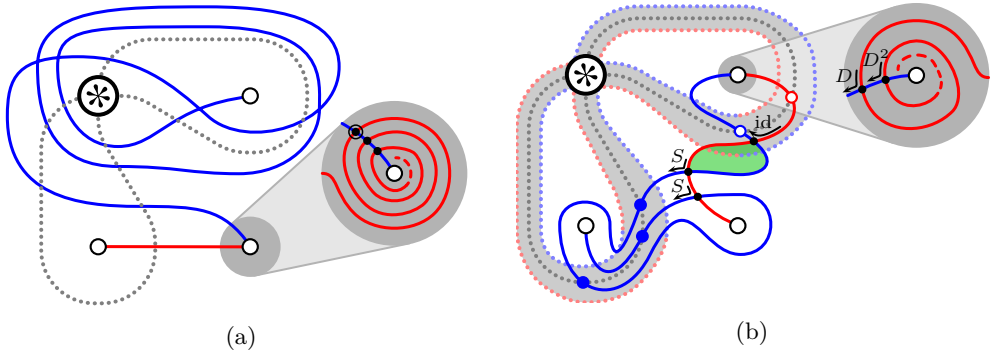
\begin{figure}[b]
	\centering
	\begin{subfigure}{0.45\textwidth}
		\centering
		\includegraphics{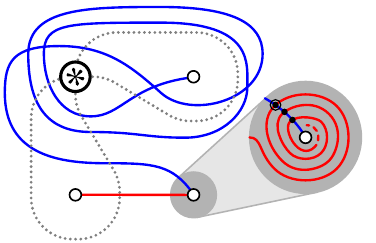}
		\caption{}\label{fig:pairing_plain}
	\end{subfigure}
	\begin{subfigure}{0.45\textwidth}
		\centering
		\labellist \tiny
		\pinlabel $\mathrm{id}$ at 92 68
		\pinlabel $D$ at 122 85
		\pinlabel $D^2$ at 134 92
		\pinlabel $S$ at 60 55
		\pinlabel $S$ at 60 44
		\endlabellist
		\includegraphics{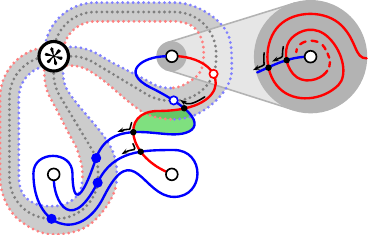}
		\caption{}\label{fig:pairing_position}
	\end{subfigure}
	\caption{%
		(a) The pairing of the curves from \cref{fig:curves_plain};
		the first intersection point in the winding region around the lower right puncture is highlighted;
		(b) Two curves in pairing position; see paragraph before \cref{thm:PairingMorLagrangianFH}. 
		The labelled arrows at the intersection points indicate the algebra elements used in the definition of the quantum grading of the intersections points. 
	}\label{fig:para}
\end{figure}

To define wrapped Lagrangian Floer homology, let \(\g\) and \(\gp\) be two immersed curves on \(\FourPuncturedSphereKh\). 
For simplicity, let us assume they are both non-compact. 
Let us further assume that \(\g\) and \(\gp\) connect different pairs of non-special tangle ends.  
With notation as in \cref{thm:pairing:BNr}, this assumption holds for \(\g=\BNr_a(T)\) and \(\gp=\BNr_a(T')\) by \cref{thm:BNr-contains-a-single-arc}, since \(K\) is a knot. 
So the endpoints of \(\g\) and \(\gp\) only meet in one tangle end; let us denote this tangle end by \(p\). 
By removing bigons between \(\g\) and \(\gp\), we can make sure that they intersect each other in a minimal number of points, like the two arcs in \cref{fig:curves_plain}. 
We call this the \textit{minimal position} for the two curves. 
Now modify the curve \(\g\) in a small neighbourhood of the tangle end \(p\), which we call the \textit{winding region of \(p\)}, such that it wraps around \(p\) infinitely many times in counter-clockwise direction. 
The vector space \(\HF(\g,\gp)\) is now freely generated by the intersection points between \(\g\) and \(\gp\). 

\begin{example}
	\cref{fig:pairing_plain} shows the Lagrangian Floer homology \(\HF(\g,\gp)\) of the curves \(\g=\BNr(Q_0)\) and  \(\gp=\BNr(T_{a})\). 
	 Note that \(T_{a}(0)=U\), so 
	 \(\HF(\g,\gp)\cong\BNr(U)\).
\end{example}

The grading on \(\HF(\g,\gp)\) is defined as follows \cite[Section~7.2]{KWZ}:
Let \(\bullet\) be an intersection point between \(\g\) and \(\gp\). 
Consider a path on \(\g\cup\gp\) that starts at an intersection point \(\aR\) of \(\g\) with the parametrisation of \(\FourPuncturedSphereKh\), turns right at \(\bullet\) and ends on an intersection point \(\bB\) of \(\gp\) with the parametrisation.
We can choose \(\aR\) and \(\bB\) such that the path only meets the parametrisation at the endpoints \(\aR\) and \(\bB\). 
This path corresponds to some algebra element \(s\in\BNAlgH\) as in the translation between curves and their associated complexes. We then define
\[
	\QGrad{q}(\bullet)
	\coloneqq
	\QGrad{q}(\bB)-\QGrad{q}(\aR)+\QGrad{q}(s).
\]
One can show that this grading is independent of any choices. 

\begin{definition}\label{def:s_of_pair}
	For \(\g\), \(\gp\), and \(p\) as above, let \(s(\g,\gp)\)
	denote the quantum grading of the first generator of \(\HF(\g,\gp)\) in the winding region around \(p\). 
\end{definition}

Recall from \cref{sec:BNr:knots}, that the Bar-Natan homology \(\BNr(K;\field)\) of a knot \(K\) is not just a bigraded vector space, but in fact a module over the ring \(\field[H]\).  
If we identify the variable \(H\) with the central algebra element
\[
H\coloneqq D+S^2 = D_{\bullet} + D_{\circ} + S_{\circ}S_{\bullet} + S_{\bullet}S_{\circ} ~ \in ~ \BNAlgH,
\]
any morphism space \(\Mor(C,C')\) can be equipped with an action by \(\field[H]\) in the obvious way and this action respects the differential. 
The first isomorphism from \cref{thm:pairing:BNr} respects these two actions by \(\field[H]\). 

The action of \(\field[H]\) on the Lagrangian Floer homology \(\HF(\g,\gp)\) of two multicurves \(\g\) and \(\gp\) is defined indirectly, 
namely by transferring the action of \(H\) on the homology \(\homology(\Mor(C,C'))\) of the morphism space between the chain complexes \(C\) and \(C'\) associated with \(\g\) and \(\gp\). 
This works as follows (see \cite[Section~5.4]{KWZ}).

\begin{definition}	
	Let \(C\) and \(C'\) be two chain complexes over the algebra \(\BNAlgH\). 
	If \(C\) and \(C'\) are fully cancelled, i.e.\ if their differentials contain no identity component, we can write the morphism space \(\Mor(C,C')\) as a mapping cone
	\[
	(\Mor(C,C'),\partial)
	\cong  
	\left[
	(\Mor^\times(C,C'),0)
	\xrightarrow{~\beta~}
	(\Mor^+(C,C'),\partial^+)
	\right],
	\]
	where \(\Mor^\times(C,C')\) consists of the morphisms that only contain idempotent components, \(\Mor^+(C,C')\) consists of the morphisms that do not contain any idempotent components and \(\beta\) and \(\partial^+\) are the restrictions of \(\partial\). 
	The map \(\beta\) induces a map 
	\[
	\left[
	\Mor^\times(C,C')
	\xrightarrow{~\beta_\ast~}
	\homology(\Mor^+(C,C'),\partial^+)
	\right]
	\]
	If we understand this map and its domain and codomain well enough, this allows us to compute the homology of the morphism space:
	\[
	\homology(\Mor(C,C'),\partial)
	\cong  
	\homology\left[
	\Mor^\times(C,C')
	\xrightarrow{~\beta_\ast~}
	\homology(\Mor^+(C,C'),\partial^+)
	\right].
	\]
	Let \(\g\) and \(\gp\) be the two multicurves corresponding to the type~D structures \(C\) and \(C'\). 
	If one puts these multicurves into a certain position, called the \textit{pairing position} in~\cite[Section~5.4]{KWZ}, we can define the Lagrangian Floer homology of \(\HF(\g,\gp)\) combinatorially. 
	
	The pairing position of two multicurves \(\g\) and \(\gp\) is generally different from the minimal position of \(\g\) and \(\gp\). 
	It is defined as follows. 
	Consider a small tubular neighbourhood of each arc of the parametrisation of \(\FourPuncturedSphereKh\). 
	Each neighbourhood has two boundary components, which we may regard as push-offs of the corresponding arc. 
	Divide each of boundary component into a left and a right half when viewed from the interior of the neighbourhood. 
	We now homotope the two multicurves such that \(\g\) avoids the left halves of these boundary components and \(\gp\) avoids the right halves and such that the curves intersect minimally under these restrictions. 
	(Explicit models for the curves are described in \cite[Section~5.4]{KWZ}.)
	Finally, we modify \(\g\) in the winding region as before. 
	This is illustrated in \cref{fig:pairing_position}.
	
	The intersection points of two multicurves in pairing position can be partitioned into two: 
	Those intersection points that lie in the neighbourhoods of the arcs are called \textit{upper intersection points} and they freely generate a vector space \(\CFtimes(C,C')\);
	the remaining intersection points are called \textit{lower intersection points} and they freely generate the vector space \( \CFplus(C,C') \). 
	By counting bigons connecting upper intersection points to lower ones, such as the bigon shaded green in \cref{fig:pairing_position}, we can define a map
		\[ d\co\CFtimes(C,C')\rightarrow\CFplus(C,C').\]
	We denote its mapping cone by \( \CF(C,C')\) and its homology is \(\HF(\g,\gp)\).
\end{definition}

The reader may have already noticed that each upper intersection point \(\bullet\) corresponds to a pair of intersection points of \(\g\) and \(\gp\) with the arc in whose neighbourhood the point \(\bullet\) lies. 
This identification of \(\CFtimes(C,C')\) and \(\Mor^\times(C,C')\) can be extended to the entirety of the two constructions \cite[Theorem~5.22]{KWZ}:

\begin{theorem}\label{thm:PairingMorLagrangianFH}
	Let \(\g\) and \(\gp\) be two multicurves in pairing position and \(C\) and \(C'\) their associated type~D structures.
	Then there exist bigraded isomorphisms 
	\(\Mor^\times(C,C')\cong \CFtimes(\g,\gp)\)
	and
	\(\Mor^+(C,C')\cong \CFplus(\g,\gp)\)
	such that the following diagram commutes:
	\[
	\begin{tikzcd}
	\Mor^\times(C,C')
	\arrow{r}{\beta_\ast}
	\arrow[leftrightarrow]{d}{\cong}
	&
	\homology(\Mor^+(C,C'),\partial^+)
	\arrow[leftrightarrow]{d}{\cong}
	\\
	\CFtimes(\g,\gp)
	\arrow{r}{d}
	&
	\CFplus(\g,\gp)
	\end{tikzcd}
	\]
	In particular, 
	\(\HF(\g,\gp)\cong \homology(\Mor(C,C'),\partial).\)\myqed
\end{theorem}

The morphism space \(\Mor(C,C')\) between two type~D structures \(C\) and \(C'\) over \(\BNAlgH\) carries a natural action \(H\) which is defined by multiplying a morphism by the central element \(H\in\BNAlgH\). 
Of course, this action commutes with the differential, so we obtain an induced (but not necessarily canonical) action on the mapping cone
\[
\left[
\Mor^\times(C,C')
\xrightarrow{~\beta_\ast~}
\homology(\Mor^+(C,C'),\partial^+)
\right]
\]
and both induce the same action on \(\homology(\Mor(C,C'),\partial)\). 
We can now use the isomorphisms from \cref{thm:PairingMorLagrangianFH} to transfer this action to \(\CF(\g,\gp)\) and ultimately to \(\HF(\g,\gp)\).

We will now assume that \(\g\) and \(\gp\) are two non-compact curves.  
For computing \(\HF(\g,\gp)\), minimal positions of \(\g\) and \(\gp\) are generally more convenient than pairing positions. 
Theorem~5.25 from \cite{KWZ}, or rather its proof (see \cite[Theorem~4.45]{pqMod}), tells us how to pass from one to the other. 
Namely, we can identify each intersection point \(x\) in a minimal position with either a lower intersection point \(x_0\in\CF(\g,\gp)\) or a formal sum \(\sum\varepsilon_i x_i\), where \(\varepsilon_i\in\{\pm1\}\) and \(x_i\) are upper generators that are connected via a chain of bigons. 

In the first case, the action can be easily understood, see \cite[Remark~5.23]{KWZ}, and such an intersection point has infinite order exactly if it lies in the winding region of \(\g\) and~\(\gp\). 

The second case requires a bit more care: Under the isomorphism from \cref{thm:PairingMorLagrangianFH}, \(\sum\varepsilon_i x_i\) corresponds to some element \(\sum\varepsilon_i f_i\), where \(f_i\in\Mor^\times(C,C')\) is a certain morphism with a single idempotent component.
By construction, \(\beta_\ast(\sum\varepsilon_i f_i)=0\), i.e.\ there exists some \(f_+\in\Mor^+(C,C')\) such that \(\partial^+(f_+)=\beta(\sum\varepsilon_i f_i)\).
So the intersection point \(x\in\HF(\g,\gp)\) corresponds to the element \([\sum\varepsilon_i f_i -f_+]\in\homology(\Mor(C,C'))\). 
Thus, by definition, \(H^n.x\in\HF(\g,\gp)\) is the element corresponding to \([\sum\varepsilon_i H^n.f_i -H^n.f_+]\) which we may regard as an element in \(\homology(\Mor^+(C,C'))\). 
If \(n\) is large enough (namely strictly greater than the length of any differential in \(C\) and \(C'\)), this element is non-zero if and only if \(x\) lies in the winding region of \(\g\) and \(\gp\), and in this case, it has infinite order. 

This argument shows: 

\begin{lemma}\label{lem:H-tower}
	Let \(\g\) and \(\gp\) be two non-compact curves on \(\FourPuncturedSphereKh\) connecting different pairs of punctures. 
	Then 
	\[
		s(\g,\gp)=
		\max
		\{
		\QGrad{q}(x)
		\mid
		\text{homogeneous } 
		x\in\HF(\g,\gp)
		\co
		H^n.x\neq0
		\text{ for all }
		n>0 
		\}.\myqed
	\]
\end{lemma}

\begin{proposition}\label{thm:s2s}
	With notation as in \cref{thm:PairingMorLagrangianFH}, 
	\[
	s_c(K)=s(-\BNr_a(T),\BNr_a(T'))+1.
	\]
\end{proposition}

\begin{proof}
	Recall from \cref{defprop:s} that \(s_c(K)\) is defined as the quantum grading of the generator of the free summand \(\F_c[H]\) in any decomposition of \(\BNr(K;\F_c)\) as in \eqref{eq:BNr:decomposition} for any knot \(K\).
	Equivalently, 
	\[
	s_c(K)
	=
	\max
	\{
	\QGrad{q}(x)
	\mid
	\text{homogeneous } 
	x\in\BNr(K;\F_c)
	\co
	H^n.x\neq0
	\text{ for all }
	n>0 
	\}.
	\]
	By the Pairing Theorem, 
	\(\BNr(K;\F_c)\) is bigraded isomorphic to the Lagrangian Floer homology of \(-\BNr(T)\) and \(\BNr(T')\), shifted by 1 in quantum grading. 
	Furthermore, the wrapped Lagrangian Floer homology of two multicurves is equal to the direct sum of the wrapped Lagrangian Floer homologies between their individual components. 
	This identity naturally respects the \(\field[H]\)-module structures. 
	The wrapped Lagrangian Floer homology between two curves of which at least one is compact is finite dimensional as a vector space over \(\field\) and hence, it is \(H\)-torsion. 
	So we may replace \(-\BNr(T)\) and \(\BNr(T')\) by \(-\BNr_a(T)\) and \(\BNr_a(T')\), respectively, without changing the free part of their wrapped Lagrangian Floer homology.
	Now apply \cref{lem:H-tower}.  
\end{proof}

\section{Rasmussen invariants and Conway spheres}\label{sec:RasmussenCurves}

This section forms the technical heart of this paper and contains the proofs of the main results. 

\begin{definition}\label{def:i_p}
	Let \(p\) be a non-special puncture of \(\FourPuncturedSphereKh\).
	We define a cyclic order on the set \(C_p\) of non-compact curves on \(\FourPuncturedSphereKh\) that have exactly one end on \(p\). 
	We do this as follows: 
	Given three curves \(\gamma_1,\gamma_2,\gamma_3\in C_p\) that are pairwise distinct as ungraded curves, we may assume without loss of generality that each pair of curves intersects minimally.  
	(This can always be achieved by homotopies that remove innermost bigons.)
	Then \([\gamma_1,\gamma_2,\gamma_3]\) is cyclically ordered with respect to \(p\), or \textit{\(p\)-ordered} for short, if it induces the counterclockwise cyclic order on the intersection points of the curves with a sufficiently small circle around~\(p\).
	
	Given four non-compact curves \(\gi,\gpi,\gii,\gpii\in C_p\) with \(\textcolor{red}{\gamma_i}\neq\textcolor{blue}{\gamma'_j}\) for all \(i,j\in\{1,2\}\) (as ungraded curves), 
	we define %the quantity \(i_p(\gi,\gpi,\gii,\gpii)\in\{0,1\}\) as follows:
%	If \(\gi\neq\gii\) and \(\gpi\neq\gpii\) (again, as ungraded curves) and \([\gi,\gpi,\gii,\gpii]\) is \(p\)-ordered, we set \(i_p(\gi,\gpi,\gii,\gpii)=1\); otherwise, we set \(i_p(\gi,\gpi,\gii,\gpii)=0\). 
	\[
	i_p(\gi,\gpi,\gii,\gpii)
	\coloneqq
	\begin{cases*}
		1
		&
		\parbox{7.5cm}{if \(\gi\neq\gii\) and \(\gpi\neq\gpii\) (as ungraded curves)\\ and \([\gi,\gpi,\gii,\gpii]\) is \(p\)-ordered}
		\\
		0
		&
		otherwise.
	\end{cases*}
	\]
\end{definition}

For the next result, recall the definition of an expression \(s(\g,\gp)\) from \cref{def:s_of_pair}.

\begin{lemma}[Four Curves Lemma]\label{lem:4curves}
	Let \(p\), \(p'\), and \(p''\) be the three non-special punctures of \(\FourPuncturedSphereKh\).
	Suppose \(\gi\) and \(\gii\) are two curves with ends on the punctures $p$ and \(p'\), and \(\gpi\) and \(\gpii\) two curves with ends on the punctures $p$ and \(p''\). 
	Then
	\begin{equation}\label{eq:4curves}
	s(\gi,\gpi)+s(\gii,\gpii)-2\cdot i_p(\gi,\gpi,\gii,\gpii)
	=
	s(\gi,\gpii)+s(\gii,\gpi)-2\cdot i_p(\gi,\gpii,\gii,\gpi).
	\end{equation}
\end{lemma}

\begin{figure}[t]
	\centering
	\begin{subfigure}{0.45\textwidth}
		\centering
		\labellist \tiny
		\pinlabel $\textcolor{red}{a_{1}}$ at 100 10
		\pinlabel $\textcolor{red}{a_{2}}$ at 87 176
		\pinlabel $\textcolor{blue}{b_{1}}$ at 179 98
		\pinlabel $\textcolor{blue}{b_{2}}$ at 10 88
		\pinlabel $x_{11}$ at 134 97
		\pinlabel $x_{21}$ at 112 97
		\pinlabel $x_{12}$ at 55 89
		\pinlabel $x_{22}$ at 38 89
		\pinlabel $s_{11}$ at 150 85
		\pinlabel $s_{12}$ at 55 110
		\pinlabel $H\!\cdot\!s_{22}$ at 75 110
		\pinlabel $s_{21}$ at 127 85
		\pinlabel $s_{22}$ at 38 110
		\pinlabel $s_\star$ at 140 150
		\endlabellist
		\includegraphics{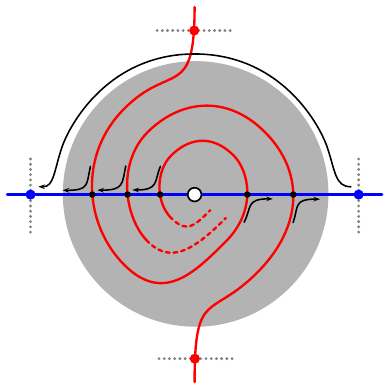}
		\caption{}\label{fig:FourCurvesLemma:1}
	\end{subfigure}	
	\begin{subfigure}{0.45\textwidth}
		\centering
		\labellist \tiny
		\pinlabel $\textcolor{red}{a_{1}}$ at 141 23
		\pinlabel $\textcolor{red}{a_{2}}$ at 57 18
		\pinlabel $\textcolor{blue}{b_{1}}$ at 179 98
		\pinlabel $\textcolor{blue}{b_{2}}$ at 10 88
		\pinlabel $x_{11}$ at 134 97
		\pinlabel $x_{21}$ at 112 97
		\pinlabel $x_{22}$ at 84 89
		\pinlabel $x_{12}$ at 60 89
		\pinlabel $s_{11}$ at 151 85
		\pinlabel $s_{21}$ at 128 85
		\pinlabel $s_{12}$ at 40 101
		\pinlabel $s_{22}$ at 65 101
		\pinlabel $s_\star$ at 140 150
		\endlabellist
		\includegraphics{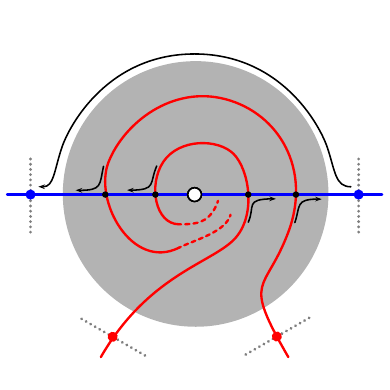}
		\caption{}\label{fig:FourCurvesLemma:2}
	\end{subfigure}
	\caption{An illustration of the curves in the winding region around the puncture \(p\) in the two cases a and b in the proof of \cref{lem:4curves}}\label{fig:FourCurvesLemma}
\end{figure}
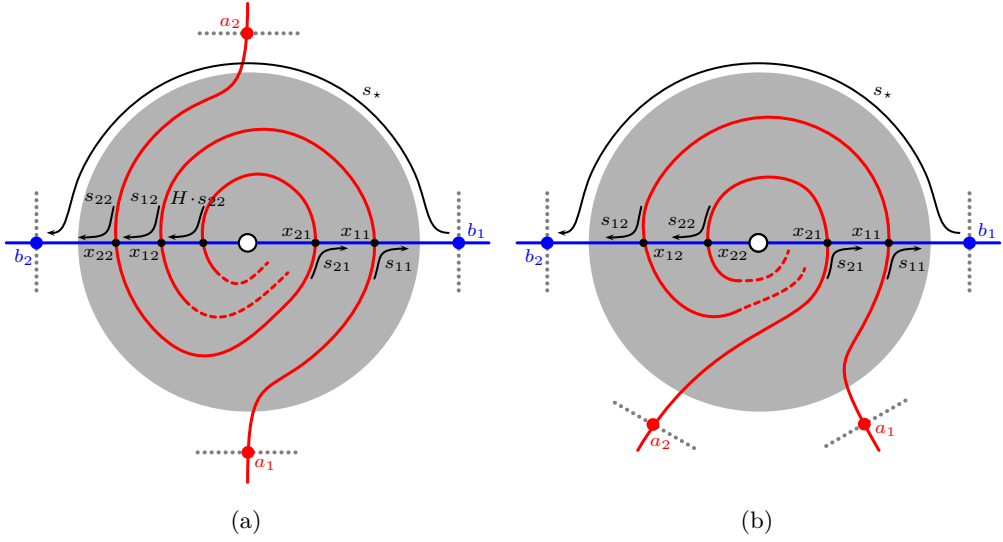

Note that the condition on the punctures at which the curves end ensures that the quantities \(i_p(\gi,\gpi,\gii,\gpii)\) and \(i_p(\gi,\gpii,\gii,\gpi)\) are well-defined. 
Note also that at most one of them is non-zero. 

\begin{proof}[Proof of \cref{lem:4curves}]
	If \(\gi=\gii\) or \(\gpi=\gpii\) as ungraded curves, the statement follows from the gluing theorem (\cref{thm:pairing:BNr}) combined with the fact that for any two curves \(\g\) and \(\gp\),
	\[
	\HF(\QGrad{q}^{x}\g,\QGrad{q}^{y}\gp)=\QGrad{q}^{y-x}\HF(\g,\gp),
	\]
	as for morphism spaces \cite[Definition~2.2]{KWZ}.
	Since the punctures \(p'\) and \(p''\) are distinct, this leaves the case that all four curves are distinct as ungraded curves. 
	We introduce some notation. 
	For \(i,j\in\{1,2\}\), let \(x_{ij}\) be the first generator in the winding region around \(p\) for the pair \((\textcolor{red}{\gamma_i},\textcolor{blue}{\gamma'_j})\). 
	Then by definition of the quantum grading,
	\[
	s(\textcolor{red}{\gamma_i},\textcolor{blue}{\gamma'_j})
	=
	\QGrad{q}(x_{ij})
	=
	\QGrad{q}(b_{j})-\QGrad{q}(a_{i})+\QGrad{q}(s_{ij})
	\]
	where \(b_j\) is the first generator near the end \(p\) of \(\textcolor{blue}{\gamma'_j}\), \(a_i\) is the first generator near the end \(p\) of \(\textcolor{red}{\gamma_i}\), and \(s_{ij}\in\BNAlgH\) is the algebra element that labels the morphism corresponding to \(x_{ij}\). 
	Since equation \eqref{eq:4curves} is symmetric in \(\gi\) and \(\gii\) as well as \(\gpi\) and \(\gpii\), it suffices to consider the following two subcases:
	\newcommand{\myitheading}[1]{\medskip\noindent\textit{#1.}}%
	
	\myitheading{Case (a): \([\gi,\gpi,\gii,\gpii]\) is \(p\)-ordered} 
	In this case, equation \eqref{eq:4curves} is equivalent to
	\[
	\QGrad{q}(s_{11})+\QGrad{q}(s_{22})-2
	=
	\QGrad{q}(s_{12})+\QGrad{q}(s_{21}).
	\] 
	Let \(s_\star\in\BNAlgH\) correspond to the path of minimal length from \(b_1\) to \(b_2\).
	Then
	\[
	\QGrad{q}(s_{12})-\QGrad{q}(s_{11})
	=
	\QGrad{q}(s_\star)
	=
	\QGrad{q}(H\cdot s_{22})-\QGrad{q}(s_{21})
	=
	\QGrad{q}(s_{22})-2-\QGrad{q}(s_{21})
	\]
	as	\cref{fig:FourCurvesLemma:1} illustrates. 
	This is equivalent to the desired identity. 
	
	\myitheading{Case (b): \([\gi,\gpi,\gpii,\gii]\) is \(p\)-ordered}
	In this case, equation \eqref{eq:4curves} is equivalent to
	\[
	\QGrad{q}(s_{11})+\QGrad{q}(s_{22})
	=
	\QGrad{q}(s_{12})+\QGrad{q}(s_{21}).
	\]
	This can be seen in the same way as the identity in Case~(a) by comparing the algebra elements \(s_{ij}\) to the algebra element \(s_\star\) corresponding to the path of minimal length from \(b_1\) to \(b_2\):
	\[
	\QGrad{q}(s_{12})-\QGrad{q}(s_{11})
	=
	\QGrad{q}(s_\star)
	=
	\QGrad{q}(s_{22})-\QGrad{q}(s_{21}).
	\]
	This case is illustrated in \cref{fig:FourCurvesLemma:2}.
\end{proof}

Often, it is useful to lift immersed curves to a covering space of \(\FourPuncturedSphereKh\), namely the plane \(\mathbb{R}^2\) minus the integer lattice \(\mathbb{Z}^2\). 
We may regard \(\mathbb{R}^2\) as the universal cover of the torus, and the torus as the two-fold branched cover of the sphere branched at four marked points; 
then the integer lattice \(\mathbb{Z}^2\) is the preimage of the branch set.
We choose the covering map such that the left arc of the parametrisation of \(\FourPuncturedSphereKh\) (the one corresponding to \(\Do\) in \cref{fig:para:abstract}) is lifted to straight vertical lines and the other arc (\(\Di\)) is lifted to straight horizontal lines. 
This covering space along with the lifts of \(\BNr_a(T)\) of a selection of tangles \(T\) is illustrated in \cref{fig:lifts:loose}.

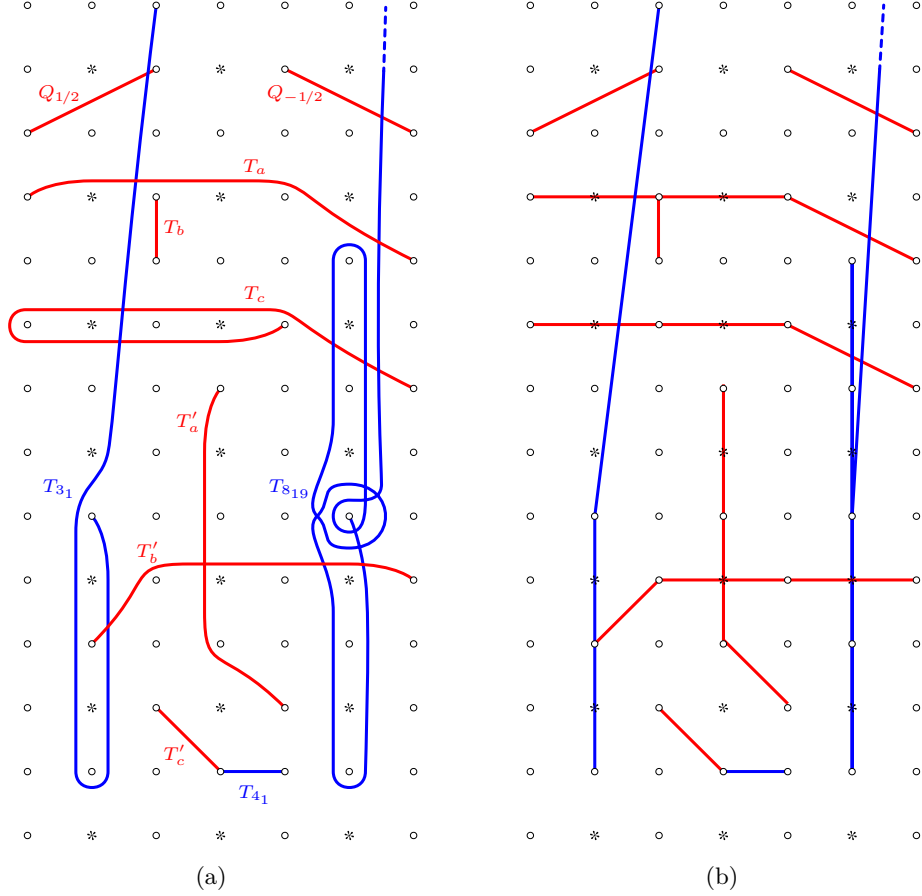
\begin{figure}[t]
		\begin{subfigure}{0.45\textwidth}
		\centering
		\labellist \tiny
		\pinlabel $\textcolor{blue}{T_{3_1}}$ at 22 135
		\pinlabel $\textcolor{blue}{T_{4_1}}$ at 95 20
		\pinlabel $\textcolor{blue}{T_{8_{19}}}$ at 108 135
		\pinlabel $\textcolor{red}{Q_{\nicefrac{1}{2}}}$ at 22 283
		\pinlabel $\textcolor{red}{Q_{\nicefrac{-1}{2}}}$ at 110 283
		\pinlabel $\textcolor{red}{T_a}$ at 95 256
		\pinlabel $\textcolor{red}{T_b}$ at 65 233
		\pinlabel $\textcolor{red}{T_c}$ at 95 208
		\pinlabel $\textcolor{red}{T'_a}$ at 70 160
		\pinlabel $\textcolor{red}{T'_b}$ at 55 111
		\pinlabel $\textcolor{red}{T'_c}$ at 65 35
		\endlabellist
		\includegraphics{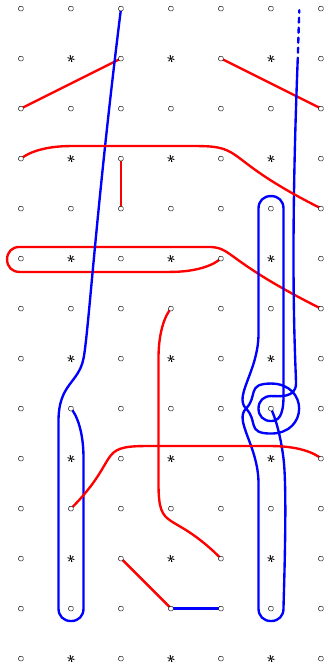}
		\caption{}\label{fig:lifts:loose}
	\end{subfigure}	
	\begin{subfigure}{0.45\textwidth}
		\centering
		\includegraphics{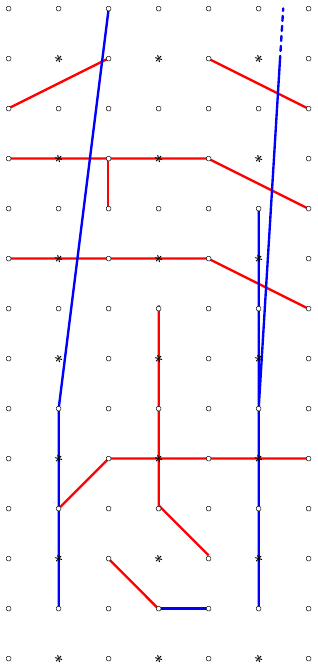}
		\caption{}\label{fig:lifts:pulled-tight}
	\end{subfigure}
	\caption{(a) The lifts of the curves \(\BNr_a(T)\) for some companion tangles (blue) and the pattern tangles (red) from \cref{fig:mainthm,fig:pattern_tangles:winding2}, and (b)~their singular peg-board representatives. 
	\(3_{1} = T_{2,3}\) and \(8_{19} = T_{3,4}\) denote the positive torus knots. 
	The slope of the finite slope segment of the curve in (b) for \(8_{19}\) is 16. 
	The invariants of \(T'_d\) and \(T'_e\) agree with the mirror of the invariant of \(T_{3_1}\) up to single vertical half-twists. 
	}\label{fig:lifts}
\end{figure}

Recall the following ``normal form'' for curves in~\(\FourPuncturedSphereKh\) from \cite[Section~2.3]{KWZ}; see also \cite[Section~7.1]{HRW}.

\begin{definition}\label{def:peg_board_rep}	
	Consider the standard Riemannian metric on \(\PuncturedPlane\), which induces a Riemannian metric on \(\FourPuncturedSphereKh\). 
	Fix some \(\varepsilon\) with \(0<\varepsilon<\nicefrac{1}{2}\). 
	An \textit{\(\varepsilon\)-peg-board representative} of a non-compact curve  \(\gamma\) in~\(\FourPuncturedSphereKh\) is a representative of the homotopy class of \(\gamma\) which has minimal length among all representatives that (away from the two ends) have distance \(\varepsilon\) to all four punctures in~\(\FourPuncturedSphereKh\). 
\end{definition}

The intuition behind this definition is to think of the four punctures of \(\FourPuncturedSphereKh\) as pegs of radii \(\varepsilon\) and then to imagine pulling the curve \(\gamma\) ``tight'', like a rubber band. 
This is illustrated in \cref{fig:lifts:pulled-tight}.

\begin{definition}
	Given \(\gamma\in C_p\), consider the lifts of \(\varepsilon\)-peg-board representatives of \(\gamma\), one for each \(\varepsilon\) with \(0<\varepsilon<\nicefrac{1}{2}\), all with respect to the same lift of the basepoint \(p\).
	We define a \textit{singular peg-board representative of~\(\gamma\)} as the piecewise linear curve that is the limit of these lifts as \(\varepsilon\rightarrow0\). 
	Since singular peg-board representatives of~\(\gamma\) are unique up to deck transformations, this allows us to define the \textit{slope \(\sigma_p(\gamma)\) of~\(\gamma\)} at~\(p\) as the slope of the singular peg-board representative of~\(\gamma\) near the end \(p\). 
	The cyclic order on \(C_p\) (\cref{def:i_p}) induces a strict total order on the curves of the same slope at~\(p\); we denote this order by \(<_p\). 
\end{definition}

\begin{definition}
	For any slope \(\nicefrac{p}{q}\), we write \(\arc{\nicefrac{p}{q}}\) for the curve \(\BNr(Q_{\nicefrac{p}{q}};\field)\). 
\end{definition}

For \(\varepsilon>0\) sufficiently small, the lift of any \(\varepsilon\)-peg-board representative of \(\arc{\nicefrac{p}{q}}\) is a straight line segment of slope \(\nicefrac{p}{q}\) connecting two non-special punctures. 
In particular, the curve \(\arc{\nicefrac{p}{q}}\) is independent of the field \(\field\). 
For \(\nicefrac{p}{q}\in\{0,\infty\}\), the absolute bigrading is well-defined. 
Otherwise, whenever necessary, we will specify the absolute bigrading on \(\arc{\nicefrac{p}{q}}\) by specifying the orientation of~\(Q_{\nicefrac{p}{q}}\). 

%\begin{definition}
%	Let \(T\) be a Conway tangle with \(\conn{T}=\Ni\) or \(\ConnectivityX\). 
%	We define 
%	\(\theta_c(T)\coloneqq\theta_c(P^T)\), where \(P^T\) is the pattern of the tangle \(T\) from \cref{def:patterntangle}. 
%\end{definition}

We now relate the invariants \(\theta_c\) from \cref{def:theta,def:theta:winding2} to the slopes of the multicurve invariants \(\BNr_a\).

\begin{proposition}\label{prop:eta-slopes}
	Let \(P\) be a pattern of wrapping number \(2\) and \(T_P\) a pattern tangle for \(P\) (see \cref{def:patterntangle}). 
%	Let \(T\) be a Conway tangle with \(\conn{T}=\Ni\) or \(\ConnectivityX\). 
	Let \(\sigma\) be the slope of \(\gb\coloneqq\BNr_a(T_P;\F_c)\) near 
	\[
	\begin{cases*}
	\text{the lower right tangle end}
	&
	if \(P\) has winding number 0 (i.e.\ \(\conn{T_P}=\Ni\));
	\\
	\text{the lower left tangle end}
	&
	if \(P\) has winding number ±2 (i.e.\ \(\conn{T_P}=\ConnectivityX\)). 
	\end{cases*}
	\]
	Then \(\theta_c(P)=\infty\) if and only if \(\sigma=\infty\). Moreover, 
	\begin{equation}\label{eq:eta-slopes}
	\theta_c(P)=
	\begin{cases*}
	\lceil\nicefrac{\sigma}{2}\rceil 
	&
	if \(\sigma\in\Q\smallsetminus 2\Z\) 
	\\
	\nicefrac{\sigma}{2} &
	if \(\sigma\in2\Z\) and \(\gb<_p\arc{\sigma}\)
	\\
	\nicefrac{\sigma}{2}+1 &
	if \(\sigma\in2\Z\) and \(\gb>_p\arc{\sigma}\)
	\end{cases*}
	\end{equation}
\end{proposition}

\begin{proof}
	Let us first consider the winding number 0 case. 
	For \(n,m\in\Z\), let \(\arcr{2n}=\BNr(Q_{2n})\), and \(\arcr{2m}=\BNr(Q_{2m})\), where \(Q_{2n}\) and \(Q_{2m}\) are oriented compatibly with the \(\infty\)-rational filling. 
	Let \(p\) be the lower right tangle end. 
	Then by the Four Curves Lemma,
	\begin{align*}
	& s(\arcr{2n},\gb)+s(\arcr{2m},\arcb{\infty})
	-2\cdot i_p(\arcr{2n},\gb,\arcr{2m},\arcb{\infty})
	\\
	=~ &
	s(\arcr{2n},\arcb{\infty})+s(\arcr{2m},\gb)
	-2\cdot i_p(\arcr{2n},\arcb{\infty},\arcr{2m},\gb).
	\end{align*}
	By \cref{thm:s2s}, \(s_c(\arcr{2n},\arcb{\infty})=s_c(U)-1=s_c(\arcr{2m},\arcb{\infty})\). 
	Similarly, \(s_c(\arcr{2n},\gb)=s_c(T_P(2n))-1\) and \(s_c(\arcr{2m},\gb)=s_c(T_P(2m))-1\).
	Recall from \cref{lem:basic:satellites}\ref{lem:basic:satellites:pattern_of_unknot} that \(T_P(2n)=P_n(U)\) and \(T_P(2m)=P_m(U)\).
	Thus, 
	\[
	s_c(P_n(U))
	=
	s_c(P_m(U))
	+2\cdot i_p(\arcr{2n},\gb,\arcr{2m},\arcb{\infty})
	-2\cdot i_p(\arcr{2n},\arcb{\infty},\arcr{2m},\gb).
	\] 
	If \(\sigma=\infty\), then clearly \(i_p(\arcr{2n},\gb,\arcr{2m},\arcb{\infty})=i_p(\arcr{2n},\arcb{\infty},\arcr{2m},\gb)=0\), and so \(\theta_c(P)=\infty\). 
	If \(\sigma\neq\infty\), let \(m\) be equal to the right hand side of \eqref{eq:eta-slopes}. 
	Then \(i_p(\arcr{2n},\arcb{\infty},\arcr{2m},\gb)=0\) for all \(n\). 
	Moreover, \(i_p(\arcr{2n},\gb,\arcr{2m},\arcb{\infty})=1\) if and only if $n<m$; so
	\[
	s_c(P_n(U))
	=
	\begin{cases*}
	s_c(P_m(U))&
	if \(n\geq m\)
	\\
	s_c(P_m(U))+2
	&
	if \(n<m\),
	\end{cases*}
	\] 
	and hence \(\theta_c(P)=m\). 
	
	Let us now consider the winding number ±2 case. 
	As before, let \(n,m\in\Z\) and \(\gb=\BNr_a(T_P;\F_c)\), \(\arcr{2n}=\BNr(Q_{2n})\), and \(\arcr{2m}=\BNr(Q_{2m})\), but where now \(Q_{2n}\) and \(Q_{2m}\) are oriented compatibly with the \(+1\)-rational filling. 
	Furthermore, let \(N\gg0\) be some large odd integer and \(\arcb{N}=\BNr(Q_{N})\),\ where \(Q_N\) is oriented compatibly with the \(0\)-rational filling.
	Let \(p\) be the lower left tangle end. 
	Then by the Four Curves Lemma,
	\begin{align*}
	& s(\arcr{2n},\gb)+s(\arcr{2m},\arcb{N})
	-2\cdot i_p(\arcr{2n},\gb,\arcr{2m},\arcb{N})
	\\
	=~ &
	s(\arcr{2n},\arcb{N})+s(\arcr{2m},\gb)
	-2\cdot i_p(\arcr{2n},\arcb{N},\arcr{2m},\gb).
	\end{align*}
	By \cref{thm:s2s}, \(s_c(\arcr{2n},\arcb{N})=s_c(Q_N(2n))-1=s_c(T_{2,2n-N})-1\) and similarly, \(s_c(\arcr{2m},\arcb{N})=s_c(T_{2,2m-N})-1\). Since \(N\gg 0\), we may assume that \(2n-N\) and \(2m-N\) are both negative. Therefore,
	\[
		s_c(T_{2,2n-N})-s_c(T_{2,2m-N})
		=2(n-m).
	\]
	Also, \(s_c(\arcr{2n},\gb)=s_c(P_n(U))-1\) and \(s_c(\arcr{2m},\gb)=s_c(P_m(U))-1\).
	Thus, 
	\[
	s_c(P_n(U))
	=
	s_c(P_m(U))
	+2(n-m)
	+2\cdot i_p(\arcr{2n},\gb,\arcr{2m},\arcb{N})
	-2\cdot i_p(\arcr{2n},\arcb{N},\arcr{2m},\gb).
	\] 
	With \(\arcb{N}\) playing the role of \(\arcb{\infty}\), we may now argue in the same way as in the winding number 0 case.
\end{proof}

Companion tangles \(T_K\) have rather restricted multicurve invariants since they are cap-trivial:

\begin{lemma}[{cp.\ \cite[Section~3]{KWZ_strong_inversions}}]\label{lem:slopes-for-cap-trivial}
	Let \(T\) be a cap-trivial tangle with \(\conn{T}=\No\). 
	Then the slope of \(\BNr_a(T;\F_c)\) near the lower left tangle end is either infinite or an even integer \(n\in2\Z\), in which case \(\BNr_a(T_K;\F_c)\) is equal to \(\BNr(Q_{n};\F_c)\) up to some grading shift.  
	The slope of \(\BNr_a(T;\F_c)\) near the lower right tangle end is always an integer. 
	Moreover, if \(c=2\), this integer is even. 
\end{lemma}

\begin{proof}
	By assumption, \(T(\infty)\) is the unknot \(U\). 
	So by the Gluing Theorem, 
	\[
	\HF\Big(\arcr{\infty},\BNr(T;\F_c)\Big)
	\cong
	\BNr(U)
	\cong
	\F_c[H].
	\]
	Therefore, \(\BNr(T;\F_c)\)---and in particular \(\BNr_a(T;\F_c)\)---can be homotoped such that it avoids the vertical arc \(\arcr{\infty}\) altogether. 
	Together with the assumption about the connectivity of \(T\) and \cref{thm:BNr-contains-a-single-arc}, this implies the first two statements.
	The final statement follows from the fact that the curve \(\BNr_a(T;\F_2)\) does not wrap around the special puncture \cite[Theorem~4.1]{KWZ_linear}. 
\end{proof}

\subsection{Patterns with winding number 0}

The following proposition is the basis for all results in this subsection. 

\begin{proposition}\label{prop:main:Ni}
	Let \(T\) be a Conway tangle with \(\conn{T}=\Ni\) and let \(T'\) be a cap-trivial tangle. 
	Let \(p\) be the lower right tangle end and \(n\in\Z\). 
	Let \(\g=\BNr_a(-T;\F_c)\) and \(\gp=\BNr_a(T';\F_c)\).  
	Then for all characteristics \(c\),
	\[
	s_c(T\cup T')
	=
	s_c(T(-2n))
	+
	\begin{cases*}
	2 &
	if \([\arcb{2n},\g,\gp,\arcr{\infty}]\) is \(p\)-ordered
	\\
	-2 &
	if \([\gp,\g,\arcb{2n},\arcr{\infty}]\) is \(p\)-ordered
	\\
	0 & otherwise
	\end{cases*}
	\]
\end{proposition}

\begin{proof}
	 Let \(\arcb{2n}=\BNr(Q_{2n};\F_c)\), where \(Q_{2n}\) is oriented compatibly with the \(\infty\)-rational closure. 
	Then by the Four Curves Lemma,
	\begin{align*}
	& s(\g,\gp)+s(\arcr{\infty},\arcb{2n})
	-2\cdot i_p(\g,\gp,\arcr{\infty},\arcb{2n})
	\\
	=~ &
	s(\g,\arcb{2n})+s(\arcr{\infty},\gp)
	-2\cdot i_p(\g,\arcb{2n},\arcr{\infty},\gp).
	\end{align*}
	By \cref{thm:s2s}, \(s(\arcr{\infty},\arcb{2n})=s_c(Q_{2n}(\infty))-1=s_c(U)-1=-1\) and similarly, \(s(\arcr{\infty},\gp)=-1\), since \(T'\) is a cap-trivial tangle. 
	Moreover, \(s(\g,\gp)=s_c(T\cup T')-1\) and \(s(\g,\arcb{2n})=s_c(T(-2n))-1\). 
	Therefore, 
	\[
	s_c(T\cup T')
	=
	s_c(T(-2n))
	+2\cdot i_p(\g,\gp,\arcr{\infty},\arcb{2n})
	-2\cdot i_p(\g,\arcb{2n},\arcr{\infty},\gp).
	\]
	This is equivalent to the desired identity. 
\end{proof}

\begin{corollary}
	Let \(T\) be a Conway tangle with \(\conn{T}=\Ni\) such that \(\theta_c(P^T)=\infty\) for the associated pattern \(P^T\). 
	Then for any cap-trivial tangle  \(T'\), 
	\[
	s_c(T\cup T')
	=
	s_c(T(2t))
	\quad
	\text{for all }t\in\Z.
	\]
\end{corollary}

\begin{proof}
	By \cref{prop:eta-slopes}, the slope of \(\BNr_a(T;\F_c)\) near the lower right tangle end is \(\infty\). 
	The same is true for \(\BNr_a(-T;\F_c)=-\BNr_a(T;\F_c)\). 
	On the other hand, the slope of \(\BNr_a(T;\F_c)\) is finite by \cref{lem:slopes-for-cap-trivial}. 
	So the claim follows from \cref{prop:main:Ni}. 
\end{proof}

\begin{corollary}\label{cor:main:Ni}
	Let \(T\) be a Conway tangle with \(\conn{T}=\Ni\) and let \(T'\) be a cap-trivial tangle. 
	Let \(\sigma'_2\) be the slope of \(\BNr_a(T';\F_2)\) near the lower right tangle end. 
%	Note that \(\sigma'_2\in\Z\) by \cref{lem:slope-for-cap-trivial}. 
	Then
	\[
	s_2(T\cup T')
	=
	s_2(T(-\sigma'_2)).
	\]
\end{corollary}

\begin{proof}	
	By \cref{lem:slopes-for-cap-trivial}, \(\sigma'_2\) is an even integer. 
	Setting \(n=\nicefrac{\sigma'_2}{2}\) in \cref{prop:main:Ni}, we obtain the desired identity, unless \([\arcb{\sigma'_2},\g,\gp,\arcr{\infty}]\) or \([\gp,\g,\arcb{\sigma'_2},\arcr{\infty}]\) are \(p\)-ordered. 
	The following lemma says that this is impossible. 
\end{proof}

\begin{lemma}\label{lem:no-trapped-curves}
	Let \(\g=\BNr_a(-T;\F_2)\) for some tangle \(T\) and \(\gp=\BNr_a(T';\F_2)\) for some cap-trivial tangle \(T'\). 
	Suppose \(\g\) and \(\gp\) have the same slope \(\sigma\in2\Z\) near the lower right tangle end \(p\). 
	Suppose \(\arcb{\sigma}<_p\g<_p\gp\) or \(\gp<_p\g<_p\arcb{\sigma}\). Then \(\conn{T}=\No\). 
\end{lemma}

\begin{figure}[t]
	\begin{subfigure}{0.3\textwidth}
		\centering
		\labellist \tiny
		\pinlabel $S$ at 31 65
		\pinlabel $S$ at 52 65
		\pinlabel $S$ at 31 87
		\pinlabel $S$ at 52 87
		\pinlabel $D$ at 41.2 63
		\pinlabel $D$ at 41.2 89
		\pinlabel $D$ at 28 76
		\pinlabel $D$ at 55 76
		\endlabellist
		\includegraphics{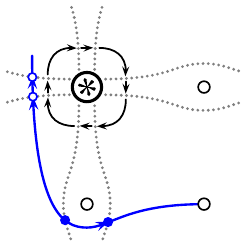}
		\caption{}
	\end{subfigure}	
	\begin{subfigure}{0.3\textwidth}
		\centering
		\labellist \tiny
		\pinlabel $S$ at 31 65
		\pinlabel $S$ at 52 65
		\pinlabel $S$ at 31 87
		\pinlabel $S$ at 52 87
		\pinlabel $D$ at 41.2 63
		\pinlabel $D$ at 41.2 89
		\pinlabel $D$ at 28 76
		\pinlabel $D$ at 55 76
		\endlabellist
		\includegraphics{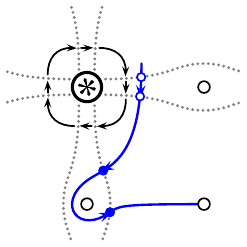}
		\caption{}
	\end{subfigure}
	\begin{subfigure}{0.3\textwidth}
		\centering
		\labellist \tiny
		\pinlabel $S$ at 31 65
		\pinlabel $S$ at 52 65
		\pinlabel $S$ at 31 87
		\pinlabel $S$ at 52 87
		\pinlabel $D$ at 41.2 63
		\pinlabel $D$ at 41.2 89
		\pinlabel $D$ at 28 76
		\pinlabel $D$ at 55 76
		\endlabellist
		\includegraphics{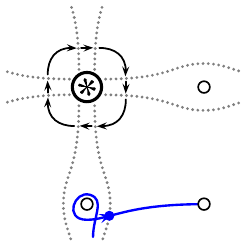}
		\caption{}
	\end{subfigure}
	\caption{An illustration for the proof of \cref{lem:no-trapped-curves}: lifts of curves \(\gp\) to \(\PuncturedPlane\)}%
	\label{fig:slope-for-cap-trivial}
\end{figure}

\begin{proof}
	Since by \cref{thm:Kh:Twisting:mod2} the multicurve invariant \(\BNr\) is natural with respect to twisting, we may assume that \(\sigma=0\) after replacing \(T\) by \(T+Q_{-\sigma}\). Note that this preserves the connectivity of \(T\). Then \(\arcb{\sigma}=\arcb{0}\) is represented by the chain complex consisting of a single object \(\Do\).
	Up to mirroring, we may also assume \(\arcb{0}<_p\g<_p\gp\).
	Then in particular \(\arcb{0}<_p\gp\), so the complex representing \(\gp\) near \(p\) has the following form:
	\[
	\begin{tikzcd}[column sep=40pt]
	\Do
	\arrow[leftarrow]{r}{D_\bullet}
	&
	\Do
	&
	\cdots
	\arrow[leftarrow,swap]{l}{S}
	\end{tikzcd}
	\quad
	\text{or}
	\quad
	\begin{tikzcd}[column sep=40pt]
	\Do
	\arrow[leftarrow]{r}{D_\bullet^j}
	&
	\Do
	&
	\cdots
	\arrow[leftrightarrow,swap]{l}{S^i}
	\end{tikzcd}
	\]
	for some integers \(i\geq1\) and \(j\geq2\); see \cref{fig:slope-for-cap-trivial}.
	Since \(\arcb{0}<_p\g<_p\gp\), the same is true for the  chain complex $X$ representing \(\g\). 
	Taking the mapping cone with \(H\cdot\id_X\) produces the following complex \(X'\) given by the solid arrows:
	\[
	\begin{tikzcd}[column sep=40pt,row sep=30pt]
	\Do
	\arrow[leftarrow]{r}{D_\bullet}
	\arrow{d}{H}
	&
	\Do
	\arrow{d}{H}
	&
	\cdots
	\arrow[dotted,bend right=30,swap]{ll}{S}
	\arrow[leftarrow,swap]{l}{S}
	\\
	\Do
	\arrow[leftarrow]{r}{D_\bullet}
	&
	\Do
	\arrow[dotted,swap]{lu}{\iota_\bullet}
	&
	\cdots
	\arrow[dotted,bend left=30]{ll}{S}
	\arrow[leftarrow,swap]{l}{S}
	\end{tikzcd}
	\quad
	\text{or}
	\quad
	\begin{tikzcd}[column sep=40pt,row sep=30pt]
	\Do
	\arrow[leftarrow]{r}{D_\bullet^j}
	\arrow{d}{H}
	&
	\Do
	\arrow{d}{H}
	&
	\cdots
	\arrow[leftrightarrow,swap]{l}{S^i}
	\\
	\Do
	\arrow[leftarrow]{r}{D_\bullet^j}
	&
	\Do
	\arrow[dotted,swap]{lu}{D_\bullet^{j-1}}
	&
	\cdots
	\arrow[leftrightarrow,swap]{l}{S^i}
	\end{tikzcd} 
	\]
	Applying the Clean-Up Lemma \cite[Lemma~2.17]{KWZ} to the dotted arrows changes the complex as follows: 
	\[
	\begin{tikzcd}[column sep=40pt,row sep=30pt]
	\Do
	\arrow{d}{H}
	&
	\Do
	\arrow{d}{H}
	&
	\cdots
	\arrow[leftarrow,swap]{l}{S}
	\\
	\Do
	&
	\Do
	&
	\cdots
	\arrow[leftarrow,swap]{l}{S}
	\end{tikzcd}
	\quad
	\text{or}
	\quad
	\begin{tikzcd}[column sep=40pt,row sep=30pt]
	\Do
	\arrow{d}{H}
	&
	\Do
	\arrow{d}{H}
	&
	\cdots
	\arrow[leftrightarrow,swap]{l}{S^i}
	\\
	\Do
	&
	\Do
	&
	\cdots
	\arrow[leftrightarrow,swap]{l}{S^i}
	\end{tikzcd} 
	\]
	(The rest of the complex remains unchanged.)
	In particular, we have split off a component of the form
	\[
	\begin{tikzcd}
	\Do
	\arrow{r}{H}
	&
	\Do
	\end{tikzcd}
	\]
	which represents a rational curve \(\mathbf{r}_1(0)\) of length 1 and slope 0 \cite[Section~2]{KWZ_linear}. Since \(X'\) is a component of \(\Khr(T)\), so is the component \(\mathbf{r}_1(0)\). Rational curves of odd length detect connectivity by \cite[Theorem~5.6]{KWZ_linear}, so \(\conn{T}=\No\).
%	The argument for the case \(\gp<_p\g<_p\arcb{\sigma}\) is identical to the above up to reversing all arrows.
\end{proof}

	We expect the proof of \cref{lem:no-trapped-curves} to generalise to arbitrary fields of coefficients. 
	The deep result about the invariant \(\BNr(T';\F_2)\) that is contained in the hypothesis of the lemma is the fact that for \(c=2\), the slope \(\sigma'\) is even. 
	Computations suggest these slopes to be always even, see also \cite[Questions~3.4]{KWZ_strong_inversions}. 

The proof of \cref{lem:no-trapped-curves} can be easily adapted to show the following general fact about rational tangle replacements in tangle sums:

\begin{proposition}\label{prop:tangle_sum_replacement}
	Let \(T\) be a cap-trivial tangle with connectivity \(\No\). Then there exists a unique integer \(\sigma\in2\mathbb{Z}\) with the following property: 
	For any oriented tangle \(R\) with connectivity \(\Ni\),
	\[
	\BNr_a(T+R;\F_2)
	=
	\BNr_a(Q_{\sigma}+R;\F_2)
	\]
	as bigraded curves, where \(R\) is oriented compatibly with the tangle sum \(T+R\) and the orientation of \(Q_{\sigma}+R\) is compatible with the orientation of \(R\). 
	\qed
\end{proposition}

As with \cref{cor:main:Ni}, we expect \cref{prop:tangle_sum_replacement} to hold over arbitrary coefficients. 

\begin{remark}
	\cref{prop:tangle_sum_replacement} implies that for \(R=Q_{\nicefrac{p}{q}}\) with \(p\) odd and \(q\) even, 
	\[
	\BNr_a(T+R;\F_2)
	=
	\BNr_a(Q_{\sigma+\nicefrac{p}{q}};\F_2)
	=
	\arc{\sigma+\nicefrac{p}{q}}.
	\]
	This explains the phenomenon that the authors already exploited in \cite{postcard} for showing that the Rasmussen invariant \(s_3\) is distinct from other \(s_c\): The curves \(\BNr_a(T_{8_{19}}+Q_{\nicefrac{-1}{2}};\F_c)\) are equal to \(\arc{\sigma_c-\nicefrac{1}{2}}\) and, interestingly, \(\sigma_3\) is different from \(\sigma_c\) for other coefficients \(c\).  
\end{remark}

  The integer invariant \(\theta_c(K)\) from \cref{def:theta} admits the following geometric interpretation. 

\begin{corollary}\label{cor:theta-from-slopes}
	Given a knot \(K\subset S^3\), let \(\sigma_c\) be the slope of \(\BNr_a(T_K;\F_c)\) near the bottom right tangle end. 
	Then
	\(\theta_c(K)=\lceil\nicefrac{\sigma_c}{2}\rceil\). 
	In particular, \(\theta_2(K)=\nicefrac{\sigma_2}{2}\). 
\end{corollary}

\begin{proof}
	By definition, \(\Wh_t(K)=T_K(2t+\nicefrac{1}{2})=T\cup T'\) with \(T=Q_{-2t-\nicefrac{1}{2}}\) and \(T'=T_K\). 
	By \cref{lem:slopes-for-cap-trivial}, the slope \(\sigma'=\sigma_c\) of \(\BNr_a(T')\) is an integer. 
	So by \cref{prop:main:Ni}, for any fixed \(n=t\), we obtain
	\[
	s_c(\Wh_t(K))
	=
	s_c(Q_{-2t-\nicefrac{1}{2}}(-2t))
	+
	\begin{cases*}
	2 &
	if \(2t+\nicefrac{1}{2}<\sigma_c\)
	\\
	0 & otherwise
	\end{cases*}
	\]
	Observe that \(Q_{-2t-\nicefrac{1}{2}}(-2t)=Q_{-\nicefrac{1}{2}}(0)\) is the unknot, so the first summand above vanishes. 
	Thus, for	\(t=\lceil\nicefrac{\sigma_c}{2}\rceil\), 
	\(s_c(\Wh_t(K))=0\) and \(s_c(\Wh_{t-1}(K))=2\). 
	So \(t=\lceil\nicefrac{\sigma_c}{2}\rceil\) agrees with \(\theta_2(K)\). 
\end{proof}

We can finally prove \cref{thm:main}, which we restate here for clarity:

\maintm*

%\begin{theorem}\label{thm:main:Ni}
%	For any knot \(K\) and pattern \(P\) with wrapping number 2 and winding number 0, 
%	\[
%	s_2(P(K))=s_2(P(U,-\theta_2(K))).
%	\] 
%\end{theorem}

\begin{proof}
	The condition on the winding number means that we can write \(P(K)=T_P\cup T_K\) for some pattern tangle \(T_P\) with connectivity \(\conn{T_P}=\Ni\). 
	We now combine \cref{cor:main:Ni,cor:theta-from-slopes}. 
\end{proof}

\subsection{\texorpdfstring{Patterns with winding number \(\bm{\pm2}\)}{Patterns with winding number ±2}}

The following two results play the same role for winding number \(\pm2\) patterns as \cref{prop:main:Ni} does for winding number 0 patterns. 

\begin{proposition}\label{prop:theta-rationals}
	Let \(T\) be a Conway tangle with \(\conn{T}=\ConnectivityX\) and let \(T'\) be a cap-trivial tangle with \(\lk(T')=0\). 
	Suppose \(\BNr_a(T';\F_c)\) is equal to \(\arc{2n}=\BNr(Q_{2n};\F_c)\) for some integer \(n\) up to some grading shift. 
	Then 
	\[
	s_c(T\cup T')
	=
	s_c(T(-2n))+6n.
	\]
\end{proposition}

\begin{proposition}\label{prop:theta-non-rationals}	
	Let \(T\) be a Conway tangle with \(\conn{T}=\ConnectivityX\) and let \(T'\) be a cap-trivial tangle. 
	Let \(\g=\BNr_a(-T;\F_c)\) and \(\gp=\BNr_a(T';\F_c)\).  
	Suppose the slope of \(\gp\) near the lower left tangle end \(p\) is equal to \(\infty\). 
	Then for all coefficients \(c\) and \(n\in\Z\),
	\[
	s_c(T\cup T')
	=
	s_c(T(-2n))
	+
	s_c(T'(2n+1))
	-
	\begin{cases*}
	2 &
	if \([\g,\arc{2n},\gp]\) is \(p\)-ordered
	\\
	0 & otherwise
	\end{cases*}
	\]
\end{proposition}

\begin{proof}[Proof of \cref{prop:theta-rationals}]
	Let us equip \(T'\) and \(Q_{2n}\) with an orientation compatible with the \(1\)-rational closure. 
	The key ingredient in this proof is that we can work out the absolute grading shift between $\gp\coloneqq\BNr_a(T';\F_c)$ and $\arc{2n}$. 
	For this, let $\tilde{T}'$ be the tangle obtained from $T'$ by reversing the orientation of one of the two strands. By assumption, \(\lk(T')=0\), so \(\BNr(T';\F_c)=\BNr(\tilde{T}';\F_c)\) as bigraded multicurves by \cite[Proposition~4.8]{KWZ}. 
	By the same principle, if \(\tilde{Q}_{2n}\) denotes the oriented tangle obtained from $Q_{2n}$ by reversing the orientation of one of the two strands, then
	\[
	\BNr(\tilde{Q}_{2n};\F_c)=q^{-6\lk(Q_{2n})}\cdot\BNr(Q_{2n};\F_c)=q^{6n}\cdot\BNr(Q_{2n};\F_c).
	\]
	Together with \(\tilde{T}'(\infty)=U\), these considerations imply that
	\begin{align*}
	s(\BNr(\Ni;\F_c),\gp)
	&=
	s_c(\tilde{T}'(\infty))
	=
	s_c(U)-1
	=
	s_c(\tilde{Q}_{2n}(\infty))
	\\
	&=
	s(\BNr(\Ni;\F_c),\BNr(\tilde{Q}_{2n};\F_c))
	=
	s(\BNr(\Ni;\F_c),q^{6n}\cdot\BNr(Q_{2n};\F_c)).
	\end{align*}
	Hence \(\gp=q^{6n}\cdot\BNr(Q_{2n};\F_c)\) and
	\begin{align*}
	s_c(T\cup T')
	&=
	s(\BNr_a(-T;\F_c),\gp)+1
	=
	s(\BNr_a(-T;\F_c),q^{6n}\cdot\BNr(Q_{2n};\F_c))+1
	\\
	&
	=
	s_c(T\cup Q_{2n})+6n
	=
	s_c(T(-2n))+6n.
	\qedhere
	\end{align*}
\end{proof}

\begin{proof}[Proof of \cref{prop:theta-non-rationals}]
	Let \(\arcr{2n+1}=\BNr(Q_{2n+1};\F_c)\) and \(\arcb{2n}=\BNr(Q_{2n};\F_c)\), where \(Q_{2n+1}\) is oriented compatibly with the \(\infty\)-rational filling and \(Q_{2n}\) is oriented compatibly with the \(1\)-rational filling. 
	Then by the Four Curves Lemma, 
	\begin{align*}
	& s(\g,\gp)+s(\arcr{2n+1},\arcb{2n})
	-2\cdot i_p(\g,\gp,\arcr{2n+1},\arcb{2n})
	\\
	=~ &
	s(\g,\arcb{2n})+s(\arcr{2n+1},\gp)
	-2\cdot i_p(\g,\arcb{2n},\arcr{2n+1},\gp).
	\end{align*}
	By \cref{thm:s2s},
	\begin{align*}
	s(\arcr{2n+1},\arcb{2n})
	&=
	s_c(Q_{2n}(2n+1))-1=s_c(U)-1=-1,
	\\
	s(\g,\gp)
	&=
	s_c(T\cup T')-1,
	\\
	s(\g,\arcb{2n})
	&=
	s_c(T\cup Q_{2n})-1=s_c(T(-2n))-1, 
	\text{ and}
	\\
	s(\arcr{2n+1},\gp)
	&=
	s_c(Q_{-2n-1}\cup T')-1=s_c(T'(2n+1))-1.
	\end{align*}
	Moreover, since the slope of \(\gp\) near \(p\) is \(\infty\), \(i_p(\g,\gp,\arcr{2n+1},\arcb{2n})=0\) for all \(n\). 
	We obtain
	\[
	s_c(T\cup T')
	=
	s_c(T(-2n))
	+
	s_c(T'(2n+1))
	-2\cdot i_p(\g,\arcb{2n},\arcr{2n+1},\gp)
	\]
	which is equivalent to the statement of the proposition.
\end{proof}

Next, we explore the relationship between \(\theta_c\)-rationality (see \cref{def:theta-rational}) and the multicurve invariants, which will explain our choice of terminology.

\begin{lemma}\label{lem:theta-rational}
	Let 
	\(K\) be a knot and 
	\(\sigma\) the slope of \(\BNr_a(T_K;\F_c)\) near the lower left tangle end \(p\). 
	Then \(\sigma=\infty\) if and only if \(\theta_c(P^K)=\infty\). 
	Moreover, if \(\sigma<\infty\) then  \(\BNr_a(T_K;\F_c)\) is equal to \(\BNr(Q_{2\theta_c(K)};\F_c)\) up to some grading shift and \(\theta_c(P^K)=\theta_c(K)=\nicefrac{\sigma}{2}\). 
\end{lemma}

\begin{proof}
	Let us choose \(T_{P^K}=T_K+Q_{-1}\) as the pattern tangle for the pattern \(P^K\). 
	Then  by \cref{thm:Kh:Twisting:arcs}, the slope of \(\BNr_a(T_{P^K};\F_c)\) near \(p\) is equal to \(\sigma-1\) if \(\sigma<\infty\) and \(\infty\) if \(\sigma=\infty\). 
	So the first statement follows from \cref{prop:eta-slopes}. 
	
	Now suppose \(\sigma<\infty\). 
	\Cref{lem:slopes-for-cap-trivial} implies that \(\sigma\) is an even integer and that \(\BNr_a(T_K;\F_c)\) is equal to \(\BNr(Q_{\sigma};\F_c)\) up to some grading shift.
	So by \cref{cor:theta-from-slopes}, \(\theta_c(K)=\nicefrac{\sigma}{2}\). 
	The identity \(\theta_c(P^K)=\theta_c(K)\) now follows from \cref{prop:eta-slopes} and our observation from the beginning of the proof that the slope of \(\BNr_a(T_{P^K};\F_c)\) near \(p\) is equal to \(\sigma-1=2\theta_c(K)-1\). 
\end{proof}

\begin{proposition}\label{prop:theta-rational}
	A knot \(K\) is \(\theta_c\)-rational if and only if \(\BNr_a(T_K;\F_c)\) is equal to \(\BNr(Q_{\nicefrac{p}{q}};\F_c)\) for some slope \(\nicefrac{p}{q}\in\QPI\) up to some grading shift. 
	Moreover, if \(K\) is \(\theta_c\)-rational, then \(\nicefrac{p}{q}=2\theta_c(K)=2\thetap_c(K)\).
\end{proposition}

\begin{proof}
	By \cref{lem:cables-as-patterns}, \(\thetap_c(K)=\theta_c(P^{K^r})=\theta_c(P^K)\), so \(K\) is \(\theta_c\)-rational if and only if \(\theta_c(P^K)<\infty\). 
	By the first half of \cref{lem:theta-rational}, this is equivalent to the slope \(\sigma\) of \(\BNr_a(T_K;\F_c)\) near the lower left tangle end being finite. 

	So to prove the first statement of the proposition, it remains to see that \(\sigma<\infty\) if and only if \(\BNr_a(T_K;\F_c)\) is equal to \(\BNr(Q_{\nicefrac{p}{q}};\F_c)\) for some slope \(\nicefrac{p}{q}\in\QPI\) up to some grading shift. 
	The if-direction is easy to see, since the connectivity of \(T_K\) is \(\No\) and the curve \(\BNr(Q_{\nicefrac{p}{q}};\F_c)\) has constant slope. 
	The converse follows from the second half of \cref{lem:theta-rational}. 
	
	We have just seen that if \(K\) is \(\theta_c\)-rational, the slope \(\sigma\) is finite. 
	So, by the second half of \cref{lem:theta-rational}, \(\sigma=2\theta_c(K)=2\theta_c(P^K)\).  We now conclude with the observation that \(\thetap_c(K)=\theta_c(P^K)\) from the beginning of the proof. 
\end{proof}

%\begin{corollary}\label{cor:theta-rational-test:cables}
%	A knot \(K\) is not \(\theta_c\)-rational if and only if for all \(n,m\in\Z\), 
%	\[
%	s_c(C_{2,2n+1}(K))=s_c(C_{2,2m+1}(K))+2(n-m).
%	\]
%\end{corollary}

The above gives rise to the following simple test to determine if \(K\) is \(\theta_c\)-rational without computing the entire curve \(\BNr_a(T_K)\).

\begin{corollary}\label{cor:theta-rational-test:practical}
	A knot \(K\) is \(\theta_c\)-rational if and only if 
	\(
	s_c(P^K_{\theta_c(K)-1}(U))
	\!=
	s_c(P^K_{\theta_c(K)}(U))
	\).
%	\(
%	s_c(P^K_{\theta_c(K)-1}(U))
%	\) 
%	and 
%	\(
%	s_c(P^K_{\theta_c(K)}(U))
%	\)
%	are equal. 
\end{corollary}
\begin{proof}
	If \(K\) is not \(\theta_c\)-rational, then the identity clearly does not hold. 
	If \(K\) is \(\theta_c\)-rational, then \(\theta_c(K)=\thetap_c(K)\) by \cref{prop:theta-rational}, which implies the desired identity.
\end{proof}

We can now prove \cref{thm:theta-rationals}, which we restate here for convenience.

\thetarationals*

\begin{proof}
	The first statement was already shown in \cref{prop:theta-rational}. 
	The formula for \(s_c(P(K))\) is obtained by applying \cref{prop:theta-rationals} to \(T=T_P\) some pattern tangle for \(P\), \(T'=T_K\), and \(n=\theta_c(K)\) and using the identity \(P_{-\theta_c(K)}(U)=T(-2\theta_c(K))\) (\cref{lem:basic:satellites}\ref{lem:basic:satellites:pattern_of_unknot}). 
\end{proof}

\begin{corollary}
	Let \(K\) be a \(\theta_c\)-rational knot. Then for all \(t\in\Z\),
	\[
	s_c(C_{2,2t+1}(K))
	=
	4\theta_c(K)+2t+
	\begin{cases*}
	0
	&
	if \(\theta_c(K)\leq t\)
	\\
	2
	&
	 otherwise.
	\end{cases*}
	\]
\end{corollary}

\begin{proof}
	This follows from \cref{thm:theta-rationals} applied to \(P=C_{2,2t+1}\), noting that 
	\[
	P_{-\theta_c(K)}(U)=C_{2,2t+1-2\theta_c(K)}(U)=T_{2,2(t-\theta_c(K))+1}
	\]
	by definition and
	\[
	s_c(P_{-\theta_c(K)}(U))
	=
	2(t-\theta_c(K))+1+
	\begin{cases*}
	-1 
	&
	if \(2(t-\theta_c(K))+1>0\)
	\\
	+1
	&
	if \(2(t-\theta_c(K))+1<0\)
	\end{cases*}
	\]
	by \cref{ex:Rasmussen_of_two_n_torus_knots}.
\end{proof}

\thetanonrationals*

\begin{proof}
	Let \(P(K)=T_P\cup T_K\) for some pattern tangle \(T_P\). 
	We apply \cref{prop:theta-non-rationals} to \(T=T_P\), \(T'=T_K\), and \(n=0\). 
	We obtain
	\[
	s_c(P(K))
	=
	s_c(T_P(0))
	+
	s_c(T_K(1))
	-
	\begin{cases*}
	2 &
	if \([\g,\arc{0},\gp]\) is \(p\)-ordered
	\\
	0 & otherwise
	\end{cases*}
	\]
	where \(\g=\BNr_a(-T_P;\F_c)\) and \(\gp=\BNr_a(T_K;\F_c)\) and \(p\) is the lower left puncture. 
	Since 
	\(
	T_P(0)=P(U)
	\)
	and
	\(
	T_K(1)=C_{2,1}(K)
	\),
	it suffices to show that \([\g,\arc{0},\gp]\) is \(p\)-ordered if and only if \(\theta_c(P)>0\). 
	To see this, let \(\sigma\) be the slope of \(\BNr_a(T_P;\F_c)=-\BNr_a(-T_P;\F_c)=-\g\) near \(p\). 
	So by \cref{prop:eta-slopes}, \(\theta_c(P)>0\) if and only if \(\sigma>0\) or (\(\sigma=0\) and \(-\g>_p\arc{0}\)). 
	This is indeed equivalent to \([\g,\arc{0},\gp]\) being \(p\)-ordered, since the slope of \(\gp\) at \(p\) is \(\infty\).
\end{proof}

\begin{theorem}\label{thm:theta-non-rationals:infty}
	Let \(K\) be a knot that is not \(\theta_c\)-rational. 
	Let \(P\) be a pattern with wrapping number 2 and winding number $\pm 2$ such that \(\theta_c(P)=\infty\). 
	Then
	\[
	s_c(P(K))
	=
	s_c(P(U))+s_c(C_{2,1}(K))
	-
	\begin{cases*}
	2 &
	if \(\gamma_K<_p \gamma_P\)
	\\
	0 & otherwise
	\end{cases*}
	\] 
	where \(p\) is the lower left tangle end, 
	\(\gamma_K\coloneqq \BNr_a(T_K;\F_c)\), and 
	\(\gamma_P\coloneqq \BNr_a(-T_P;\F_c)\). 
\end{theorem}

\begin{proof}
	Similar to \cref{thm:theta-non-rationals:non-infty}, this follows from \cref{prop:eta-slopes,prop:theta-non-rationals}.
\end{proof}

\begin{example}\label{ex:theta-non-rationals:infty:after_thm}
	Let \(K\) be the positive trefoil knot and the patterns \(P^+\) and \(P^-\) defined by the pattern tangles \(T_{P^+}=T'_e=T_{-K}+Q_{+1}\) and \(T_{P^-}=T'_d=T_{-K}+Q_{-1}\), respectively.
	Then the knots \(P^\pm(K)=C_{2,\mp1}(K\# -\!K)\) are obtained by plugging \(T_{P^\pm}\) into the grey disc in \cref{fig:frontispiece:1}.
	The knot \(K\#-\!K\) is concordant to the unknot \(U\), so \(P^\pm(K)\) is concordant to  \(C_{2,\mp1}(U)=U\), and hence \(s_2(P^{\pm}(K))=0\).
%	Furthermore, since \(K\) is not \(\theta_2\)-rational, the function \(t\mapsto s_2(C_{2,1-2t}(K))\) is affine of slope~\(-2\).
%	Also,	for any integer~\(t\), \(P^\pm_t(U)=C_{2,2t+c_\pm}(-K)\) for some odd integers \(c_\pm\). These constants \(c_\pm\) only differ by \(2\), since \(P^\pm\) only differ by a single full twist.
%	Therefore, \(t\mapsto s_2(P_t^\pm(U))\) are affine functions of slope \(+2\) that differ by a constant~2.
%	So we see that there exists at most one integer \(\theta_\pm\) such that \eqref{eq:theta-non-rationals:infty:intro} holds for the corresponding pattern. 
	Moreover, \(P^\pm(U)=C_{2,\mp1}(-K)=-C_{2,\pm1}(K)\). 
	Since \(K\) is not \(\theta_2\)-rational \(s_2(C_{2,-1}(K))=s_2(C_{2,1}(K))-2\), so we obtain the following identities:
	\begin{align*}
	s_2(P^+(K))
	&=
	s_2(P^+(U))+s_2(C_{2,1}(K)),
	\\
	s_2(P^-(K))
	&=
	s_2(P^-(U))+s_2(C_{2,1}(K))-2.
	\end{align*}
	\Cref{thm:theta-non-rationals:infty} for \(c=2\) gives an alternative proof of these identities. 
	Indeed, with the same notation as in \cref{thm:theta-non-rationals:infty}, 
	\[
	\gamma_{P^{\pm}}
	=
	\BNr_a(-T_{P^{\pm}};\F_2)
	=
	\BNr_a(T_{K}+Q_{\mp1};\F_2),
	\]
	so consequently 
	\(
	\gamma^+<_p\gamma_K<_p\gamma^-
	\).
%	For instance \(K=3_1\) is not \(\theta_2\)-rational---the curve \(\gamma_K\) is shown in \cref{fig:lifts}. 
\end{example}

\section{Computations, observations, and varying characteristics}
\label{sec:differentfields}
%In this section, we will discuss how $\theta_c(K)$ may be computed
%and which patterns we have observed experimentally from such computations,
%with a particular focus on the differences between $s_c$ and $\theta_c$
%(see \cref{defprop:s} and \cref{subsec:generalize1}, respectively)
%for varying characteristics~$c$.%

The invariant \(\theta_c(K)\) can be read off directly from the curve \(\BNr_a(T_K;\F_c)\) using \cref{cor:theta-from-slopes}. 
Here is an alternative way of computing \(\theta_c(K)\). 

\begin{proposition}\label{prop:thetaalgo}
The following algorithm computes \(\theta_c(K)\) for a given knot \(K\) and a given characteristic~\(c\) that is \(0\) or a prime:\medskip

\begin{tabular}{rp{.8\textwidth}}
\emph{Step 1}  & Calculate \(s_c(\Wh(K))\). \\
\emph{Step 2a} & \rule{0pt}{1.05\baselineskip}If \(s_c(\Wh(K)) = 0\), then calculate \(s_c(\Wh_t(K))\) for \(t = -1, -2, \ldots\) \newline until
\(s_c(\Wh_t(K)) = 2\). Return \(\theta_c(K) = t + 1\). \\
\emph{Step 2b} & \rule{0pt}{1.05\baselineskip}If \(s_c(\Wh(K)) \neq 0\), then calculate \(s_c(\Wh_t(K))\) for \(t = 1, 2, \ldots\) \newline until \(s_c(\Wh_t(K)) = 0\). Return \(\theta_c(K) = t\).
\end{tabular}
\end{proposition}

\begin{proof}
That this algorithm always stops follows from \cref{prop:theta_nu}.
That it returns the correct result is immediate from the definition of \(\theta_c(K)\), see \cref{def:theta}.
\end{proof}

Note that \cref{prop:thetaalgo} builds on the well-known computability of $s_c$.
In practice, there exist a number of computer programs that calculate the Rasmussen invariants \(s_c(K)\) for a given knot \(K\) and characteristic \(c\), such as \texttt{khoca} \cite{khoca}, \texttt{SKnotJob} \cite{SKnotJob} and \texttt{kht++}~\cite{khtpp}.
With the aid of these programs, the above algorithms allow us to compute \(\theta_c(K)\) for sufficiently small knots \(K\).
For example, we have computed \(\theta_c(K)\) for various \(c\) for all prime knots up to 8~crossings, roughly half of all prime knots up to 10~crossings, and various individual knots with up to 16~crossings.
For these computations, we use the na\"ive way to construct a diagram $D'$ of $W^+_t(K)$ from a diagram $D$ of $K$, which entails that $D'$ has twice the width of $D$ (in the sense of \cite{knotinfo}), and at least four times as many crossings as $D$. So computer calculations for more complicated knots, such as iterated Whitehead doubles, seem in general unfeasible with current technology. %usually exceed currently available computational resources.
For a complete list of the values of \(\theta_c(K)\) that we have computed,
we refer the reader to an online table \cite{thetatable}, which we will continue to update.

We now make a number of observations and conjectures based on our calculations.

\subsection{Linear independence}

We find that \(\theta_2(K)\) often differs from \(\theta_c(K)\) for \(c\neq 2\), even for many small knots. For example we have
\begin{equation}\label{eq:theta_for_trefoil}
\theta_2(T_{2,3}) = 4, \qquad \theta_c(T_{2,3}) = 3
\end{equation}
for all \(c\in \{0, 3, 5, \ldots, 97\}\).
Note that one direct consequence  is that
\[
s_2(\Wh_3(T_{2,3})) = 2, \qquad s_0(\Wh_3(T_{2,3})) = 0.
\]
This provides a knot with \(s_2 \neq s_0\)
(different from the first such knot, which was found by Seed \cite[Remark~6.1]{LipshitzSarkarRefinementS}),
similar to the example of a knot with \(\tau\neq s_0\) given in \cref{eq:heddenording}.
Let us consider another computational result:
\[
\theta_3(T_{3,4}) = 9, \qquad \theta_c(T_{3,4}) = 8,
\]
for all \(c\in \{0, 2, 5, \ldots, 97\}\).
%\comment{Please double-check if $c=2$ should be included. If so, this is still wrong in \cite{thetatable}. }
This lead to the first known example \(K = \Wh_8(T_{3,4})\)
of a knot \(K\) with \(s_3(K) \neq s_0(K)\), and thus to the following theorem.
\begin{theorem}[\cite{postcard}]\label{thm:postcard}
The concordance homomorphisms \(s_0, s_2, s_3\) are linearly independent.
\end{theorem}
This result provided the authors with the initial impetus for the paper at hand.
Can \cref{thm:postcard} be extended to an independence result for all the $s_c$?
As shown by Schütz \cite{schuetz}, for every knot $K$,
\begin{equation}\label{eq:p0}
s_p(K) = s_0(K) \qquad \text{holds for all but finitely many primes $p$}.
\end{equation}
One may conjecture that this is the only restriction on the values of the $s_c$,
i.e.~that the homomorphism \((\tfrac{s_p - s_0}2)_p\colon \mathcal{C} \to \bigoplus_p \mathbb{Z} \), where $p$ ranges over all primes, is surjective \cite[Question~6.2]{LipshitzSarkarRefinementS};
maybe even its restriction to the subgroup of $\mathcal{C}$ spanned by \(\{T_{p,p+1} \mid p\text{ prime}\}\) is surjective (compare \cite[Conjecture~6.2]{schuetz}).
Unfortunately, the computation of $\theta_c(T_{5,6})$ seems to be just out of reach of the software and hardware currently available to us. Let us postpone further discussion of torus knots to \cref{subsec:torus}.

We can show the following independence results that include $\tau$ and the invariants \(\theta_c\).
\begin{proposition}\label{thm:indep}
        \begin{enumerate}[(a)]
        \item The image of the homomorphism \[(\tau, \tfrac{s_0}{2}, \tfrac{s_2}{2}, \tfrac{s_3}{2}, \theta_2)\] from the smooth concordance group \(\mathcal{C}\)
        to \(\Z^5\) contains \(\Z^4 \times4\Z\).
        \item Under the assumption that \(\theta_0\) and \(\theta_3\) are homomorphisms \(\mathcal{C}\to\Z\),
        the homomorphism \[(\tau, \tfrac{s_0}{2}, \tfrac{s_2}{2}, \tfrac{s_3}{2}, \theta_0, \theta_2, \theta_3)\] from \(\mathcal{C}\) to \(\Z^7\)
        contains \(\Z^5 \times 4\Z \times \Z\) in its image.
        \end{enumerate}
\end{proposition}
\begin{proof}
One checks that the seven row vectors of \cref{table:indep}
span a subgroup of \(\Z^7\) containing \(\Z^5\times 4\Z\times \Z\).
This shows part (b), and thus also part (a), of the statement.
\end{proof}
\begin{table}[t]
\centering
\begin{tabular}{r@{ }l|*{7}{>{\centering\arraybackslash}p{8mm}}}
\multicolumn{2}{c|}{\hspace{3em}knot} & $\tau$ & $\nicefrac{s_0}2$ & $\nicefrac{s_2}2$ & $\nicefrac{s_3}2$ & $\theta_0$ & $\theta_2$ & $\theta_3$ \\[.3ex]\hline
$T_{2,3}$  & = $3_1$\rule{0pt}{2.3ex}
& 1     & 1     & 1     & 1     & 3             & 4             & 3             \\[.3ex]
$P(3,3,-2) = T_{3,4}$ & = $8_{19}$
& 3     & 3     & 3     & 3     & 8             & 8             & 9             \\[.3ex]
$P(-5,3,-2)$ & = $10_{125}$
& 1     & 1     & 1     & 1     & 2             & 4             & 2             \\[.3ex]
$P(5,5,-4)$ & = $14n24552$
& 5     & 5     & 5     & 5     & 15            & 16            & 15            \\[.3ex]\cline{7-9}
$\Wh_2(T_{2,3})$ & = $15n115646$ 
& 0     & 1     & 1     & 1     & \multicolumn{3}{|c}{}              \\[.3ex]
$\Wh_3(T_{2,3})$ & = $14n22180$
& 0     & 0     & 1     & 0     & \multicolumn{3}{|c}{\textcolor{gray}{\footnotesize unknown}} \\[.3ex]
$\Wh_8(T_{3,4})$ & \quad \textcolor{gray}{\footnotesize (34 crossings)}
& 0     & 0     & 0     & 1     & \multicolumn{3}{|c}{}               \\[.3ex]
\end{tabular}
\medskip
\caption{A table of $\tau, s_c$ and $\theta_c$ for $c\in\{0,2,3\}$ for various knots.
The values of $\tau$ may be computed using Hedden's result that $\theta_\tau = 2\tau$
and the computer program \cite{hfkcalc}; the values of $s_c$ may be computed using any
of the programs \cite{khoca,SKnotJob,khtpp}; and the values of $\theta_c$
may be computed as described in \cref{cor:theta-from-slopes} and \cref{prop:thetaalgo}.}
\label{table:indep}
\end{table}

As a direct consequence of \cref{eq:p0}, we have the following.
\begin{corollary}\label{cor:theta_identical_for_almost_all_fields}
    For each knot \(K\), \(\theta_0(K) = \theta_c(K)\) holds for all but finitely many primes \(c\).\qed
\end{corollary}

One may now ask whether the conjectured divisibility of $\theta_2$ by 4 (see \Cref{conj:divby4}), \eqref{eq:p0}, and \cref{cor:theta_identical_for_almost_all_fields} are the only restrictions on the values of \(s_c\) and \(\theta_c\).
\begin{question}
Let the infinite vectors \(\underline{s}, \underline{\theta} \in \Z^{\infty}\) be given such that for each of the two vectors, all but finitely many entries equal the first entry.
Assume moreover that the second entry of \(\underline{\theta}\) is divisible by four.
Does there exist a knot \(K\) such that
\begin{align*}
(s_0(K), s_2(K), s_3(K), s_5(K), \ldots) &= 2\underline{s}\qquad\text{and} \\
(\theta_0(K), \theta_2(K), \theta_3(K), \theta_5(K), \ldots) &= \underline{\theta}?
\end{align*}
\end{question}
Let us note that all values of \(\theta_2\) that we could compute are divisible by 4, in accordance with \cref{conj:divby4}. In contrast, for $c\in \{0,3,5,\ldots, 97\}$, and conjecturally for all $c\neq 2$, there does not exist $n\geq 2$ such that $n$ divides $\theta_c(K)$ for all $K$.

\begin{table}[t]
	\centering
	\begin{tabular}{r|*{2}{>{\centering\arraybackslash}p{28mm}}cc}
		knot                            & \(s_c\) for \(c\leq 7\) & \(\theta_0 = \theta_5 = \theta_7\)  & \(\theta_2\)  & \(\theta_3\) \\\hline
		\(10_{125}\)                      & $-2$                     & $\mathbf{-2}$                                         & $-4$               &    $\mathbf{-2}$ \\
		\(\mathit{8_{19}}, 10_{154}, 10_{161}\)    & 6                     & \textbf{8}                                         & \textbf{8}               &    9 \\
		\(\mathit{10_{124}},  \mathit{10_{139}}, \mathit{10_{152}}\) & 8                     & \textbf{11}                                        & \textbf{12}              &    12 \\
	\end{tabular}
	\medskip
	\caption{The seven prime knots with crossing number 10 or less that do not satisfy \cref{eq:2pattern} or \cref{eq:cpattern}. The offending values are printed in boldface. Positive braid knots are printed in italics.}
	\label{table:specialknots}
\end{table}

\subsection{Alternating knots}
%Next, let us discuss our observations for alternating knots.
Recall that for all alternating knots~\(K\) (in fact, even for all quasi-alternating~\(K\)), \(s_c(K) = -\sigma(K)\) holds for all characteristics~\(c\) \cite[Theorem~3]{RasmussenSlice}.
\begin{conjecture}\label{conj:alt}
For all alternating knots \(K\), we have\smallskip
\begin{align} 
\label{eq:2pattern}\theta_2(K) & = 2s_2(K), \\[.5ex]
\label{eq:cpattern}\theta_c(K) & = \tfrac32 s_c(K)\quad \text{for all \(c\neq 2\).}
\end{align}
\end{conjecture}

\cref{conj:alt} is backed up by computations for around 70 prime alternating knots for \({c \in \{0,2,3,5,7\}}\).
Interestingly, there exist quasi-alternating knots, such as the \((5,-3,2)\)-pretzel
knot (\(10_{125}\) in the tables) that do not satisfy \cref{eq:cpattern}, see \cref{table:specialknots}.

\begin{wraptable}[11]{r}{0.52\textwidth}
        \vspace{2mm}
	\centering
	\begin{tabular}{r|ccc}
		knot $J = \Wh_t(K)$                 & \(\tau(J)\) & \(s_0(J)\) & \(s_2(J)\) \\\hline
		\rule{0pt}{2.5ex}\(t = -\sigma(K)-1\)         & 1        & 2       & 2 \\
		\rule{0pt}{2.5ex}\(t = -\tfrac32\sigma(K)-1\) & 0        & 2       & 2 \\
		\rule{0pt}{2.5ex}\(t = -2\sigma(K)-1\)        & 0        & 0       & 2 \\
	\end{tabular}
	\medskip
	\caption{An alternating knot $K$ with $\sigma(K)< 0$ and satisfying \cref{conj:alt} has three twisted positive Whitehead doubles that are linearly independent in~$\mathcal{C}$.}
	\label{table:taus0s2}
\end{wraptable}
One notes the similarity of \eqref{eq:2pattern} and \eqref{eq:cpattern} to the equation
\cref{eq:hedden}, \(\theta_\tau(K) = 2\tau(K)\), proven by Hedden for all knots \(K\). For alternating (even quasi-alternating) \(K\),
\cref{eq:hedden} implies \(\theta_\tau(K) = -\sigma(K)\).
So, if \cref{conj:alt} holds, then every alternating knot \(K\) with \(\sigma(K)\neq 0\) has three twisted Whitehead doubles
generating a $\mathbb{Z}^3$ summand of the smooth concordance group, restricted to which
\(\tau, s_0, s_2\) are linearly independent, as \cref{table:taus0s2} shows in the case $\sigma(K)<0$; the argument for $\sigma(K)>0$ is similar.

Equation \cref{eq:cpattern} for $c = 0$ was also mentioned in \cite{HeddenOrding} as a reasonable guess for two-stranded torus knots (cf.~also \cite{MR3577888}).
It does not hold for all knots, however, see \cref{table:specialknots}.
This table demonstrates that the behaviour of $\theta_c$ on non-alternating positive braid knots is rather different than its behaviour on alternating knots. 
We make the following conjecture, which we have checked on a total of 14 prime positive braid knots and $c\in\{2,3,5\}$.
\begin{conjecture}\label{conj:posbraids}
If \(K\) is a non-alternating knot that is the closure of a positive braid, then $\theta_2(K) < 2s_2(K)$ and $\theta_0(K) < \tfrac32 s_0(K)$.
In particular, \eqref{eq:2pattern} and \eqref{eq:cpattern} do not hold.
\end{conjecture}

Note that the only knots that are both alternating and closure of a positive braid are the $T_{2,2n+1}$ torus knots with $n\geq 0$ \cite{zbMATH01773643,zbMATH06479738}.
Beyond \cref{conj:posbraids}, no clear pattern seems to emerge from our computations
for the values of $\theta_c$ on positive braid knots.
%For example, the two knots 
%$14n21324$ and $14n21881$ are the closure of the seemingly similar positive braids
%$\sigma_1^7\sigma_2\sigma_1^5\sigma_2$ and $\sigma_1^5 \sigma_2 \sigma_1^3 \sigma_2^2 \sigma_1^2 \sigma_2$

\begin{wraptable}[12]{r}{0.5\textwidth}
	\vspace{1ex}
	\centering
	\begin{tabular}{r|*{5}{>{\centering\arraybackslash}p{7mm}}}
		knot         & \(s_c\)        & \(\theta_2\)    & \(\theta_3\)  & \(\theta_5\) & \(\theta_0\) \\\hline
		$T_{2,2n+1}$ &  $2n $         &    $4n$         &  $3n$         &  $3n$        &  $3n$        \\
		\scriptsize{for} $\scriptstyle 0\leq n\leq 4$ & \\\hline
		$T_{3,4}$    &  6             &    8            &  9            &  8           &  8           \\
		$T_{3,5}$    &  8             &    12           &  12           &  11          &  11          \\
		$T_{3,7}$    &  12            &    16           &  18           &  16          &              \\
		$T_{3,8}$    &  14            &    20           &  21           &  19          &              \\
		$T_{3,10}$   &  18            &    24           &  27           &              &              \\
		$T_{3,11}$   &  20            &    28           &               &              &              \\\hline
		$T_{4,5}$    &  12            &    16           &  15           &  15          & 
	\end{tabular}
	\medskip
	\caption{Computed values for $\theta_c(T_{p,q})$}
	\label{table:torusknots}
\end{wraptable}

\subsection{Torus knots}\label{subsec:torus}
Positive torus knots \(T_{p,q}\) are arguably the simplest positive braid knots.
%So let us now come back to torus knots, which are arguably the simplest positive braid knots.
\cref{table:torusknots} contains the values of $\theta_c(T_{p,q})$ that we could compute.
Those values are compatible with the following conjectures.

\begin{conjecture}\label{conj:thetactpq}
For all characteristics $c$, there exists a function \(r_c\colon \Z_{\geq 1} \to \Z\)
such that
\(
\theta_c(T_{p + q, q}) =  \theta_c(T_{p, q}) + r_c(q)
\)
for all coprime positive integers \(p\) and \(q\).
\end{conjecture}

For $c = 0$ and $c = 2$, we make the following guess for $r_c$.

\begin{conjecture}\label{conj:theta2tpq}
For all coprime positive integers \(p\) and \(q\),
\[
\theta_0(T_{p + q, q})         =  \theta_0(T_{p, q}) + q^2 - 1
\quad
\text{and}
\quad
\theta_2(T_{p + q, q})         =  \theta_2(T_{p, q}) +
	\begin{cases*}
	q^2 & if $q$ even \\
	q^2 -1 & if $q$ odd.
	\end{cases*}
\]
\end{conjecture}

Our conjectures are in accordance with conjectures by Schütz \cite[Conjectures~6.2 and 6.3]{schuetz}
on the values of $s_c$ of certain Whitehead doubles of torus knots.
Note that if \cref{conj:thetactpq} holds, then $\theta_c(T_{p,q})$ is for all $p, q$ determined by the function $r_c$,
using Euclid's algorithm and the fact that $\theta_c(T_{0,1}) = 0$ since $T_{0,1}$ is the unknot.
The knot $T_{p+q,q}$ arises from $T_{p,q}$ by a $q$-homologous right-handed twist on $q$ strands (see \cref{def:companiontwists}).
So it follows from \cref{prop:generaltwists} that, if \cref{conj:thetactpq} holds, then
$0 \leq r_c(q) \leq q^2$ holds for $c \neq 2$ (and for $c = 2$ under the assumption that $s_2$ satisfies the generalised crossing change inequality).
In order to prove \cref{conj:thetactpq} and \cref{conj:theta2tpq}, one would have to gain an understanding that goes beyond the inequalities in \cref{prop:generaltwists} of how $\theta_c$ changes under twists of the companion.

Interestingly, the classical knot signature $\sigma$ satisfies the following recursive formula for torus knots \cite[Theorem~5.2]{zbMATH03736629}:
\[
\sigma(T_{p + 2q, q})         =  \sigma(T_{p, q}) +
	\begin{cases*}
	q^2 & if $q$ even \\
	q^2 -1 & if $q$ odd.
	\end{cases*}
\]
The difference to the formula given for $\theta_2$ in \cref{conj:theta2tpq} is
that the left-hand side pertains to $T_{p + 2q, q}$ instead of $T_{p + q, q}$.

\bibliographystyle{myamsalpha}
\bibliography{main}
\end{document}